\documentclass[12pt]{amsart}
\usepackage{amsmath, amssymb, amsthm, mathrsfs, lscape, graphicx, enumerate, textcomp, colonequals, mathtools} \usepackage{fullpage}
\usepackage[backref]{hyperref}
\usepackage[alphabetic,backrefs,lite]{amsrefs}
\usepackage[usenames,dvipsnames]{color}

\usepackage[font=scriptsize,labelformat=empty]{caption,subcaption}
\usepackage{tikz}
\usepackage{tikz-cd}
\usetikzlibrary{shapes,arrows,snakes,trees,positioning,decorations.markings,cd}

\hypersetup{
  colorlinks   = true, %Colours links instead of ugly boxes
  urlcolor     = blue, %Colour for external hyperlinks
  linkcolor    = blue, %Colour of internal links
  citecolor   = blue %Colour of citations
}

\def\r{\color{red}}\def\b{\color{blue}}\def\g{\color{ForestGreen}}

\makeatletter
\newcommand{\pushright}[1]{\ifmeasuring@#1\else\omit\hfill$\displaystyle#1$\fi\ignorespaces}
\newcommand{\pushleft}[1]{\ifmeasuring@#1\else\omit$\displaystyle#1$\hfill\fi\ignorespaces}
\makeatother

\numberwithin{equation}{section}
\allowdisplaybreaks[1]

\newtheorem{thm}[subsection]{Theorem}

\newtheorem{prop}[subsection]{Proposition}
\newtheorem{corollary}[subsection]{Corollary}

\newtheorem{lemma}[subsection]{Lemma}

\newtheorem*{example}{Example}
\newtheorem{question}[equation]{Question}

\theoremstyle{definition}
\newtheorem{defn}[equation]{Definition}

\newtheorem{eg}[equation]{Example}
\newtheorem{remark}[equation]{Remark}

\newcommand{\R}{\mathbb R}
\newcommand{\Q}{\mathbb Q}
\newcommand{\Z}{\mathbb Z}
\newcommand{\N}{\mathbb N}
\newcommand{\C}{\mathbb C}
\renewcommand{\O}{\mathcal{O}}

\renewcommand{\P}{\mathbb P}

\renewcommand{\phi}{\varphi}

\newcommand{\Spec}{\mathrm{Spec}\ }

\DeclareMathOperator{\Art}{Art}
\DeclareMathOperator{\Gal}{Gal}
\DeclareMathOperator{\Aut}{Aut}
\DeclareMathOperator{\wt}{wt}

\DeclareMathOperator{\cha}{char}
\DeclareMathOperator{\divisor}{div}

\DeclareMathOperator{\lcm}{lcm}
\DeclareMathOperator{\ess}{ess}

\DeclareMathOperator{\divi}{div}
\DeclareMathOperator{\denom}{denom}

\title[Conductors and minimal discriminants]{Conductors and minimal discriminants of hyperelliptic curves: A comparison in the tame case}
\author{Padmavathi Srinivasan}
\address{Current address: School of Mathematics, University of Georgia, 452 Boyd Graduate Studies, 1023 D. W. Brooks Drive, Athens, GA 30602.}
\email{Padmavathi.Srinivasan@uga.edu}
\urladdr{https://padmask.github.io/}

\begin{document}
 
\date{\today}

\begin{abstract}
 Let $C$ be a hyperelliptic curve of genus $g$ over the fraction field $K$ of a discrete valuation ring $R$. Assume that the residue field $k$ of $R$ is perfect and that $\cha k > 2g+1$. Let $S = \Spec R$. Let $X$ be the minimal proper regular model of $C$ over $S$. Let $\Art (C/K)$ denote the Artin conductor of the $S$-scheme $X$ and let $\nu (\Delta_C)$ denote the minimal discriminant of $C$. We prove that $-\Art (C/K) \leq \nu (\Delta_C)$. The key ingredients are a combinatorial refinement of the discriminant introduced in this paper (called the metric tree) and a recent refinement of Abhyankar's inversion formula for studying plane curve singularities. We also prove combinatorial restrictions for $-\Art (C/K) = \nu (\Delta_C)$.
\end{abstract}

\maketitle
%\vspace*{-1cm}
\section{Introduction} The goal of this paper is to prove an inequality between two measures of degeneracy for a family of hyperelliptic curves, namely the Artin conductor and the minimal discriminant. Let $(R,\nu)$ be a discrete valuation ring with perfect residue field $k$ of of odd characteristic. Let $K$ be the fraction field of $R$. Let $C$ be a smooth, projective, geometrically integral curve of genus $g \geq 1$ defined over $K$. Let $S = \Spec R$. Let $X$ be a proper, flat, regular $S$-scheme with generic fiber $C$. The Artin conductor of the model $X$ is given by $\Art (X/S) =  \chi(X_{\overline{K}}) - \chi(X_{\overline{k}}) - \delta $, where $\chi$ is the \'etale Euler-characteristic and $\delta$ is the Swan conductor associated to the $\ell$-adic representation $\Gal (\overline{K}/K) \rightarrow \Aut_{\Q_\ell} (H^1_{\mathrm{et}}(X_{\overline{K}}, \Q_\ell))$ ($\ell \neq \cha k$). The Artin conductor is a measure of degeneracy of the model $X$; it is a non-positive integer that is zero precisely when either $X/S$ is smooth or when $g=1$ and $(X_k)_{\mathrm{red}}$ is smooth. 
Let $\Art (C/K)$ denote the Artin conductor of the minimal proper regular model of $C$ over $S$.

For hyperelliptic curves, there is another measure of degeneracy defined in terms of minimal Weierstrass equations. Assume that $C$ is hyperelliptic. An integral Weierstrass equation for $C$ is an
equation of the form $y^2= f(x)$ with $f(x) \in R[x]$, such that $C$ is birational to the plane curve given by this equation. The discriminant of such an equation is defined to be the
non-negative integer $\nu(\Delta_f)$, where $\Delta_f$ is the discriminant of $f$, thought of as a polynomial of degree $2 \lceil \deg(f)/2 \rceil$. A minimal Weierstrass equation is an equation for which the integer $\nu(\Delta_f)$ is as small as possible amongst all integral equations, and the corresponding integer $\nu(\Delta_C)$ is called the minimal discriminant of $C$. 

When $g=1$, we have $-\Art (C/K) = \nu(\Delta_C)$ by the Ogg-Saito formula \cite[p.156, Corollary 2]{saito2}. When $g=2$, Liu~\cite[p.52, Theoreme~1 and p.53, Theoreme~2]{liup} shows that $-\Art (C/K) \leq \nu(\Delta_C)$. In the author's thesis \cite{conddisc}, Liu's inequality was extended to hyperelliptic curves of \emph{arbitrary genus} assuming that the roots of $f$ are defined over an \emph{unramified} extension of $K$. In this paper, we extend \cite{conddisc} assuming only that $\cha(k) > 2g+1$ (in particular, the roots of $f$ are defined over a \emph{tame} extension of $K$). 

\begin{thm}\label{Tfinalthm}
Let $C$ be a hyperelliptic curve of genus $g \geq 1$ over a discretely valued field $K$
with ring of integers $R$ and perfect residue field $k$ such that $\cha(k) > 2g+1$. Let $\nu(\Delta_C)$ be the minimal discriminant of $C$ and let $\Art(C/K)$ denote the Artin conductor of the minimal regular model of $C$. Then
\[ - \Art(C/K) \leq \nu(\Delta_C) .\]
\end{thm}

\subsection{Combinatorial criterion for equality}
The techniques in this paper enable us to give a purely combinatorial explanation for why the Ogg-Saito conductor-discriminant \emph{equality} in genus $1$ is sometimes only an \emph{inequality} when $g \geq 2$. The reason for bad reduction in hyperelliptic curves in odd residue characteristic is because distinct roots of the polynomial $f$ reduce to the same element in the residue field. Roughly, the difference in the two invariants comes about because the discriminant keeps track of not just the collision of roots, but \emph{how many} roots collide at the same point. However, if we have a large number of roots coming together to order $1$ that can still be separated with a single blowup, then the conductor is still small. The inequality between the conductor and the discriminant in this case boils down to the inequality $2 \leq n(n-1)$ for any integer $n \geq 2$. This analysis is accurate if the roots of $f$ are rational and we have an even number of roots coming together, and every pair comes together to order $1$. 

More generally, even when the roots of $f$ are non-rational, for every closed point $P$ in $\divi(f)$ on $\mathbb{P}^1_R$, one can look at the multiplicity of $f$ in the local ring at $P$ -- this is a positive integer that can be viewed as a weighted sum $\wt_P$ of the roots of $f$ specializing to $P$. (See Definition~\ref{D:weights} and Lemma~\ref{L:wtsmult}). For example, for $P$ as above, we have $\wt_P = 1$ exactly when $f$ does not vanish identically along the special fiber and exactly one irreducible factor of $f$ specializes to $P$, and this factor is either linear or a translate of an Eisenstein polynomial. For equality to hold, it is necessary that all points $P$ in $\divi(f)$ have $\wt_P \leq 3$. More precisely, for every polynomial $f \in R[x]$ (for example, a polynomial $f$ such that $\nu(\Delta_f) = \nu(\Delta_C)$), our techniques produce an explicit proper regular model $X^f$ for the hyperelliptic curve with equation $y^2=f(x)$, which is sometimes the minimal model, and we can show
\begin{thm}\label{Cexbal}
 $-(\Art(X^f)) = \nu(\Delta_f)$ if and only if every $P $ in $\divi(f)$ is either a good weight $3$ point (see Definition~\ref{D:goodtype3}) or has $\wt_P \leq 2$. In particular, we have $-(\Art(X^f)) < \nu(\Delta_f)$ if there exists $P$ in $\divi(f)$ with $\wt_P \geq 4$.
\end{thm}
We also prove the following corollaries to this theorem in Section~\ref{finalproof}, by showing that the conditions in this theorem are automatically satisfied in the setting of the minimal Weierstrass equation for an elliptic curve, thus explaining the Ogg-Saito equality in genus $1$ and inequality in higher genus.
\begin{corollary}\label{Rgenus1}
 Assume that $\deg(f)$ is $3$ and that $y^2=f(x)$ is a minimal Weierstrass equation. Then  $-\Art(X^f) = \nu(\Delta_f)$.
\end{corollary} 
%If $f$ is a degree $3$ polynomial giving rise to the minimal Weierstrass equation for an elliptic curve, then every $P$ in $\divi(f)$ satisfies the conditions in Theorem~\ref{Cexbal}, thus explaining the Ogg-Saito equality in genus $1$. (See Lemma~\ref{Rgenus1}.)
%The combinatorics necessary for equality is reminiscent of the trivalent graphs that appear in many results in tropical geometry. 

\begin{corollary}\label{Ctoomuchcollision}
 We have strict inequality $-(\Art(X^f))< \nu(\Delta_f)$ whenever four or more irreducible factors of $f$ specialize to the same point in the standard model $\P^1_R$ (``non-generic collision of roots''). 
\end{corollary}

% 
% \begin{cor}(See Corollary~\ref{Ctoomuchcollision} and Remark~\ref{Rgenus1})
%  If $C$ is a hyperelliptic curve with minimal Weierstrass equation $y^2=f(x)$ and four or more irreducible factors of $f \in R[x]$ have the same reduction in $k[x]$, then we have $-(\Art(C/K)) < \nu(\Delta_C)$.
% \end{cor}

Using Theorem~\ref{Cexbal}, we are able to produce examples of hyperelliptic curves with bad reduction in every genus where we have equality and inequality. 
\begin{example} Let $a_1,a_2,\ldots,a_{2g-1}$ be any $2g-1$ elements of $R$ with pairwise distinct reductions in $k$. 
 \begin{itemize}
 \item If $C$ is the genus $g$ hyperelliptic curve given by $y^2=x^{2g+2}-t$, then we have $-(\Art(C/K)) = \nu(\Delta_C)$.
  \item Let $C$ be the genus $g$ hyperelliptic curve $y^2=(x-a_1)(x-a_1+t)(x-a_2)(x-a_2+t)\ldots(x-a_g)(x-a_g+t)$. Then $-(\Art(C/K)) = \nu(\Delta_C)$.
  \item Let $C$ be the hyperelliptic curve $y^2=(x-a_1)(x-a_1+t)(x-a_1-t)(x-a_2)(x-a_3)\ldots(x-a_{2g-1})$. Then $-(\Art(C/K)) < \nu(\Delta_C)$.
 \end{itemize}
\end{example}

When $g \geq 2$, since both the Artin conductor and the minimal discriminant are nonzero precisely when the curve $C$ has bad reduction, one might also ask if there is an inequality between the conductor and the discriminant in the other direction. The difference between the two invariants can be as large as a quadratic function in $g$. (See Example~\ref{Eotherdirection}.) 

For our explicit proper regular possibly non minimal model $X^f$ for the hyperelliptic curve, we can show $\nu(\Delta_C) \leq (g+1)(2g+1)(-\Art(X^f))$. (See Remark~\ref{Rotherdirection}.) An analogous inequality with a different notion of discriminant is proven in the semistable case in \cite[Th\'{e}or\`{e}me~1.1]{Mau}, by proving effectivity of a certain Cartier divisor on a moduli space. This leads us to the following question which we do not answer in this paper (since we have not analyzed how many contractible components our model $X^f$ might have).
\begin{question}
 Is there an explicit quadratic function $c(g)$ such that $\nu(\Delta_C) \leq c(g)(-\Art(C/K))$?
\end{question}

\subsection{Summary of earlier work on conductor-discriminant inequalities}
In genus $1$, the proof of the Ogg-Saito formula used the explicit classification of special fibers of minimal regular models of genus $1$ curves. In genus $2$, \cite{liup} defines another discriminant that is
  specific to genus $2$ curves, and compares both the Artin conductor and the minimal discriminant (our $\nu(\Delta_C)$, which Liu calls $\Delta_0$) to this third discriminant (which Liu calls $\Delta_{\min}$). This
  third discriminant $\Delta_{\min}$ is sandwiched between the Artin conductor and the minimal discriminant and is defined using a possibly non-integral Weierstrass equation such that the associated differentials generate the $R$-lattice of global sections of the relative dualizing sheaf of the minimal regular model.  It does
not directly generalize to higher genus hyperelliptic curves (but see \cite[Definition~1, Remarque~9]{liup} for a related conductor-discriminant question). Liu even provides an explicit formula for the difference between the Artin conductor and both $\Delta_0$ and $\Delta_{\min}$ that can be described in terms of the
combinatorics of the special fiber of the minimal regular model (of which there are already over $120$ types!).

Since these invariants are insensitive to unramified base extensions, we may assume that $k$ is algebraically closed. We also fix a polynomial $f \in R[x]$ such that $\nu(\Delta_C) = \nu(\Delta_f)$. The common starting point of \cite{conddisc} and this paper is to produce an explicit regular model $X^f$ admitting a finite degree $2$ map to an explicit regular model $Y^f$ of $\mathbb{P}^1_K$. It suffices to show $-\Art(X^f) \leq \nu(\Delta_f)$, since $-\Art(C/K) \leq -\Art(X^f)$. The model $Y^f$, which we call the good embedded resolution of the pair $(\P^1_R, \divi(f))$, is a blowup of $\P^1_{R}$ on which all components of $\divi(f)$ of odd multiplicity are regular and disjoint (Definition~\ref{defgoodembres}). The normalization of $Y^f$ in the function field of the hyperelliptic curve $X^f$ is an explicit regular model for $C$. In \cite{conddisc}, the assumption that the roots of $f$ are defined over $K$ ensures that all irreducible components of $\divi(f)$ are already regular in the standard model $\P^1_{R}$, and we only have to deal with making the odd multiplicity components of $\divi(f)$ disjoint. The conductor-discriminant inequality for $f$ is then proven by decomposing both $-\Art(X^f)$ and $\nu(\Delta_f)$ into local terms indexed by the vertices of the dual tree of $Y^f$. When the roots of $f$ are not defined over $K$, this analysis is much more involved, since we now need to carry out explicit embedded resolution of $\divi(f)$ in $\P^1_{R}$.

\subsection{Outline of this paper}
\subsubsection{\textbf{Explicit regular models}}
In Section~\ref{embedding}, we show that we may reduce to the case $R=k[[t]]$ by producing a polynomial $f^\sharp \in k[[t]][x]$ such that $\nu(\Delta_f) = \nu(\Delta_f^\sharp)$ and $-\Art(X^f)=-\Art(X^{f^\sharp})$. Our assumption that $\cha(k) > \deg(f)$ ensures that the roots of $f$ are defined over a tame cyclic Galois extension. This allows us to write down Newton-Puiseux expansions for the roots of $f$ with bounded denominators. The continued fraction expansions of a finite set of special exponents in these Newton-Puiseux expansions, called the characteristic exponents (see Definition~\ref{charexponents}) control the combinatorics of the special fiber of the model $Y^f$, and in turn $-\Art(X^f)$. The dual graphs of these embedded resolutions can be computed using the explicit resolution algorithm described in \cite[Theorem~3.3.1,Lemma~3.6.1]{Wall}. For the rest of the paper, it is assumed that $R=k[[t]]$.

\subsubsection{\textbf{An inductive argument and the base case}}
The proof of the conductor-discriminant inequality is an induction on the ordered pair of integers $(\deg(f),\nu(\Delta_f))$. The base case of this induction is when $f$ factors as a product of linear and shifted Eisenstein polynomials with distinct specializations in $\P^1_R$. Here $Y^f$ is $\P^1_R$ and $X^f$ is the Weierstrass model, which is regular in this case. A direct computation then shows that we have $-\Art(X^f) = \nu(\Delta_f)$ (Section~\ref{Sbasecase}). When $f$ is not of this form, we study the equation of the strict transform of $f$ after a blowup of $\P^1_R$ at the images of the nonregular points of the Weierstrass model. We use this equation along with a change of variables to define a set of replacement polynomials for the polynomial $f$ (Definition~\ref{defnreppoly}). The key calculation is to compare the conductor (and respectively the discriminant) of the polynomial $f$ to the sum of the conductors (and respectively the discriminants) of its replacement polynomials. We show that the change on the conductor side is less than or equal to the change in the discriminant side (Theorem~\ref{replacementchange}). Adding the conductor-discriminant inequalities for the replacement polynomials (which we know from the induction hypothesis) to the key inductive inequality then proves the conductor-discriminant inequality for $f$. 

Section~\ref{induction} is devoted to defining the replacement polynomials. 

\subsubsection{\textbf{Change on the conductor side during induction}}
In Section~\ref{secconductorchange}, we use the inclusion-exclusion property of the Euler-characteristic along with the Riemann-Hurwitz formula to compute the difference between the conductor of $f$ and the sum of the conductors of its replacement polynomials. This difference is the left hand side of the key inductive inequality. 

\subsubsection{\textbf{The metric tree}}
The right hand side of the key inductive inequality, which is the difference between the discriminant of $f$ and the sum of the discriminants of its replacement polynomials, is harder to compute. For this, we first  replace the discriminant $f$ by a combinatorial refinement of it, which we call the metric tree of $f$. The metric tree is introduced in Section~\ref{secmettree}. See Figure~\ref{fig:MetTree1} for an example.

%\centering
%\def\r{\color{red}}\def\b{\color{blue}}\def\g{\color{ForestGreen}}
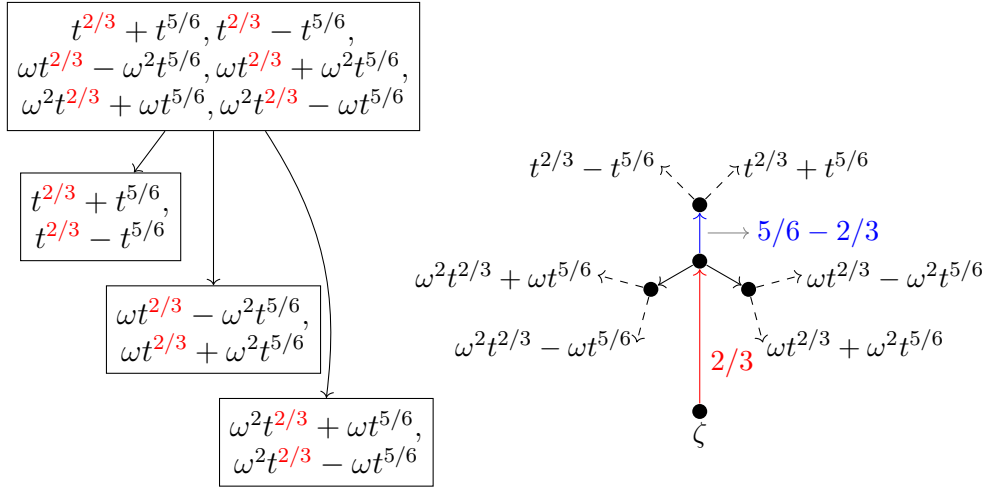
\begin{figure}[hbt!]
\hbox to \textwidth{\hss
\begin{tikzpicture}[every node/.style = {shape=rectangle, draw, align=center}, edge from parent/.style={->=latex,draw},]
  \node {$t^{\r{2/3}}+t^{5/6},t^{\r{2/3}}-t^{5/6},$\\ $\omega t^{\r{2/3}}-\omega^2 t^{5/6},\omega t^{\r{2/3}}+\omega^2 t^{5/6},$ \\ $\omega^2 t^{\r{2/3}}+\omega t^{5/6},\omega^2 t^{\r{2/3}}-\omega t^{5/6}$}
    child { node [yshift=-0.5cm] {$t^{\r{2/3}}+t^{5/6},$ \\ $t^{\r{2/3}}-t^{5/6}$} }
    child { node [yshift=-2cm] {$\omega t^{\r{2/3}}-\omega^2 t^{5/6},$ \\ $\omega t^{\r{2/3}}+\omega^2 t^{5/6}$} }
    child { node [yshift=-3.5cm] {$\omega^2 t^{\r{2/3}}+\omega t^{5/6},$ \\ $\omega^2 t^{\r{2/3}}-\omega t^{5/6}$} edge from parent[edge from parent path={(\tikzparentnode) to[bend left=20] (\tikzchildnode.north)}]
             };
\end{tikzpicture}
\hskip -3em
\tikzstyle{my dot}=[fill,circle,inner sep=1pt]
\tikzstyle{level 1}=[sibling angle=120]
\tikzstyle{level 2}=[sibling angle=90]
\tikzstyle{level 3}=[sibling angle=60]
\tikzstyle{level 4}=[sibling angle=30]
\tikzstyle{small node}=[inner sep=0pt, outer sep=0pt]
\begin{tikzpicture}[rotate=90,grow cyclic,edge from parent/.style={->,draw},scale=0.5, every node/.style={font=\small},circle, inner sep=2pt, every label/.style={rectangle,inner sep=1pt,outer sep=0pt}, baseline=-20mm]
\node (eta) at (-2,0) [fill,label=below:$\zeta$] {};
\node (noname) at (2,0) [fill] {} 
  child { node [fill] {} 
    child {node [small node,label=right:$\omega t^{2/3} + \omega^2 t^{5/6}$] {} edge from parent[dashed]} 
    child {node [small node,label=right:$\omega t^{2/3} - \omega^2 t^{5/6}$] {} edge from parent[dashed]} } 
  child { node [fill] {} 
    child {node [small node,label=right:$t^{2/3} + t^{5/6}$] {} edge from parent[color=black,dashed]} 
    child {node [small node,label=left:$t^{2/3} - t^{5/6}$] {} edge from parent[color=black,dashed]} edge from parent[color=blue] node[pin={[overlay,font=\small]0:$\b{5/6-2/3}$}]{}}
  child { node [fill] {} 
    child {node [small node,label=left:$\omega^2 t^{2/3} + \omega t^{5/6}$] {} edge from parent[dashed]} 
    child {node [small node,label=left:$\omega^2 t^{2/3} - \omega t^{5/6}$] {} edge from parent[dashed]}}; 
\draw [color=red,->] (eta) -- node[anchor=west,pos=.3, xshift=-0.5pt] {$\r{2/3}$} (noname);
\end{tikzpicture}
\hss}
\caption{Figure~\ref{fig:MetTree1}: Metric tree of the minimal polynomial of $t^{2/3}+t^{5/6}$ over $\mathbb{C}((t))$} \label{fig:MetTree1}
\end{figure}

The metric tree keeps track of the relative $t$-adic distances between all pairs of roots of $f$. It is easy to recover the discriminant of $f$ from its metric tree (Lemma~\ref{gromov}). The Galois action on the roots of $f$ extends to a Galois action on the whole metric tree. For example, if $f$ is irreducible and its roots have valuation $a/b < 1$ with $\gcd(a,b)=1$, then the metric tree of $f$ has $b$ identical subtrees glued onto one end of a segment of length $a/b$, as in Figure~\ref{fig:MetTree1} and the Galois action permutes these subtrees, keeping points on the line segment fixed. 

 \subsubsection{\textbf{The change on the discriminant side and Abhyankar's inversion formula}}
We exploit this symmetry of the metric tree, along with a refinement of Abhyankar's inversion formula from \cite{GGP} to describe the metric tree of the replacement polynomials from the metric tree of $f$ (Theorem~\ref{mettreerep1} and Theorem~\ref{mettreerep2}). Continuing with the same setup as before, if $f$ is irreducible with roots of valuation $a/b$, its replacement polynomial $g$ is also irreducible with $\deg(g) = \deg(f) a/b$. Furthermore, Abhyankar's inversion formula can be used to prove that the metric tree of $g$ is obtained by gluing $a$ identical subtrees to one end of a line segment of length $b/a-1$. The subtrees in the replacement polynomial are identical to the subtrees in $f$, except that the metric gets scaled by a factor of $b/a$. When $f$ has multiple irreducible factors, we compute the replacement polynomials of each irreducible factor separately and use a recent refinement of the inversion formula to show how to glue them together appropriately. Once we have the metric tree of the replacement polynomial, we can use Lemma~\ref{gromov} once again to compute the discriminants of the replacement polynomials, and in particular the difference in the discriminant of $f$ and its replacement polynomials (Section~\ref{Sdiscchange}).

\begin{figure}[hbt!]
\def\r{\color{red}}\def\b{\color{blue}}\def\g{\color{ForestGreen}}
\tikzstyle{my dot}=[fill,circle,inner sep=1pt]
\tikzstyle{level 1}=[sibling angle=120, level distance=2cm]
\tikzstyle{level 2}=[sibling angle=90]
\tikzstyle{level 3}=[sibling angle=60]
\tikzstyle{level 4}=[sibling angle=30]
\tikzstyle{small node}=[inner sep=0pt, outer sep=0pt]
\begin{tikzpicture}[scale=0.3, rotate=90,grow cyclic,edge from parent/.style={->,draw},]
\node (noname) at (8,0) [my dot] {} 
  child { node [my dot] {} 
    child {node [small node] {} edge from parent[dashed]} 
    child {node [small node] {} edge from parent[dashed]} }
  child { node [my dot] {} 
    child {node [small node] {} edge from parent[dashed]} 
    child {node [small node] {} edge from parent[dashed]} edge from parent node[pin={[font=\small]0:${1/6}$}]{} } 
  child { node [my dot] {} 
    child {node [small node] {} edge from parent[dashed]} 
    child {node [small node] {} edge from parent[dashed]} }; 
\node (eta) [below of=noname,node distance = 0.3*8cm, my dot] {};
\draw [->,color=red] (eta) -- node[anchor=west, pos=0.3,font=\small] {$\r{2/3}$} (noname);
\end{tikzpicture}
\hskip -0.5em
\tikzstyle{my dot}=[fill,circle,inner sep=1pt]
\tikzstyle{level 1}=[sibling angle=180, level distance=3cm]
\tikzstyle{level 2}=[sibling angle=90]
\tikzstyle{level 3}=[sibling angle=60]
\tikzstyle{level 4}=[sibling angle=30]
\tikzstyle{small node}=[inner sep=0pt, outer sep=0pt]
\begin{tikzpicture}[scale=0.3, rotate=90,grow cyclic,edge from parent/.style={->,draw},]
\node (noname) at (6,0) [my dot] {} 
  child { node [my dot] {} 
    child {node [small node] {} edge from parent[dashed]} 
    child {node [small node] {} edge from parent[dashed]} edge from parent node[pin={[overlay,font=\small]90:${{\g{1/4}}=(1/6)\cdot{\r{(3/2)}}}$}]{} } 
  child { node [my dot] {} 
    child {node [small node] {} edge from parent[dashed]} 
    child {node [small node] {} edge from parent[dashed]} }; 
\node (eta) [below of=noname,node distance = 0.3*6cm, my dot] {};
\draw [->,color=blue] (eta) -- node[overlay,anchor=west, pos=0.3,font=\small,xshift=-0.3em] {${\b{1/2}}={\r{(3/2)}}-1$} (noname);
\node[overlay,shape=rectangle] at (6,7) {$\rightsquigarrow$};
\node[overlay,shape=rectangle] at (6,-7) {$\rightsquigarrow$};
\end{tikzpicture}
\hskip 0.5em
\tikzstyle{my dot}=[fill,circle,inner sep=1pt]
\tikzstyle{level 1}=[sibling angle=90, level distance=6cm]
\tikzstyle{level 2}=[sibling angle=90]
\tikzstyle{level 3}=[sibling angle=60]
\tikzstyle{level 4}=[sibling angle=30]
\tikzstyle{small node}=[inner sep=0pt, outer sep=0pt]
\begin{tikzpicture}[scale=0.3, grow cyclic,edge from parent/.style={->,draw},]
\node (noname) at (12,0) [my dot] {} 
  child { node [my dot] {} 
    child {node [small node] {} edge from parent[dashed]} 
    child {node [small node] {} edge from parent[dashed]} edge from parent node[pin={[overlay,font=\small]90:${1/2={\g{(1/4)}}\cdot{\b{(2/1)}}}$}]{} }; 
\node (eta) [below of=noname,node distance = 0.3*12cm,my dot] {};
\draw [->] (eta) -- node[overlay,anchor=west,pos=0.5,font=\small] {${1} = {\b{(2/1)}} - 1$} (noname);
\end{tikzpicture}
\caption{Figure~\ref{fig:MetInduction}: Metric tree of an irreducible $f$ $\rightsquigarrow$ Metric tree of its replacement polynomial} \label{fig:MetInduction}
\end{figure}
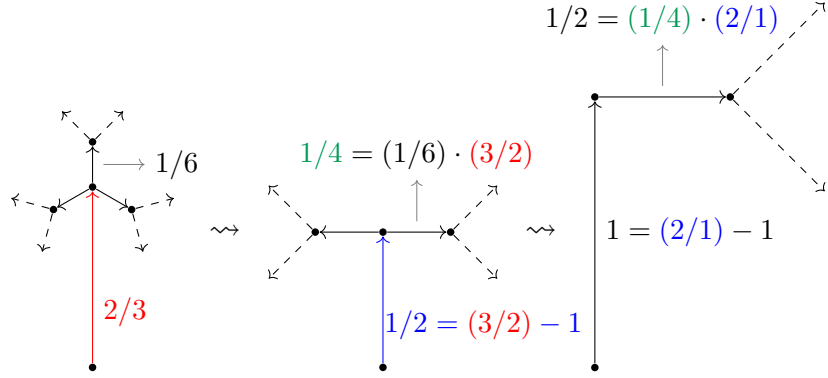

\subsubsection{\textbf{Termination of induction}}
Finally, in Section~\ref{finalproof}, we put together the results of the previous three sections to prove the key inductive inequality (Theorem~\ref{replacementchange}). We prove that the induction terminates (Corollary~\ref{Cdiscind}), and study the exact combinatorial restrictions needed for equality to hold (Theorem~\ref{Cexbal}).

\subsection{Related work}
In the semistable case, work of Kausz \cite{Kau} (when $p \neq 2$) and Maugeais \cite{Mau} (all $p$) compares
the Artin conductor to yet another notion of discriminant. 

Concurrent to and independent of our work, the authors of \cite{DDMM} introduced the cluster picture of a polynomial $f$, which is the same as the metric tree introduced in this paper. The authors compute many arithmetic invariants attached to hyperelliptic curves in terms of the cluster picture of $f$ in the semistable case. In contrast, our results do not require the semistability hypothesis. 

In \cite{Kohls}, Kohls compares the conductor exponent $\phi$ for the Galois representation $\Gal (\overline{K}/K) \rightarrow \Aut_{\Q_\ell} (H^1_{\mathrm{et}}(X_{\overline{K}}, \Q_\ell))$ with the minimal discriminant of superelliptic curves, by studying the Galois action on the special fiber of the semistable model as in \cite{BW_Glasgow}. In \cite{BKSW}, the authors define minimal discriminants of Picard curves (degree $3$ cyclic covers of $\P^1_K$) and compare the conductor exponent and the minimal discriminant for such curves. Our results are stronger than these results in the case of hyperelliptic curves, since $-\Art(C/K)=n-1+\phi$, where $n$ is the number of irreducible components in the special fiber of the minimal proper regular model of $C$.

In \cite{NowFar}, Faraggi and Nowell describe the special fibers of snc models of hyperelliptic
curves when the splitting field of $f$ is tamely ramified. Their approach is to resolve the tame quotient singularities that show up when you take the quotient of the semistable model (which they explicitly describe using the cluster picture/metric tree) by the Galois action. We cannot directly use their constructions, since the conductor-discriminant inequality does not hold with the minimal snc-model in place of the minimal regular model. This inequality already fails in genus $1$ when the minimal regular model does not coincide with the minimal snc-model.

The conductor-discriminant inequality also holds in the wild case when $\delta \neq 0$ in genus $1$ and genus $2$ due to Ogg, Saito and Liu. In \cite{oddpaper}, in joint work with Andrew Obus, we extend the conductor-discriminant inequality to all hyperelliptic curves when $\cha(k) \neq 2$, using the so-called ``Mac Lane valuations''. These give an explicit way of describing the entire regular resolution directly in terms of lower degree approximations of the roots of $f$, without having to write down Newton-Puiseux expansions of the roots of $f$ first. Our results in \cite{oddpaper} reprove the results in this paper using different techniques and also covers wild ramification. However, the combinatorial criterion for equality $-(\Art(X^f)) = \nu(\Delta_f)$ is more transparent and easier to analyze using the techniques in this paper, since we analyze the the difference in the two sides of the inequality after each blow up instead of writing the entire regular model all at once. We also hope that the inductive argument on metric trees would be of independent interest to the more combinatorially-inclined reader.

\subsection{Notation}\label{notation}
The invariants $-\Art(X/S)$ and $\nu(\Delta)$ are unchanged when we extend scalars to the strict Henselization. So from the very beginning, we let $R$ be a complete discrete valuation ring with algebraically closed residue field $k$. Assume that $\cha k \neq 2$. Let $K$ be the fraction field of $R$ and $\overline{K}$ be a separable closure of $K$. Let $\nu \colon \overline{K} \rightarrow \Q \cup \{\infty\}$ be the unique extension of the discrete valuation on $K$ to $\overline{K}$. Let $t \in R$ be a uniformizer; $\nu(t) = 1$. Let $S = \Spec R$. Let $C$ be a hyperelliptic curve over $K$ with genus $g \geq 2$. 

Let $y^2-f(x) = 0$ be an {\textup{\color{blue}integral Weierstrass equation}} for $C$, i.e., $f(x) \in R[x]$ and $C$ is birational to the plane curve given by this equation. The {\textup{\color{blue}discriminant}} of a Weierstrass equation ${\color{blue}d_f}$ equals the discriminant of $f$ considered as a polynomial of degree $2g+2$. A {\textup{\color{blue}minimal Weierstrass equation}} for $C$ is a Weierstrass equation for $C$ such that $\nu(d_f)$ is as small as possible amongst all integral Weierstrass equations for $C$. The {\textup{\color{blue} minimal discriminant $\nu (\Delta_C)$ of $C$}} equals $\nu(d_f)$ for a minimal Weierstrass equation $y^2-f(x)$ for $C$. Fix such an equation. 

For any proper regular curve $Z$ over $S$, we will denote the special fiber of $Z$ by $Z_s$, the generic fiber by $Z_\eta$ and the geometric generic fiber by $Z_{\overline{\eta}}$. We will denote the function field of an integral scheme $Z$ by $K(Z)$, the local ring at a point $z$ of a scheme $Z$ by $\O_{z}$ and the unique maximal ideal in $\O_{z}$ by $\mathfrak{m}_{z}$. For $f \in K(Z)$, we will denote the divisor of $f$ by $\divi(f)$ and the divisor of zeroes of $f$ by $\divi_0(f)$. The reduced scheme attached to a scheme $Z$ will be denoted $Z_{\mathrm{red}}$. If $Z$ is a smooth divisor on a smooth scheme $Z'$, then we will denote the corresponding discrete valuation on $K(Z')$ by $\nu_Z$.

We will let $\mathbb{P}^{1,\mathrm{Berk}}_{L}$ denote the Berkovich projective line over the field $L$, and let $\zeta$ denote its Gauss point. 

\section{The Artin conductor/Deligne discriminant}
 Let $X$ be an integral proper $S$-scheme of relative dimension $1$. Fix $\ell \neq \cha k$. Let $\chi$ denote the compactly-supported Euler-characteristic for the $\ell$-adic \'etale topology. Let $\delta$ be the Swan conductor associated to the $\ell$-adic representation $\mathrm{Gal}\ (\overline{K}/K) \rightarrow \mathrm{Aut}_{\Q_\ell} \ (H^1_{\mathrm{et}}(X_{\overline{\eta}}, \Q_\ell))$ ($\ell \neq \mathrm{char}\ k$) \cite[p.153]{saito2}.

\begin{defn}
 The (negative of) the {\textup{\textsf{\color{blue}Artin conductor}}} of $X$, or alternately, the {\textup{\textsf{\color{blue}Deligne discriminant}}} of $X$, denoted {\textup{\textsf{\color{blue}$-\Art (X/S)$}}} is given by
 \[-\Art (X/S) \colonequals \chi(X_s)-\chi(X_{\overline{\eta}})+\delta .\]
\end{defn}

Let $Y$ be a regular integral $2$-dimensional $S$-scheme and let $f$ be a rational function on $Y$ that is not a square. Assume that the residue field at any closed point of $Y$ is not of characteristic $2$. Let $X$ be the normalization of $Y$ in $K(Y)(\sqrt{f})$. Let $\divisor(f) = \sum_{i \in I} m_i \Gamma_i$, and let $B = \sum_{m_i \ \text{odd}} \Gamma_i$.  
\begin{lemma}\label{formula} Keep the notation from the paragraph above. Assume $\cha(k) \neq 2$.
 \[ -\Art (X/S) = 2(\chi(Y_s)-\chi(Y_{\overline{\eta}}))-(\chi(B_s)-\chi(B_{\overline{\eta}})) + \delta .\]
 If $X_{\overline{\eta}}$ is a hyperelliptic curve with equation $y^2=f(x)$ and $\cha(k) > \deg(f)$, then $\delta = 0$.
\end{lemma}
\begin{proof} This is the Riemann-Hurwitz formula applied to the finite branched tame degree $2$ covers $X_{\overline{\eta}} \rightarrow Y_{\overline{\eta}}$ and $X_s \rightarrow Y_s$. Let $R$ be the inverse image of $B$ in $X$; then the map $V \colonequals X \setminus R \rightarrow U \colonequals Y \setminus B$ is \'{e}tale. Since $\chi(V) = d \chi(U)$ for any tame \'{e}tale degree $2$ cover $V \rightarrow U$ of varieties over an algebraically closed field of characteristic $\neq \ell$ and since $\cha(k) \neq 2$, we have $\chi(X_s \setminus R_s) = 2 \chi(Y_s \setminus B_s)$ and $\chi(X_{\overline{\eta}} \setminus R_{\overline{\eta}}) = 2 \chi(Y_{\overline{\eta}} \setminus B_{\overline{\eta}})$. Since $k$ is algebraically closed and $\ell$-adic \'{e}tale cohomology satisfies the same dimension and exactness axioms as singular cohomology, the proof of the formula now follows from excision. If $X_{\overline{\eta}}$ is hyperelliptic and $\cha(k) > \deg(f)$, then $\cha(k) > 2g+1$ and hence $\delta=0$.
\end{proof}

\section{An explicit regular model}
In this section, we construct a good regular model of a hyperelliptic curve with minimal Weierstrass equation $y^2=f(x)$ by first constructing a suitable embedded resolution of the pair $(\mathbb{P}^1_R,\divisor(f))$ (Lemma~\ref{goodembres}), and then taking its normalization in a degree $2$ extension of its function field (Lemma~\ref{useful}). We also prove a lemma about when two such embedded resolutions of pairs $(Y,\Gamma)$ and $(Y',\Gamma')$ are isomorphic (Lemma~\ref{blowupiso}), which we will use in Section~\ref{induction} for inductive arguments.

\begin{defn}
 Given a regular arithmetic surface $Y$ and a Weil divisor $\Gamma = \sum m_i \Gamma_i$, define the {\textup{\textsf{\color{blue}underlying odd divisor}}} $\Gamma_{\mathrm{odd}}$ by $\Gamma_{\mathrm{odd}} \colonequals \sum_{m_i \mathrm{odd}} m_i \Gamma_i$  .
\end{defn}

\begin{defn}\label{defgoodembres} Let $Y \rightarrow S$ be a regular arithmetic surface, let $f \in K(Y)$. A {\textup{\textsf{\color{blue}good embedded resolution}}} of the pair $(Y,\divisor(f))$ is another regular arithmetic surface $Y'$ such that $Y'$ fits in a sequence $Y'\colonequals Y_n \rightarrow Y_{n-1} \cdots \rightarrow Y_1 \rightarrow Y_0 \colonequals Y$, where each $Y_{i}$ is obtained by blowing up the nonregular points of the closed subscheme $[\divisor(f)_{\mathrm{odd}}]_{\mathrm{red}}$ on $Y_{i-1}$, and such that the divisor $[\divisor(f)_{\mathrm{odd}}]_{\mathrm{red}}$ on $Y'$ is regular. 
\end{defn}

\begin{lemma}\label{goodembres}\hfill
\begin{enumerate}[\upshape(a)]
 \item A good embedded resolution exists for every pair $(Y,\divisor(f)$ as above.
 \item If $Y' \rightarrow Y$ is a good embedded resolution of $\divisor(f)$ and $\divisor(f)_{\mathrm{odd}} = \sum m_i \Gamma_i$ on $Y'$, then $\Gamma_i$ is regular for every $i$ and $\Gamma_i$ and $\Gamma_j$ do not intersect if $i \neq j$.
\end{enumerate} 
\end{lemma}
\begin{proof}\hfill
\begin{enumerate}
  \item Since $R$ is assumed to be a complete discrete valuation ring and therefore excellent, we may use the results of \cite[Chapter~9]{liu}. The construction of $Y'$ is analogous to the proof of embedded resolution in \cite[p.404, Chapter~9, Theorem~2.26]{liu} and we sketch the details. We first blow up closed points of $Y$ to make the irreducible components of $\divisor(f)_{\mathrm{odd}}$ regular as in \cite[p.405, Chapter~9, Lemma~2.32]{liu}, and then do some further blowups to separate components of $\divisor(f)_{\mathrm{odd}} \subset Y$ as in the construction of a normal crossings model in \cite[p.404, Chapter~9, Theorem~2.26]{liu}. The main difference is that we do not care about making horizontal components of $\divisor(f)_{\mathrm{odd}}$ transverse to exceptional curves that appear with even multiplicity. Once we get to the point that at most two irreducible components $\Gamma$ and $\Gamma'$ of $\divisor(f)_{\mathrm{odd}}$ pass through any point $P$, then one further blowup at $P$ produces a curve that appears with even 
multiplicity in $\divisor(f)$ and separates $\Gamma$ and $\Gamma'$ (\cite[Lemma~2.3]{conddisc}).  
  \item The reduced curve $[\divisor(f)_{\mathrm{odd}}]_{\mathrm{red}}$ on the regular surface $Y'$ is locally given by the vanishing of a single function by \cite[p.117, Chapter~4, Proposition~1.12]{liu}. By \cite[p.378, Chapter~9, Proposition~1.8]{liu}, the zero locus of a single function on a regular surface is regular at a point $P$ if and only if the function is not in $\mathfrak{m}_P^2$. Putting these two facts together, it follows that $[\divisor(f)_{\mathrm{odd}}]_{\mathrm{red}}$ is regular if and only if its irreducible components are regular and pairwise disjoint. \qedhere
\end{enumerate}
\end{proof}

\begin{remark}\label{localdefgoodembres}
 From the local nature of the construction, we see that we may also talk about the good embedded resolution of a pair $(\mathcal{O},\divisor(f))$, where $\mathcal{O}$ is a regular $2$-dimensional $R$-algebra and $f \in \mathcal{O}$. 
\end{remark}

We record the following corollary which will be useful for inductive arguments in Section~\ref{induction}.
\begin{corollary}\label{blowupiso}
 Let $\mathcal{O}$ be a regular $2$-dimensional $k$-algebra. Let $u,q,g,h \in \mathcal{O}$ be such that $g = uq^2h$ and $u$ is a unit in $\mathcal{O}$. Then the good embedded resolutions of the pairs $(\mathcal{O},\divisor(g))$ and $(\mathcal{O},\divisor(h))$ are isomorphic. Furthermore $\divisor(g)_{\mathrm{odd}} = \divisor(h)_{\mathrm{odd}}$ on the resolution. 
\end{corollary}
\begin{proof} Since $g = uq^2h$ implies that $\divisor(g)_{\mathrm{odd}} = \divisor(h)_{\mathrm{odd}}$, it follows from Definition~\ref{defgoodembres} and Remark~\ref{localdefgoodembres} that the good embedded resolutions of the pairs $(\mathcal{O},\divisor(g))$ and $(\mathcal{O},\divisor(h))$ are isomorphic. 
\end{proof}

\begin{defn}\label{models}
 Let ${\color{blue}{Y^f}}$ be a good embedded resolution of the pair $(\mathbb{P}^1_R,\divisor(f))$, and let the branch locus ${\color{blue}{B^f}} \colonequals \divisor(f)_{\mathrm{odd}}$ on $Y^f$.
 Define ${\color{blue}{X^f}}$ to be the normalization of $Y^f$ in $K(x)[y]/(y^2-f(x))$.
\end{defn}

\begin{lemma}\label{useful}
 The model $X^f$ is regular.
\end{lemma}
\begin{proof}
This follows from \cite{conddisc}[Lemma~2.1].
\end{proof}

\section{Reduction to the equicharacteristic case}\label{embedding}

The goal of this section is to show that we may assume $R=k[[t]]$ without any loss of generality. Fix a (set-theoretic) section $k \rightarrow R$ of the natural surjective reduction map $R \rightarrow k$ sending $0$ to $0$. (If $k \subset R$, fix the identity section.) Elements in the image of this section will be called {\textup{\textsf{\color{blue}{lifts}}}}. Let $n$ be a positive integer coprime to $p$. Every element $a \in R[t^{1/n}]$, has a unique expansion of the form $a = \sum_{m \in \Z_{\geq 0}} a_m t^{m/n}$ (the {\textup{\textsf{\color{blue}{Newton-Puiseux expansion}}}}) such that every $a_m$ is a lift.

We will let $\nu$ denote the discrete valuation on both $\bigcup_{n > 1, (n , \cha k) = 1}R[t^{1/n}]$ and $\bigcup_{n > 1, (n , \cha k) = 1} k[[t^{1/n}]]$. If $a = \sum_{m \in \Z_{\geq 0}} a_m t^{m/n}$ is the Newton-Puiseux/$t$-adic expansion of an element in one of these rings, then $\nu(a) = m/n$, where $m$ is the smallest integer with $a_m \neq 0$.
\begin{prop}\label{tiltdisc} 
Let $f \in R[x]$ be a separable polynomial with $\deg f < \cha k$ if $\cha k > 0$. Then there exists a separable polynomial $f^\sharp \in k[[t]][x]$ of the same degree with the following properties.
\begin{enumerate}[\upshape (a)] 
\item There is a bijection of the roots $\{\alpha_1,\ldots,\alpha_r\}$ of $f$ with the roots $\{\beta_1,\beta_2,\ldots,\beta_r\}$ of $f^\sharp$ that satisfy
 \begin{itemize}
  \item $\nu(\alpha_i) = \nu(\beta_i)$ for all $i$, and
  \item $\nu(\alpha_i - \alpha_{i'}) = \nu(\beta_i - \beta_{i'})$ for all $i \neq i'$.
 \end{itemize}
 In particular $\Delta_{f} = \Delta_{f^\sharp}$.
 \item The special fibers of the models $X_f$ and $X_{f^\sharp}$ from Definition~\ref{models} are isomorphic. In particular $\Art(X_f) = \Art(X_{f^\sharp})$. 
\end{enumerate} 
\end{prop}
\begin{proof} 
 Since $\cha k > \deg (f)$ and $k$ is algebraically closed, the splitting field of $f$ is a totally ramified tame extension of $K$, and therefore cyclic \cite[Chapter~IV,\S~2,Proposition~7,Corollary~2]{serrelocal}. Since $K$ is complete and $k$ is algebraically closed, by Kummer theory, we may further assume that this splitting field equals $K(t^{1/n})$ for some integer $n \geq 1$. Since $f$ is monic and integral, it follows that all roots of $f$ are contained in $R[t^{1/n}]$. Let $g_1,g_2,\ldots,g_l$ be the irreducible factors of $f$, and let $f=ut^bg_1 \ldots g_l$ for some $b \in \{0,1\}$ and unit $u \in R$. For each irreducible factor $g_j$, pick a root $\alpha_{i}$ of $g_j$ and write down its Newton-Puiseux expansion $\alpha_i \colonequals \sum_{m \in \Z_{\geq 0}} a_m^i t^{m/n}$.  Since $f$ is separable, there exists an integer $M$ such that
 \begin{itemize}
  \item for every $i$, we have $\deg g_i = \lcm (\denom(m/n) \ | \ m \leq M, a_m^{i} \neq 0)$, where $\denom(m/n)$ is the denominator of the rational number $m/n$ when written in lowest form.
  \item whenever $i \neq j$, we have $\alpha_i \mod t^{M/n} \neq \alpha_j \mod t^{M/n}$ in $R[t^{1/n}]$.
 \end{itemize}
 
  Define
 \begin{align*} g_i^\sharp &\colonequals \textup{The minimal polynomial of } \sum_{m =0}^M \overline{a_m^{i}} t^{m/n} \textup{ in } k[[t]][x] \\
 f^\sharp &\colonequals t^b\prod_{i=1}^l g_i^\sharp.\end{align*}
 \begin{enumerate}[\upshape (a)]
  \item For every $i$ with $1 \leq i \leq l$, fix a primitive $(\deg g_i)^{\mathrm{th}}$ root of unity $\zeta_i \in R$ ($\deg g_i < \deg f$ and is therefore prime to $\cha k$ if $\cha k > 0$). Since the splitting field of $g_i$ is $K(t^{1/\deg g_i})$ with Galois group generated by $t^{1/\deg g_i} \mapsto \zeta_i t^{1/\deg g_i}$, every root of $g_i$ has the form $\sum_{m \in \Z_{\geq 0}} \zeta_i^{jm} a_m^i t^{m/n}$ for a unique $j$ such that $1 \leq j \leq \deg g_i$. (This may not be the Newton-Puiseux expansion of the element using the chosen lifts, since we only fixed a set-theoretic section $k \rightarrow R$, but that is okay since we do not need this for what follows.) Since the splitting field of $g_i^\sharp$ is also $k((t^{1/\deg g_i}))$, it follows that every root of $g_i^\sharp$ has the form $\sum_{m \in \Z_{\geq 0}} \overline{\zeta_i^{jm} a_m^i} t^{m/n}$ for a unique $j$ such that $1 \leq j \leq \deg g_i$. Extend the list $\{\alpha_1,\ldots,\alpha_l\}$ to a complete set of roots of $f$, and set $\beta_i = \sum_{m \in \Z_{\geq 0}} \overline{\zeta_i^{jm} a_m^i} t^{m/n}$ if $\alpha_i = \sum_{m \in \Z_{\geq 0}} \zeta_i^{jm} a_m^i t^{m/n}$. Since the $\zeta_i$ are units in $R$, it follows that for any two indices $i,j$, we have $\zeta_i^{jm} a_m^i = 0$ if and only if $\overline{\zeta_i^{jm} a_m^i} = 0$. This implies that  
  \[ \nu(\sum_{m \in \Z_{\geq 0}} \zeta_i^{jm} a_m^i t^{m/n}) = \min (m/n \ | \ \zeta_i^{jm} a_m^i \neq 0 ) = \min (m/n \ | \ \overline{\zeta_i^{jm} a_m^i} \neq 0 ) = \nu (\sum_{m \in \Z_{\geq 0}} \overline{\zeta_i^{jm} a_m^i} t^{m/n}) .\]  Similarly, if $\alpha_i = \sum_{m \in \Z_{\geq 0}} \zeta_i^{jm} a_m^i t^{m/n}$ and $\alpha_{i'} = \sum_{m \in \Z_{\geq 0}} \zeta_{i'}^{jm} a_m^{i'} t^{m/n}$, then $\nu(\alpha_i - \alpha_{i'}) = \min ( m/n \ | \ \zeta_i^{jm} a_m^i \neq \zeta_{i'}^{j'm} a_m^{i'})$. Since $\overline{\zeta} \neq 1$ for any root of unity $\zeta \neq 1$ of order prime to $\cha k$, it follows that $\zeta_i^{jm} a_m^i = \zeta_{i'}^{j'm} a_m^{i'}$ if and only if $\overline{\zeta_i^{jm} a_m^i} = \overline{\zeta_{i'}^{j'm} a_m^{i'}}$ for any choice of indices $m,i,i',j,j'$. This shows $\nu(\alpha_i - \alpha_{i'}) = \nu(\beta_i - \beta_{i'})$ since $\nu(\beta_i - \beta_{i'}) = \min ( m/n \ | \ \overline{\zeta_i^{jm} a_m^i} \neq \overline{\zeta_{i'}^{j'm} a_m^{i'}})$. Finally, we have 
 \[ \Delta_f = \sum_{i \neq i'} \nu(\alpha_i - \alpha_{i'}) = \sum_{i \neq i'} \nu(\beta_i - \beta_{i'}) = \Delta_{f^\sharp} . \]
 \item Since $f$ and $f^\sharp$ have the same degree, the generic fibers of $X_f$ and $X_{f^\sharp}$ have the same $\ell$-adic Euler characteristic ($=4-\deg f$ or $3-\deg f$ depending on whether $\deg f$ is even or odd). It suffices to show that the same is true of the special fibers. Since $\cha k > \deg(f)$ if $\cha k > 0$, it follows that $\delta=0$. The Riemann-Hurwitz formula Lemma~\ref{formula} implies that it is enough to prove the following three things. 
 \begin{itemize}
  \item The special fibers of $Y_f$ and $Y_{f^\sharp}$ are isomorphic.
  \item The order of vanishing of $f$ along any irreducible component of $(Y_f)_s$ is equal to the order of vanishing of $f^\sharp$ along the corresponding irreducible component of $(Y_{f^\sharp})_s$.
  \item There is a bijection between the horizontal components of $\divi (f)$ and those of $\divi (f^\sharp)$ such that the points where these divisors intersect the special fiber also correspond under the above isomorphism, and the multiplicities of $f$ and $f^\sharp$ in the local ring at the point of intersection are also equal.
 \end{itemize} 
 
  Let $Y'_f$ and $Y'_{f^\sharp}$ be the minimal surfaces with a map to $\mathbb{P}^1_R$ and $\mathbb{P}^1_{k[[t]]}$ respectively such that all horizontal components of $\divi(f)$ and $\divi(f^\sharp)$ respectively are regular. The special fibers of $Y'_f$ and $Y'_{f^\sharp}$, and the incidence of the horizontal components with the special fiber are algorithmically determined by the exponents in the Newton-Puiseux expansions of roots of $f$ and $f^\sharp$, and the positions where any two such Newton-Puiseux expansions differ -- this can be seen by using the explicit resolution algorithm as described in \cite[Theorem~3.3.1,Lemma~3.6.1]{Wall} using the continued fraction expansions of the exponents appearing in the expansions.  For instance, the above described algorithm in the case when $f$ is irreducible shows that the dual graph of the minimal resolution of $\divisor(f)$ is a tree that consists of a single horizontal main segment, with finitely vertical segments attached at specified vertices that are determined by the characteristic exponents of the expansion of a root, i.e., the `jump positions' in the l.c.m. of the denominators of the exponents of partial truncations of the Newton-Puiseux expansions (See Definition~\ref{charexponents} for a definition of characteristic exponents, and \cite[p.61, Figure~3.5]{Wall} for a picture of a typical dual graph). (The l.c.m. of the denominators of the exponents is initialized to be $1$, and it increases to $n$ by the time we get to the truncation of a root mod $t^{M/n}$; it stays $n$ thereafter). The number of components in each segment of the dual graph can likewise be determined by suitably normalized `Farey sequences' of rational numbers.

The same algorithm applies in both the equicharacteristic and mixed characteristic cases, once we are guaranteed the existence of Newton-Puiseux series with exponents of bounded denominators. The proof even shows that the resolution is completely determined by the $M$-truncations of the roots of $f$ and $f^\sharp$ respectively, where the integer $M$ is chosen as in part(a) of this Proposition. It is obtained by gluing together the embedded resolutions of the divisors corresponding to the irreducible factors $g_i$ of $f$ appropriately. The gluing data is determined by the positions where the Newton-Puiseux series differ. 
%In \cite{Wall}[Proposition~4.3.8], using the explicit equations that the algorithm produces for the strict transforms, 

The explicit resolution algorithm also show that the multiplicities of the strict transforms of irreducible components $f$ at their points of incidence with the special fiber of the blowup is determined by the exponents appearing in the Newton-Puiseux expansion \cite{Wall}[Proposition~4.3.8]. These multiplicities in turn determine the order of vanishing of $f$ along any component of the exceptional curve by \cite{liu}[Chapter~9, Proposition~2.23]. \cite{conddisc}[Lemma~2.2] can now be used to show that the additional blowups required to separate intersecting odd vertical components to produce $Y_f$ from $Y'_f$, and to produce $Y_{f^\sharp}$ from $Y'_{f^\sharp}$ also coincide. This gives the required isomorphism of special fibers of $Y_f$ and $Y_{f^\sharp}$, preserving the required incidence data. 
\qedhere  
 \end{enumerate}
\end{proof}

\begin{remark}
 Another way to justify these claims is using the theory of Mac Lane valuations as in \cite{oddpaper}. Mac Lane valuations give a way of ``labelling'' the irreducible components that appear in the resolutions $Y_f$ and $Y_{f^\sharp}$ as a valuation on $K(\P^1_K)$. These labels are in terms of certain ``key polynomials $\phi_i$'' and rational numbers $\lambda_i$, and in our case come from the minimal polynomials of truncations of Newton-Puiseux expansions just before a jump position, and the essential exponents at the jump position. (This is explained in \cite[Remark~5.26]{oddpaper}) Since the labels for the components of $Y_f$ and $Y_{f^\sharp}$ can be paired up, we see that the corresponding components are in bijective correspondence (see \cite[Section~5.4]{oddpaper}). The order of vanishing of $f$ along these components can be directly computed from the corresponding Mac Lane labels of the irreducible components, and agree for $f$ and $f^\sharp$. The specialization of horizontal components are also completely determined by the Mac Lane descriptions of these irreducible components (see \cite[Corollary~5.4]{oddpaper}). 
\end{remark}

This Proposition shows that we may assume $R = k[[t]]$ for proving $-\Art(X_f) \leq \Delta_f$ without any loss of generality. {\color{blue}{In the rest of the paper, let $R=k[[t]]$.}}
 
\section{Base case of induction}\label{Sbasecase}
\begin{lemma}\label{RegularHorizontal}
Let $g \colonequals \sum c_i x^i \in R[x]$ be a monic irreducible polynomial with $\deg g \geq 2$. Let $\Gamma \colonequals \divisor(g)$. Assume that $\Gamma$ intersects the special fiber of $\mathbb{P}^1_R$ at the origin. Then $\Gamma$ is regular if and only if $\nu(c_0) = 1$.
\end{lemma}
\begin{proof}
 The curve $\Gamma$ has a unique closed point $P$ corresponding to $x=t=0$. The point $P \in \mathbb{P}^1_R$ corresponds to a maximal ideal $\mathfrak{m}$ in the regular local ring $\O_{\mathbb{P}^1_R, P}$, and $\O_{\Gamma, P} = \O_{\mathbb{P}^1_R, P}/(g)$ . The curve $\Gamma$ is regular if and only if $\O_{\Gamma, P}$ is a regular local ring, which happens if and only if the defining equation $g \notin \mathfrak{m}^2$. Since $\deg g \geq 2$, this happens if and only if $\nu(c_0) = 1$.
\end{proof}

\begin{lemma}\label{evenbase}
 Assume that $f=ug_1g_2\ldots g_l$ where $u \in R$ is a unit, and the $g_i \in R[x]$ are pairwise distinct monic irreducible polynomials of degree $n_i$. Let $\Gamma_i = \divisor (g_i)$ on $\mathbb{P}^1_R$. If $n_i \geq 2$, assume that $g_i(x) = h_i(x+a_i)$ for some Eisenstein polynomial $h_i$ and for some $a_i \in R$. Assume that the $\Gamma_i$ intersect the special fiber of $\mathbb{P}^1_R$ at distinct points. Let $Y^f$ and $X^f$ be as in Definition~\ref{models}. Then $Y^f = \mathbb{P}^1_R$ and  $X^f$ is regular, and,
  \[ -\Art(X^f/S) = \nu(\Delta_f) = \sum_{i=1}^l (n_i-1). \]
\end{lemma}
\begin{proof}
 Since translation by $a_i$ is an isomorphism of $\mathbb{P}^1_R$, Lemma~\ref{RegularHorizontal} tells us that all the $\Gamma_i$ are regular. Since we also assumed that the $\Gamma_i$ intersect the special fiber of $\mathbb{P}^1_R$ at distinct points, it follows that $Y^f = \P^1_R, B = \divisor (f)_{\mathrm{odd}}$ and $X^f$ is regular by Lemma~\ref{useful}.  Since we assumed that the $g_i$ are pairwise distinct irreducible polynomials, it follows that $B_{\overline{\eta}} \subset \mathbb{P}^1_{\overline{K}}$ is a sum of $\deg f$ or $\deg f+1$ distinct closed points depending on whether $\deg f$ is even or odd. Similarly, our assumption that the $\Gamma_i$ intersect the special fiber at distinct points implies that $B_s \subset \mathbb{P}^1_k$ is a sum of $l$ or $l+1$ distinct closed points depending on whether $\deg f$ is even or odd. Since $Y^f_s = \P^1_k$ and $Y^f_{\overline{\eta}} = \mathbb{P}^1_{\overline{K}}$ and the $\ell$-adic Euler characteristic of a closed point over an algebraically closed field 
is $1$, it follows that
 \begin{align*}
  -\Art (X^f/S) &= 2(\chi(Y^f_s)-\chi(Y^f_{\overline{\eta}}))-(\chi(B_s)-\chi(B_{\overline{\eta}})) \\
  &= 2(2-2)-(l-\deg f) \\
  &= \sum_{i=1}^l (n_i-1).
 \end{align*}
Assume that $n_i \geq 2$. Since $\cha k > \deg f \geq n_i$ and $k$ is algebraically closed, any root of $g_i$ generates a tame totally ramified Kummer extension of $K$ and $\nu(\Delta_{g_i}) = n_i-1$ by \cite[Chapter~IV, \S~1, Proposition~4]{serrelocal}. Since the $\Gamma_i$ intersect the special fiber $Y^f_s$ at distinct points, it follows that
\[ \nu(\Delta_f) = \sum_{i=1}^l \nu(\Delta_{g_i}) = \sum_{i=1}^l (n_i-1) .\qedhere \]
\end{proof}

\section{The inductive step: replacement polynomials}\label{induction}
In Lemma~\ref{badpoints} and Corollary~\ref{badpointscorollary}, we first prove that the conditions in Lemma~\ref{evenbase} are in fact necessary and sufficient for regularity of the the standard Weierstrass model.
In Section~\ref{outline}, we describe an inductive proof strategy to prove $-\Art(X^f) \leq \nu(\Delta_f)$, where $X^f$ is the regular model from Lemma~\ref{useful}.
\subsection{Setup}\label{setup}
Recall that $R$ can be assumed to be the ring $k[[t]]$. Let $f=ut^bg_1g_2\ldots g_l$ be the prime factorization of the squarefree polynomial $f$ in $R[x]$, where $u$ is a unit, $b \in \{0,1\}$ and the $g_i$ are pairwise distinct monic irreducible polynomials in $R[x]$. Let $n_i \colonequals \deg (g_i)$. 
Let $\Gamma_i$ be the irreducible horizontal divisor corresponding to $g_i$ in $Y_0 \colonequals \mathbb{P}^1_R$. Let $P_i$ be the closed point of $\mathbb{P}^1_k$ where $\Gamma_i$ intersects the special fiber of $Y_0$, and let $(x-a_{P_i}) \subset k[x]$ be the corresponding maximal ideal. Let $\lambda_i \colonequals \nu(g_i(a_{P_i}))$. 

 Let $\infty$ be the closed point at infinity on the special fiber of $Y_0$. Let $\Gamma_{\infty}$ be the scheme-theoretic closure in $\mathbb{P}^1_R$ of $\infty \in \mathbb{P}^1_K$. Let \[ A \colonequals \begin{cases} \{ P_1,P_2,\ldots,P_l \} \quad \quad \quad \quad \quad \, \text{if}\  \deg f \ \text{is even, and,} \\ \{ P_1,P_2,\ldots,P_l \} \cup \{\infty\} \quad \quad \text{if}\  \deg f \ \text{is odd.} \end{cases} \] 
 For every $P$ in $A \setminus \{ \infty \}$, let 
\begin{align*}
 C_P &\colonequals \{ g_i \ | \ 1 \leq i \leq l, \ P_i = P \}, \\
 C_P^{< 1} &\colonequals \{ g_i \in C_P \ | \ \lambda_i/n_i < 1 \},  \ \text{and}, \\
 C_P^{\geq 1} &\colonequals \{ g_i \in C_P \ | \ \lambda_i/n_i \geq 1 \}.  
\end{align*}
Note that $C_P^{< 1}$ consists precisely of those irreducible factors specializing to $P$ that have roots of valuation $<1$ after we move $P$ to $x=t=0$ by a change of variables, and likewise $C_P^{\geq 1}$ are those factors whose roots have valuation $\geq 1$ after a change of variables.

\begin{defn}\label{D:weights}[Weights] Let $P$ in $A$. Define
\[ \widetilde{\wt_P} \colonequals \begin{cases} \sum_{g_i \in C_P} \min (n_i,\lambda_i) =  \sum_{g_i \in C_P^{<1}} \lambda_i + \sum_{g_i \in C_P^{\geq 1}} n_i \quad\quad &\textup{if } P \neq \infty, \\
\textup{parity of } \deg(f) \quad\quad &\textup{if } P = \infty.\end{cases}  \]
Define the {\textup{\textsf{\color{blue}{weight $\wt_P$}}}} of $P$ to be
\[ \wt_P = b+ \widetilde{\wt_P}.\] 
\end{defn}

We also define the notion of good weight $3$ points appearing in the statement of Theorem~\ref{Cexbal}.
\begin{defn}\label{D:goodtype3}
 $P \in A$ is a {\textup{\textsf{\color{blue}good weight $3$ point}}} if $b=0,\wt_P = 3$ and we are in one of the following cases:
 \begin{enumerate}
  \item the irreducible polynomials in $C_P$ specialize to at least two distinct points after a single blow up.
  \item $C_P$ consists of two irreducible polynomials $f_1,f_2$ that specialize to the same point in the exceptional curve $E_P$ such that $\min(n_1,\lambda_1) = 1$ and $(n_2, \lambda_2) \in \{ (2,3),(3,2) \}$.
  \item $C_P$ consists of a single irreducible polynomial $f_1$ such that $(n_1,\lambda_1) \in \{ (3,4),(4,3),(3,5),(5,3) \}$.
 \end{enumerate}
\end{defn}

\begin{lemma}\label{L:wtsmult}
 For any $P$ in $A \setminus \{\infty\}$, the multiplicity $\mu_P(f)$ of $f$ in the local ring of $\mathbb{P}^1_R$ at $P$ is $\wt_P$.
\end{lemma}
\begin{proof}
 Since $\mu_P(\Gamma_i) = \lambda_i$ if $g_i \in C_P^{<1}$ and $\mu_P(\Gamma_i) = n_i$ if $g_i \in C_P^{\geq 1}$, it follows that the multiplicity $\mu_P(f)$ of $\divisor (f)$ at $P$ is given by
\[ \mu_P(f) = \mu_P(ut^bg_1g_2\ldots g_l) = b \mu_P(t) + \sum \mu_P(g_i)= b+ \left( \sum_{g_i \in C_P^{<1}} \lambda_i \right) + \left( \sum_{g_i \in C_P^{\geq 1}} n_i \right) .\qedhere\]
\end{proof}

Let
\[ A_{\mathrm{bad}} \colonequals \left\{ P \in A \ \bigg| \ \wt_P \geq 2 \right\}.  \]
Let $X_0$ be the normalization of $Y_0$ in $K(x)(\sqrt{f})$. Let $\pi_0$ denote the associated finite map $X_0 \rightarrow Y_0$.  Let $B \subset Y_0$ be the branch locus of $\pi_0$. 

\begin{lemma}\label{badpoints}
Let $X_0^{\mathrm{sing}}, B^{\mathrm{sing}}$ be the (possibly empty) sets of nonregular points of $X_0$ and $B$ respectively. Then $A_{\mathrm{bad}} = B^{\mathrm{sing}} = \pi_0(X_0^{\mathrm{sing}}) $.
\end{lemma}
\begin{proof}
As a Weil divisor $B$ is the sum of the odd components of $\divisor (f)$, and therefore it follows that if $\deg f$ is odd, then $B = b \cdot \divisor(t) + \sum \Gamma_i + \Gamma_{\infty}$ and if $\deg f$ is even, then $B = b \cdot \divisor(t) + \sum \Gamma_i$.

We will first show that $B^{\mathrm{sing}} = A_{\mathrm{bad}}$. Note $\infty \in A_{\mathrm{bad}}$ precisely when $b=1$ and $\deg(f)$ is odd. If $b=1$ and $\deg f$ is odd, then $\infty$ lies on two different irreducible components of $B$, namely $\divisor(t)$ and $\Gamma_{\infty}$. In this case \cite[p.~129, Chapter~4, Corollary~2.12]{liu}  implies that $\infty$ is a nonregular point of $B$. Since $\divisor(t)$ and $\Gamma_{\infty}$ are both regular at $\infty$, and $\infty \notin \Gamma_i$ for every $i$, it follows that $\infty$ is a regular point of $B$ in all other cases. Since $B$ is cut out by $f$ at $P$ when $P \neq \infty$, \cite[p.~129, Chapter~4, Corollary~2.12]{liu} implies that $P \in B^{\mathrm{sing}}$ if and only if $f \in \mathfrak{m}_{Y_0,P}^2$, i.e, if and only if the multiplicity $\mu_P(f) \geq 2$. Lemma~\ref{L:wtsmult} now completes the proof of $A_{\mathrm{bad}} = B^{\mathrm{sing}}$.

Let $\tilde{P} \in X_0$ and $\pi_0(\tilde{P}) = P \in Y_0$. We will now show that $\tilde{P} \in X_0^{\mathrm{sing}}$ if and only if $P \in B^{\mathrm{sing}}$. Since $Y_0$ is regular and $\pi_0 \colon X_0 \setminus \pi_0^{-1}(B) \rightarrow Y_0 \setminus B$ is \'{e}tale, it follows that $X_0 \setminus \pi_0^{-1}(B)$ is regular by \cite[p.~49, Proposition~9]{blr}. If $P \in B$ is regular, then $f \notin \mathfrak{m}_{Y_0,P}^2$ by \cite[p.~129, Chapter~4, Corollary~2.12]{liu}, which in turn implies that $b=0$ and that $P$ lies on a unique irreducible component of $B$. By \cite[p.~129, Chapter~4, Corollary~2.15]{liu}, it follows that $f$ is part of a system of parameters for $Y_0$ at $P$, i.e, there exists another rational function $g$ such that $f$ and $g$ generate the maximal ideal $\mathfrak{m}_{Y_0,P}$. Since $\mathcal{O}_{X_0,\tilde{P}} = \mathcal{O}_{Y_0,P}[y]/(y^2-f)$, it follows that the maximal ideal at $\tilde{P} \in R$ is generated by $y$ and $g$, and is therefore also regular. If $P \in B^{\mathrm{
sing}}$, then $f \in \mathfrak{m}_{Y_0,P}^2$, which in turn implies that $y^2-f \in \mathfrak{m}_{X_0,\tilde{P}}^2$. Since $\mathcal{O}_{X_0,\tilde{P}} = \mathcal{O}_{Y_0,P}[y]/(y^2-f)$, it follows that $\dim (\mathfrak{m}_{X_0,\tilde{P}}/\mathfrak{m}_{X_0,\tilde{P}}^2) \geq \dim (\mathfrak{m}_{Y_0,P}/\mathfrak{m}_{Y_0,P}^2) + 1 = 3$ and therefore $\tilde{P}$ is not a regular point of $X_0$. \qedhere
\end{proof}

\begin{corollary}\label{badpointscorollary}
The scheme $X_0$ is regular if and only if $f$ satisfies the hypotheses of Lemma~\ref{evenbase}.
\end{corollary}
\begin{proof}
The set $A_{\mathrm{bad}}$ is empty if and only if $f$ satisfies the hypotheses of Lemma~\ref{evenbase}.
\end{proof}

Let $Y^f \rightarrow \cdots \rightarrow Y_1 \rightarrow Y_0$ be the good embedded resolution of the pair $(\mathbb{P}^1_R,\divisor(f))$ and let $X^f$ be the normalization of $Y^f$ in $K(Y_0)(\sqrt{f})$. By Definition~\ref{defgoodembres} $Y_1$ is the blowup of $Y_0$ along the closed subscheme $A_{\mathrm{bad}} = B^{\mathrm{sing}}$ of $B$. For $P \in A_{\mathrm{bad}}$, let $E_P$ be the exceptional curve for the blowup $Y_1 \rightarrow Y_0$ at $P$.
\begin{corollary}\label{fmultexc}
 The order of vanishing $\nu_{E_P}(f)$ of $f$ along $E_P$ is $\wt_P$ if $P \neq \infty$ and even if $P = \infty$.
\end{corollary}
\begin{proof}
First assume $P = \infty$. By Lemma~\ref{badpoints}, the exceptional curve $E_P$ is defined if and only if $b=1$ and $\deg(f)$ is odd. In this case, the corollary follows since $\infty \notin \Gamma_i$, and $\mu_{\infty}(\divisor (t)) = \mu_{\infty}(\Gamma_{\infty}) = 1$. Now assume $P \neq infty$. The order of vanishing of $f$ along $E_P$ equals the multiplicity $\mu_P(f)$ of $\divisor (f)$ on $\mathbb{P}^1_R$ at $P$. The corollary follows from Lemma~\ref{L:wtsmult}.   
\end{proof}

\subsection{Outline of inductive proof strategy}\label{outline}
We now outline an inductive strategy for proving $-\Art (X^{f}/S) \leq \nu(\Delta_{f})$; we shall henceforth refer to this inequality as the {\textup{\textsf{\color{blue}conductor-discriminant inequality for $f$}}}. We will prove the conductor-discriminant inequality for $f$ by induction on the ordered pair of integers $(\deg(f),\nu(\tilde{\Delta}_f))$. The case $A_{\mathrm{bad}} = \emptyset$ is Lemma~\ref{evenbase}. 

If $A_{\mathrm{bad}} \neq \emptyset$, for every $P \in A_{\mathrm{bad}} \setminus \{\infty\}$, we will construct a new pair of squarefree polynomials $f_P^{\infty},f_P^{\neq \infty}$ in $R[x]$  (see Definition~\ref{defnreppoly}). We call $\{ f_P^{\infty} \ | \ P \in A_{\mathrm{bad}} \setminus \{\infty\}, \deg(f_P^{\infty})\geq 1  \} \cup \{ f_P^{\neq \infty} \ | \ P \in A_{\mathrm{bad}} \setminus \{\infty\}, \deg(f_P^{\neq \infty})\geq 1 \}$ the collection of {\textup{\textsf{\color{blue}replacement polynomials for $f$}}}. These replacement polynomials come from the equations of the strict transform of $\divisor(f)$ after one blow up at $P$ (see section~\ref{stricttransformeqn}). 
In Section~\ref{finalproof}, we prove the {\textup{\textsf{\color{blue}key inductive inequality}}}
\begin{thm}\label{replacementchange} 
\begin{enumerate}[\upshape (a)] \hfill
 \item 
 \begin{multline*}-\Art (X^f/S)- \left( \sum_{\substack{P \in A_{\mathrm{bad}} \setminus \{\infty\} \\ \deg(f_P^{\infty}) \geq 1 }}-\Art (X^{f_P^{\infty}}/S) +\sum_{\substack{P \in A_{\mathrm{bad}} \setminus \{\infty\} \\ \deg(f_P^{\neq \infty}) \geq 1 }}-\Art (X^{f_P^{\neq \infty}}/S) \right) \leq \\ \nu(\Delta_f)-\sum_{\substack{P \in A_{\mathrm{bad}} \setminus \{\infty\} \\ \deg(f_P^{\infty}) \geq 1 }} \nu(\Delta_{f_P^{\infty}}) -\sum_{\substack{P \in A_{\mathrm{bad}} \setminus \{\infty\} \\ \deg(f_P^{\neq \infty}) \geq 1 }} \nu(\Delta_{f_P^{\neq \infty}}).\end{multline*} 
 \item Equality holds if and only if we have $\wt_P \in \{2,3\}$ for each $P \in A_{\mathrm{bad}}$.
 \item The left hand side of the inequality in part(a) is nonnegative. The right hand side is strictly positive except when $b=0, \wt_P = 3$ for every $P \in A_{\mathrm{bad}}$.
\end{enumerate}

\end{thm}
We will show in Corollary~\ref{Cdiscind} that either the degree or the discriminant decreases after at most two such replacement steps. The induction hypothesis then gives the conductor-discriminant inequality for the replacement polynomials of $f$. Adding all these inequalities to the one in Theorem~\ref{replacementchange} then proves the conductor-discriminant inequality for $f$. In Section~\ref{finalproof}, we also prove Theorem~\ref{Cexbal}, by analyzing when the condition for equality in Theorem~\ref{replacementchange} holds for $f$ and for all its replacement polynomials.

The rest of this section is devoted to defining the main objects for the inductive step, the replacement polynomials. 

\subsection{Replacement polynomials $f_P^{\infty}$ and $f_P^{\neq \infty}$ and equations for the strict transform of $\divisor(g_i)$}\label{stricttransformeqn}\label{defreplacepoly} 
In this section, we prove Lemma~\ref{eqforstricttransform} which gives an explicit equation for the strict transform of the irreducible components of $\divisor(f)$ passing through a given $P \in A_{\mathrm{bad}}$ after one blowup. We will use these equations along with the Weierstrass preparation theorem and a change of variables to define the replacement polynomials mentioned in the outline above (Definition~\ref{defnreppoly}).

\subsubsection{Notation} 
Let $f,g_i,b,a_{P_i},\lambda_i,n_i, Y_0, X_0, C_P, C_P^{<1}, C_P^{\geq 1}$ be as in subsection~\ref{setup}. Fix $P \in A_{\mathrm{bad}} \setminus \{\infty\}$. Let $Y_1 \rightarrow Y_0$ be the blowup of $Y_0 \colonequals \mathbb{P}^1_R$ at $A_{\mathrm{bad}}$, and let $E$ be the exceptional curve for the blowup at $P$, and let $H$ be the strict transform of the divisor of zeroes of $g_i$. Let $Q \in (Y_1)_s(k)$ be the point where $E$ meets the rest of $(Y_1)_s$. By replacing $g_i(x)$ by $\tilde{g_i}(x) \colonequals g_i(x+a_{P_i})$, we may assume that the point $P$ corresponds to the origin $x=t=0$ on the special fiber of $Y_0$.
Since $\tilde{g_i}$ is irreducible, all its roots in $\overline{K}$ have valuation $\lambda_i/n_i$. Let $P_i \in (Y_1)_s(k)$ be the intersection $P_i \colonequals H \cap E$. We omit the proof of the following lemma. 
\begin{lemma}\label{specialization}
If $\lambda_i/n_i \geq 1$ (or equivalently, if $i \in C_P^{\geq 1}$), then $P_i \neq Q$ and the ideal $\mathfrak{m}_{P_i}$ of functions on $Y_1$ vanishing at $P_i$ is generated by $\tfrac{x}{t}-c$ and $t$ for some $c \in R$. If $\lambda_i/n_i < 1$ (or equivalently, if $i \in C_P^{< 1}$), then $P_i = Q$, and the ideal $\mathfrak{m}_Q$ is generated by $\tfrac{t}{x}$ and $x$. 
\end{lemma}

\begin{defn}
Define
\[ \tilde{g}_i^{\mathrm{new}} \colonequals \begin{cases} \tilde{g}_i/x^{\lambda_i} \quad \quad \quad \quad \text{if} \ i \in C_P^{< 1} \\ \tilde{g}_i/t^{n_i} \quad \quad \quad  \quad \text{if}\ i \in C_P^{\geq 1}.  \end{cases} \]
\end{defn}

\begin{lemma}\label{eqforstricttransform}
 The strict transform $H$ of the divisor of zeroes of $\tilde{g_i}$ is cut out by $\tilde{g}_i^{\mathrm{new}}$ in the local ring $\mathcal{O}_{P_i}$ and $\nu_E(\tilde{g}_i^{\mathrm{new}}) = 0$. If $i \in C_P^{\geq 1}$, then $\tilde{g}_i/x^{n_i}$ is a unit in the local ring $\mathcal{O}_Q$.
\end{lemma}
\begin{proof} If we let $\mu_P(\tilde{g_i})$ denote the multiplicity at $P$ of $\divisor(\tilde{g_i}) \subset \mathbb{P}^1_R$, then we have $\divisor(\tilde{g_i}) = \mu_P(\tilde{g_i})E+H$ in $\mathrm{Div}(Y_1)$. We now compute $\mu_P(\tilde{g_i})$. Let $\tilde{g_i}(x) \colonequals \sum_{j=0}^{n_i-1} c_j x^j + x^{n_i}$. Since $\tilde{g_i}$ is irreducible and all its roots in $\overline{K}$ have valuation $\lambda_i/n_i$, a Newton polygon argument shows that $\nu(c_j) \geq (n_i-j)\lambda_i/n_i$ for all $j$. Since $\mathfrak{m}_P$ is generated by $x$ and $t$, these bounds show that when $i \in C_P^{<1}$, i.e., when $\lambda_i/n_i < 1$, we have $c_0 \in \mathfrak{m}_P^{\lambda_i}$ and all the other terms of $\tilde{g_i}$ are in $\mathfrak{m}_P^{\lambda_i+1}$ and therefore $\mu_P(\tilde{g_i}) = \lambda_i$. Similarly, when $i \in C_P^{\geq 1}$, i.e., when $\lambda_i/n_i \geq 1$, we have $x^{n_i} \in \mathfrak{m}_P^{n_i}$ and all the other terms of $\tilde{g_i}$ are in $\mathfrak{m}_P^{n_i+1}$ and therefore $\mu_P(\tilde{g_i}) = n_i$. 
 
Finally in the local ring $\mathcal{O}_{P_i}$, we have $\divisor(x) = E$ when $i \in C_P^{<1}$ and $\divisor(t) = E$ when $i \in C_P^{\geq 1}$. When $i \in C_P^{\geq 1}$, since $t/x \in \mathcal{O}_Q$, we also have $(\tilde{g_i}-x^{n_i})/x^{n_i}$  in $\mathfrak{m}_Q$, which in turn implies that $\tilde{g_i}/x^{n_i} \in \mathcal{O}_Q$ is a unit. \qedhere
\end{proof}

We now obtain the replacement polynomials $f_P^{\infty}$ and $f_{P}^{\neq \infty}$ by doing a natural change of variables on the defining equation $\tilde{g_i}^{\mathrm{new}}$ of the strict transform of $H$. The idea behind the change of variables is to replace the triple $\mathbb{P}^1_R$ with coordinate $x$ and the divisor of zeroes of $\tilde{g_i}$, with $\mathbb{P}^1_R$ with coordinate $x/t$ (and $t/x$ respectively) and the divisor of zeroes of $\tilde{g_i}^{\mathrm{new}}$ rewritten in new coordinates when $i \in C_P^{\geq 1}$ (and when $i \in C_P^{<1}$ respectively). Recall that $Q$ is the point where the exceptional curve at $P \in Y_0$ for the blowup $Y_1 \rightarrow Y_0$ meets the rest of the special fiber of $Y_1$. The reason we replace $x$ by $t$ when $i \in C_P^{<1}$ is that $x = 0$ cuts out the exceptional curve $E_P$ in the local ring $\mathcal{O}_Q$, just as $t=0$ cuts out the special fiber $(Y_0)_s$ in the local ring $\mathcal{O}_P$.

Recall in Section~\ref{embedding}, we proved we may assume $R=k[[t]]$ for our purposes. In this case, since $x$ and $\tfrac{t}{x}$ generate $\mathfrak{m}_Q$, we have a canonical isomorphism of the completed local ring $\hat{\mathcal{O}_{Q}} \cong k[[x,\tfrac{t}{x}]]$ by the Cohen structure theorem. For $i \in C_P^{<1}$, view the germ of the function $\tilde{g_i} \in \mathcal{O}_{Q} \hookrightarrow \hat{\mathcal{O}_{Q}}$ as a bivariate power series via this isomorphism.
\begin{defn}\label{defnreppoly}
If $i \in C_P^{<1}$, let $h_i^\Diamond \in k[[t,x]]$ be the power series obtained by making the change of variables  $x \mapsto t$ and $\tfrac{t}{x} \mapsto x$ in the power series $\tilde{g_i}(x,\tfrac{t}{x})/x^{\lambda_i} \in \hat{\mathcal{O}_{Q}} \cong k[[x,\tfrac{t}{x}]]$, i.e., $h_i^\Diamond = \tilde{g_i}(t,x)/t^{\lambda_i}$. Use the Weierstrass preparation theorem to write $h_i^\Diamond = u^\Diamond h_i$ for some unit $u^\Diamond \in k[[t,x]]$ and monic polynomial $h_i \in k[[t]][x]$. If $i \in C_P^{\geq 1}$, let $h_i(x) \colonequals \tilde{g_i}^{\mathrm{new}}(tx).$  Let
\[ b_P \colonequals \begin{cases} 0 \quad \textup{if $v_{E_P}(f)$ is even} \\ 1 \quad \textup{if $v_{E_P}(f)$ is odd.} \end{cases} \]
 \begin{align*} f_P^{\infty}(x) &\colonequals t^{b_P} x^b \prod_{g_i \in C_P^{< 1}} h_i(x), \ \text{and}, \\
f_P^{\neq \infty}(x) &\colonequals t^{b_P} \prod_{g_i \in C_P^{\geq 1}} h_i(x).\end{align*}
The {\textup{\textsf{\color{blue}replacement polynomial for $g_i$}}} is defined to be $h_i$, and the {\textup{\textsf{\color{blue}replacement polynomials for $f$}}} is the set $\{ f_P^{\infty} \ | \ P \in A_{\mathrm{bad}} \setminus \{\infty\},\ \deg f_P^{\infty} \geq 1 \} \cup \{ f_P^{\neq \infty} \ | \ P \in A_{\mathrm{bad}} \setminus \{\infty\}, \deg f_P^{\neq \infty} \geq 1 \}$. 
\end{defn}

\begin{lemma}\label{degrep}
 Fix $i \in C_P^{<1}$ with $\deg h_i \geq 1$. Recall $\lambda_i = \nu_K(\tilde{g_i}(0))$. Then $h_i$ is irreducible and $\deg h_i = \lambda_i$.
\end{lemma}
\begin{proof}
 Let $\tilde{g_i} = \sum_{i=0}^{n_i-1} c_i x^i+x^{n_i}$. In the proof of Lemma~\ref{eqforstricttransform}, we showed that $c_0 \in \mathfrak{m}_Q^{\lambda_i}$ and $c_i \in \mathfrak{m}_Q^{\lambda_i+1}$ for $i > 0$. Rewriting $\tilde{g_i}^{\mathrm{new}} = \tilde{g_i}/x^{\lambda_i} \in \hat{\mathcal{O}_Q} \cong k[[x]][[t/x]]$ as $\sum_{i=0}^\infty \tilde{c_i} (t/x)^i$ for $\tilde{c_i} \in k[[x]]$ and using $c_0 \in \mathfrak{m}_Q^{\lambda_i}$ and $c_i \in \mathfrak{m}_Q^{\lambda_i+1}$ for $i > 0$, we see that $\tilde{c_{\lambda_i}} \in k[[x]] \setminus xk[[x]]$ and $\tilde{c_i} \in xk[[x]]$ for $i < \lambda_i$. The Weierstrass preparation theorem then shows that if we write $\tilde{g_i}/x^\lambda_i = \tilde{u^{\Diamond}} \tilde{h_i}$ for a unit $\tilde{u^{\Diamond}} \in \hat{\mathcal{O}_Q}^*$ and monic polynomial $\tilde{h_i} \in k[[x]][t/x]$, then $\deg_{t/x} \tilde{h_i} = \lambda_i$. Since $h_i$ is obtained from $\tilde{h_i}$ by the change of variables $x \mapsto t, t/x \mapsto x$, we get $\deg_x h_i = \deg_{t/x} \tilde{h_i} = \lambda_i$.
 
 The $h_i$ are irreducible for each $i \in C_P^{<1}$ since Lemma~\ref{eqforstricttransform} shows that $h_i$ up to a change of variables equals $\tilde{g_i}^{\mathrm{new}}$ and $\tilde{g_i}^{\mathrm{new}}$ cuts out the irreducible divisor corresponding to the strict transform of $\divisor(g_i)$ after we blow up $P$.
\end{proof}

Let $(\widetilde{n_i},\widetilde{\lambda_i})$ be the pair of integers associated to the replacement polynomial $h_i$ the same way $(n_i,\lambda_i)$ is associated to $f_i$.
\begin{remark}\label{rootscoefficients}
 Let $i \in C_P^{\geq 1}$. Then the polynomials $h_i$ and $\tilde{g_i}$ are monic irreducible polynomials of the same degree, and furthermore, division by $t$ gives a bijection from the roots of $h_i$ to the roots of $\tilde{g_i}$. In particular, $(\widetilde{n_i},\widetilde{\lambda_i}) = (n_i,\lambda_i-n_i)$ and therefore $\min(\widetilde{n_i},\widetilde{\lambda_i}) \leq (n_i,\lambda_i)$.
\end{remark}

\begin{remark}\label{Rcomprep} When $i \in C_P^{<1}$, the relation between the roots of $h_i$ and the roots of $\tilde{g_i}$ is more complicated than in Remark~\ref{RAbaddeg}; for instance in Lemma~\ref{degrep} we proved that $\deg h_i = \lambda_i < n_i = \deg \tilde{g_i}$.  However, Theorem~\ref{expofreplacement} lets us relate certain coefficients and exponents of the Newton-Puiseux expansions of the roots of $h_i$ to those of $\tilde{g_i}$, which in turn lets us compute $\Delta_{\tilde{g_i}} - \Delta_{h_i}$. We will show in Corollary~\ref{Cnewpair} that $(\widetilde{n_i},\widetilde{\lambda_i}) = (\lambda_i,n_i-\lambda_i)$ and therefore $\min(\widetilde{n_i},\widetilde{\lambda_i}) \leq (n_i,\lambda_i)$. \end{remark}

\begin{remark}\label{RAbaddeg}
 From the definitions of $A_{\mathrm{bad}}$ (\S~\ref{setup}) and the replacement polynomials $f_P^{\infty}$ and $f_P^{\neq \infty}$ (Definition~\ref{defnreppoly}), Lemma~\ref{degrep} and Remark~\ref{rootscoefficients}, it follows that if $P \in A_{\mathrm{bad}} \setminus \{\infty\}$, then we cannot have $\deg(f_P^{\infty}) = \deg(f_P^{\neq \infty}) = 0$.
\end{remark}

\begin{lemma}\label{repsqfree}
 The replacement polynomials are squarefree.
\end{lemma}
\begin{proof}
Fix $P \in A_{\mathrm{bad}} \setminus \{\infty\}$, and let $Q$ be the point where the exceptional curve $E_P$ for the blowup of $\P^1_R$ at $P$ meets the strict transform of the special fiber of $\P^1_R$. If $i,j \in C_P^{\geq 1}$ and $i \neq j$, then $h_i \neq h_j$. Combined with the previous remark, this proves $f_P^{\neq \infty}$ is squarefree. 

A polynomial $g \in R[x]$ is squarefree if and only if $\divisor(g) = \sum \Gamma_i$ for pairwise distinct irreducible Weil divisors $\Gamma_i$. In $\O_Q$, by Lemma~\ref{eqforstricttransform} $\tilde{g_i}/x^{\lambda_i}$ cuts out the strict transform of $\tilde{g_i}$ after the blowup at $P$ for every $i \in C_P^{<1}$, the function $t/x$ cuts out the strict transform of the special fiber of $\mathbb{P}^1_R$, and the function $x$ cuts out $E_P$. It follows that $\{ \divisor(t/x), \divisor(x) \} \cup \{ \divisor(\tilde{g_i}/x^{\lambda_i}) \ | \ i \in C_P^{<1} \}$ is a collection of pairwise distinct irreducible Weil divisors in $\Spec \O_Q$. Up to the relabelling $x \mapsto t$ and $t/x \mapsto x$, this shows $\divisor(f_P^{\infty})$ is a sum of pairwise distinct irreducible Weil divisors, and therefore $f_P^{\infty}$ is squarefree. \qedhere
\end{proof}

\section{Computing change in conductor}\label{secconductorchange}
In this section, we compute the left hand side of the key inductive inequality Theorem~\ref{replacementchange}. The main idea is to relate the good embedded resolutions of the replacement polynomials to that of $f$ (Lemma~\ref{treereplacethm}) and use the additivity of the $\ell$-adic Euler characteristic combined with the Riemann-Hurwitz formula (Corollary~\ref{diffchiY} and Theorem~\ref{condchange}).
\subsection{Relating good embedded resolutions and branch loci of $f$ and its replacement polynomials}
We continue to use the notation from Section~\ref{setup}, Definition~\ref{defgoodembres} and Definition~\ref{defnreppoly}.

\begin{defn}\label{Dparitybit}
 Define the {\textup{\textsf{\color{blue}parity integer $d$}}} to be $0$ or $1$ depending on whether $\deg (f)$ is even or odd. For each $P \in A_{\mathrm{bad}} \setminus \{\infty\}$, when $\deg (f_P^{\infty}) \geq 1$ 
 define {\textup{\textsf{\color{blue}$d_P^{\mathrm{nod}}$}}} to be $0$ or $1$ depending on whether $\deg (f_P^{\infty})$ is even or odd. Similarly when $\deg (f_P^{\neq \infty}) \geq 1$ define {\textup{\textsf{\color{blue}$d_P^{\mathrm{sm}}$}}} using the parity of $\deg (f_{P}^{\neq \infty})$. 
\end{defn}

\begin{lemma}\label{Lparitydegree}
  The closed point $\infty$ is in $A$ precisely when $d=1$, and $\infty \in A_{\mathrm{bad}}$ when $b=d=1$ (or equivalently, when $bd=1$).
\end{lemma}
\begin{proof}
 This follows from the definitions of the sets $A$ and $A_{\mathrm{bad}}$ in Section~\ref{setup}.
\end{proof}

\begin{lemma}\label{affine}
 Let $f=u t^b g_1 \ldots g_l \in R[x]$ be the irreducible factorization of a squarefree polynomial, and let $Y$ and $Z$ be the good embedded resolutions of the pairs $(\P^1_R,\divisor(f))$ and $(\mathbb{A}^1_R,\divisor(f))$ respectively and let $B,B^\circ$ be $\divisor(f)_{\mathrm{odd}}$ on $Y$ and $Z$ respectively. Let $d$ be as in Definition~\ref{Dparitybit}. Then
 \[ \chi(Y_s)-\chi(Z_s) = 1+bd = \begin{cases} 2 \quad \quad \textup{if } b=1 \ \textup{and } d=1 \\ 
 1 \quad \quad  \textup{otherwise}. \end{cases} \]
 \[ \chi(B_s) - \chi(B^{\circ}_s) = b+d = \begin{cases} 2 \quad \quad \textup{if } b=1 \ \textup{and } d=1\\
 0 \quad \quad \textup{if } b=0 \ \textup{and } d=0\\
 1 \quad \quad \textup{otherwise}. \end{cases} \]
\end{lemma}
\begin{proof}
 From the definition of good embedded resolutions, it is clear that $Z_s \subset Y_s$ and $B^\circ \subset B$ and to analyze the complements, we have to understand the behaviour of $f$ at the closed point at $\infty$ on the special fiber of $\P^1_R$. Definition~\ref{defgoodembres} and Lemma~\ref{badpoints} imply that the blowup $\pi \colon Y \rightarrow \P^1_R$ is not an isomorphism in a neighbourhood of $\infty \in \P^1_k \subset \P^1_R$ if and only if $b=1$ and $\deg f$ is odd, and in this case let $E_{\infty}$ be the exceptional curve for the blowup at $\infty$, and let $Q$ be the point where it meets the strict transform of the special fiber of $\P^1_R$. Let $\Gamma_{\infty}$ be the scheme-theoretic closure in $Y$ of the point at infinity on the generic fiber $\P^1_K$, let $\infty = (\Gamma_{\infty})_s$ and let $\Gamma$ be the strict transform of the special fiber of $\P^1_R$ in $Y$. Then $\Gamma_\infty \subset B$ if and only if $\deg f$ is odd, and $\Gamma \subset B$ if and only if $b=1$. We consider four cases based on the parity of $b$ and $\deg f$.
 
 If $b=1$ and $\deg f$ is odd, then $\nu_{E_\infty}(f) = \mu_{\infty}(\divisor(f)) = 0$ and therefore $E_{\infty} \cap B = \{Q,\infty\}$ and $\divisor(f)_{\mathrm{odd}}$ is regular at these points. Therefore the map from $Y$ to this blowup is an isomorphism in a neighbourhood of $E_{\infty}$ (i.e. points of $E_{\infty}$ are not blown up any further on passing to the good embedded resolution $Y$). In this case, we have $Y_s \setminus Z_s = E_\infty \cong \mathbb{P}^1_k$ and $B_s \setminus B^\circ_s = \{Q,\infty\}$. Since $k$ is algebraically closed, it follows that $\chi(\P^1_k) = 2$ and $\chi(k-\text{rational point}) = 1$. Since $\chi$ is an additive functor, we have $\chi(Y_s) - \chi(Z_s) = \chi(\P^1_k) = 2$ and $\chi(B_s) - \chi(B^\circ_s) = \chi(\{Q\}) + \chi(\{\infty\}) = 2$. Similarly, one can check that if $b=0$ and $\deg f$ is even, then $Y_s \setminus Z_s = \infty$ and $B =B^\circ$ and in all other cases $Y_s \setminus Z_s = B_s \setminus B^\circ_s = \infty$. Since $\chi(\infty)=1$, the lemma follows. \qedhere
\end{proof}

Let $Y = Y_n \rightarrow Y_{n-1} \cdots \rightarrow Y_1 \rightarrow Y_0 = \P^1_R$ be be the good embedded resolution of the pair $(\P^1_R,\divisor(f))$ as in Definition~\ref{defgoodembres}. Fix $P \in A_\mathrm{bad}$. Let $E_P$ be the exceptional curve for the blowup $Y_1 \rightarrow Y_0$ at $P$, and let $\Gamma$ be the strict transform of the special fiber of $Y_0$ in $Y$. Recall in Definition~\ref{defnreppoly}, we defined $b_P \in \{0,1\}$ as the parity of $\nu_{E_P}(f)$.

Let $Q_P \colonequals \Gamma \cap E_P \in Y_s(k)$. For $P \neq \infty$, let $a \in R$ such that $x-a$ specializes to $P$ in $Y_0$, and let $Q'_P \colonequals \divi_0(x-a) \cap E_P \in (Y_1)_s(k)$. (These are the points $\infty$ and $0$ respectively on $E_P \cong \P^1_k$ in the coordinate $(x-a)/t$.) If $\infty \in A_{\mathrm{bad}}$, let $Q'_{\infty}$ be the closed point where the scheme-theoretic closure of the point at infinity in $\P^1_K$ meets the special fiber of $Y^f$.

Let $B={\divisor(f)}_{\mathrm{odd}} \subset Y$. For $P \neq \infty$, let $Z_P^{\mathrm{sm}},Z_P^{\mathrm{nod}}$ denote the good embedded resolutions of the pairs $(\mathbb{A}^1_R,\divisor(f_P^{\neq \infty}))$ and $(\mathbb{A}^1_R,\divisor(f_P^{\infty}))$ respectively and let $B_P^\mathrm{sm},B_P^{\mathrm{nod}}$ denote $\divisor(f_P^{\neq \infty})_{\mathrm{odd}},\divisor(f_P^{\infty})_{\mathrm{odd}}$ on $Z_P^{\mathrm{sm}},Z_P^{\mathrm{nod}}$ respectively.

\begin{lemma}\label{treereplacethm}
Keep the notation from the three paragraphs above.
\begin{enumerate}[\upshape(a)]
 \item $(Z_P^{\mathrm{sm}})_{s,\mathrm{red}}$ and $(Z_P^{\mathrm{nod}})_{s,\mathrm{red}}$ (and similarly $(B_P^{\mathrm{sm}})_{\mathrm{red}}$ and $(B_P^{\mathrm{nod}})_{\mathrm{red}}$ respectively) are naturally isomorphic to closed subschemes of $Y_{s,\mathrm{red}}$ (and $B_{\mathrm{red}}$ respectively). 
 \item 
 \[ Y_{s,\mathrm{red}} \setminus \left( \bigcup_{\substack{P \in A_\mathrm{bad} \setminus \{\infty\} \\  \deg f_P^{\neq \infty} \geq 1}} (Z_P^{\mathrm{sm}})_{s,\mathrm{red}} \cup \bigcup_{\substack{P \in A_\mathrm{bad} \setminus \{\infty\} \\ \deg f_P^{\infty} \geq 1}} (Z_P^{\mathrm{nod}})_{s,\mathrm{red}} \right) \]
 equals
 \begin{align*} &E_{\infty} \cup \left( \Gamma \setminus A_{\mathrm{bad}} \right) 
 \displaystyle{\bigcup_{\substack{P \in A_\mathrm{bad} \setminus \{\infty\} \\  \deg f_P^{\neq \infty} = 0}}} \{Q_P'\} \cup \displaystyle{\bigcup_{\substack{P \in A_\mathrm{bad} \setminus \{\infty\} \\ \deg f_P^{\infty} = 0}}} \{Q_P\} &\quad \quad \quad \textup{if } b=1 \textup{ and } d=1, \textup{ and,} \\  
 &\Gamma \setminus A_{\mathrm{bad}} \displaystyle{\bigcup_{\substack{P \in A_\mathrm{bad} \setminus \{\infty\} \\  \deg f_P^{\neq \infty} = 0}}} \{Q'_P\} \cup \displaystyle{\bigcup_{\substack{P \in A_\mathrm{bad} \setminus \{\infty\} \\ \deg f_P^{\infty} = 0}}} \{Q_P\} &\quad \quad \quad  \textup{otherwise.}
 \end{align*} 
 \noindent For $P \in A_{\mathrm{bad}} \setminus \{\infty\}$, the left hand side of the intersection below inside $Y_{s,\mathrm{red}}$ is nonempty if and only if $\deg f_P^{\neq \infty} \geq 1$ and $\deg f_P^{\infty} \geq 1$, and in this case we have 
 \[ (Z_P^{\mathrm{sm}})_{s,\mathrm{red}} \cap (Z_P^{\mathrm{nod}})_{s,\mathrm{red}} = E_P \setminus \{Q_P,Q'_P\}\cong \mathbb{P}^1_k \setminus \{0,\infty\} \subset \mathbb{P}^1_k \cong E_P \subset Y_{s,\mathrm{red}} .\] 
 
 \item
  \[ B_{s,\mathrm{red}} \setminus \left( \bigcup_{\substack{P \in A_\mathrm{bad} \setminus \{\infty\} \\ \deg f_P^{\neq \infty} \geq 1}} (B_P^{\mathrm{sm}})_{s,\mathrm{red}} \cup \bigcup_{\substack{P \in A_\mathrm{bad} \setminus \{\infty\} \\ \deg f_P^{\infty} \geq 1}} (B_P^{\mathrm{nod}})_{s,\mathrm{red}} \right) \] 
 equals 
 \begin{align*} (\Gamma \setminus A_{\mathrm{bad}}) \cup \{Q'_{\infty}\} \displaystyle{\bigcup_{\substack{P \in A_\mathrm{bad} \setminus \{\infty\} \\  b_P=1, \deg f_P^{\neq \infty} = 0}}} \{Q'_P\} \cup \displaystyle{\bigcup_{\substack{P \in A_\mathrm{bad} \setminus \{\infty\} \\ \deg f_P^{\infty} = 0}}} \{Q_P\} &\quad \quad \quad \textup{if}\ b=1 \textup{ and } d=1,  \\  
 \Gamma \setminus A_{\mathrm{bad}} \displaystyle{\bigcup_{\substack{P \in A_\mathrm{bad} \setminus \{\infty\} \\  b_P=1, \deg f_P^{\neq \infty} = 0}}} \{Q'_P\} \cup \displaystyle{\bigcup_{\substack{P \in A_\mathrm{bad} \setminus \{\infty\} \\ \deg f_P^{\infty} = 0}}} \{Q_P\} &\quad \quad \quad \textup{if}\ b=1 \textup{ and } 
 d=0, \textup{ and,}  \\ 
 A \setminus A_{\mathrm{bad}} \displaystyle{\bigcup_{\substack{P \in A_\mathrm{bad} \setminus \{\infty\} \\  b_P=1, \deg f_P^{\neq \infty} = 0}}} \{Q'_P\} \cup \displaystyle{\bigcup_{\substack{P \in A_\mathrm{bad} \setminus \{\infty\} \\ b_P=1, \deg f_P^{\infty} = 0}}} \{Q_P\} &\quad \quad \quad \textup{if}\  b=0. \end{align*}

 \noindent The left hand side of the intersection below inside $B_{s,\mathrm{red}}$ is nonempty if and only if $b_P = 1, \deg f_P^{\neq \infty} \geq 1$ and $\deg f_P^{\infty} \geq 1$,
 and in this case we have \[ (B_P^{\mathrm{nod}})_{s,\mathrm{red}} \cap (B_P^{\mathrm{sm}})_{s,\mathrm{red}} = E_P \setminus \{Q_P,Q'_P\} \cong \P^1_k \setminus \{0,\infty\} \subset \mathbb{P}^1_k \cong E_P \subset B_{s,\mathrm{red}}.\]

\end{enumerate} 
\end{lemma}
\begin{proof}
Fix $P \in A_{\mathrm{bad}} \setminus \{\infty\}$. For $i \in C_P$, if we let $a_{P_i}$ be as in Section~\ref{setup}, since we are working over the equicharacteristic ring $R=k[[t]]$, we have $a_{P_i} = a_{P_j}$ for all $i,j \in C_P$.  We move $P$ to the origin by the map $x \mapsto x+a_{P_i}$ and work with $\tilde{g_i}$ instead of $g_i$ for all $i \in C_P$.  We will construct $h^{\geq 1} \in \Spec R[x/t]$ (and $h^{<1} \in \mathcal{O}_Q$ respectively) with the properties that the special fiber of the good embedded resolution of the pair $(\Spec R[x/t],h^{\geq 1})$ (and $(\mathcal{O}_Q,\divisor(h^{<1}))$ respectively) is naturally a subset of $(Y^f_s)_{\mathrm{red}}$. We will then show that up to a change of variables and multiplication by a unit, the function $h^{\geq 1}$ equals $f_P^{\neq \infty}$ (and $h^{<1}$ equals $f_P^{\infty}$ respectively). 

The reason we only get isomorphisms of the underlying reduced subschemes is that the change of variables to go from $h^{<1}$ to $f_P^\infty$ uses the isomorphism of complete local $k$-algebras $k[[x,t/x]] \cong k[[t,x]]$ given by $x \mapsto t$ and $t/x \mapsto x$, which is not an isomorphism of $k[[t]]$-algebras. Hence, we do not expect the multiplicities of the components in the special fiber of a good embedded resolution to agree, and we only get equalities of the underlying reduced subschemes. 
\begin{enumerate}[\upshape(a)]
\item We first show that $(Z_P^{\mathrm{sm}})_{s,\mathrm{red}}$ is a closed subset of $Y_{s,\mathrm{red}}$. Using the formula for $\nu_{E_P}(f)$ from Corollary~\ref{fmultexc}, we get
\[ f = ut^b \left(\prod_{i \in C_P^{<1}} t^{\lambda_i} (x/t)^{\lambda_i} (\tilde{g_i}/x^{\lambda_i}) \right) \left(\prod_{i \in C_P^{\geq 1}} t^{n_i} (\tilde{g_i}/t^{n_i}) \right)  = \left( u \prod_{i \in C_P^{<1}} (x/t)^{\lambda_i} (\tilde{g_i}/x^{\lambda_i}) \right) t^{\nu_E(f)} \left(\prod_{i \in C_P^{\geq 1}} (\tilde{g_i}/t^{n_i}) \right). \]
By Lemma~\ref{eqforstricttransform}, $\nu_{E_P}(\tilde{g_i}/x^{\lambda_i}) = 0$ for $i \in C_P^{<1}$. Since $\nu_{E_P}(x/t) = 0$ and since $(\tilde{g_i}/x^{\lambda_i})$ specializes to the point at $\infty$ on $E_P$, it follows that $\left( u \prod_{i \in C_P^{<1}} (x/t)^{\lambda_i} (\tilde{g_i}/x^{\lambda_i}) \right)$ is a unit on the affine patch $\Spec R[x/t]$ of the blowup $Y_1 \rightarrow Y_0$.  By Definition~\ref{goodembres}, it also follows that the good embedded resolution of $(\Spec R[x/t],\divisor(f))$ is a closed subset of the good embedded resolution of $(Y_0,\divisor(f))$. Let $h^{\geq 1} \colonequals t^{b_P}\left(\prod_{i \in C_P^{\geq 1}} (\tilde{g_i}/t^{n_i}) \right)$. Since $b_P = 1$ when $\nu_{E_P}(f)$ is odd and $b_P = 0$ when $\nu_{E_P}(f)$ is even, Lemma~\ref{blowupiso} implies that the good embedded resolutions of the pairs $(\Spec R[x/t],\divisor(f))$ and $(\Spec R[x/t],\divisor(h^{\geq 1}))$ are equal, and furthermore $\divisor(f)_{\mathrm{odd}} = \divisor(h^{\geq 1})_{\mathrm{odd}}$ on the resolution. Since $h^{\geq 1}$ is $f_{P}^{\neq \infty}$ up to the change of variables $x/t \mapsto x$, it follows that these two pairs have the same embedded resolution as the pair $(\Spec R[x],\divisor(f_P^{\neq \infty}))$. Let $\Gamma_{\neq \infty}$ be the strict transform of the special fiber of $\mathbb{A}^1_R$ in the good embedded resolution of the pair $(\mathbb{A}^1_R,\divisor(f_P^{\neq \infty}))$. From the definition of $f_P^{\neq \infty}$, it follows that $\Gamma_{\neq \infty} \subset B_P^{\mathrm{sm}}$ if and only if $b_P = 1$, or equivalently, if and only if $E_P \subset \divisor(f)_{\mathrm{odd}} = B \subset Y$. Putting the above identifications together, and identifying $(\Gamma_{\neq \infty})_{\mathrm{red}}$ with $(E_P)_{\mathrm{red}}$, we get that $(Z_P^{\mathrm{sm}})_{s,\mathrm{red}}$ (and $(B_P^{\mathrm{sm}})_{\mathrm{red}}$ respectively) is a closed subset of $Y_{s,\mathrm{red}}$ (of $B_{\mathrm{red}}$ respectively). 
%\padma{Did I add in/forget the extra blowup if both $b$ and $b_P$ are odd -- will this mess up calculations in the last section? :( }

We now show that $(Y^{f_P^{\infty}}_s)_{\mathrm{red}}$ is a closed subset of $(Y^f_s)_{\mathrm{red}}$.  Recall that $Q \in (Y_1)_s(k)$ is the point where the exceptional curve $E$ for the blowup at $P$ meets the rest of $(Y_1)_s$. For each $i \in C_P^{\geq 1}$, Lemma~\ref{specialization} shows that $\tilde{g_i}/x^{n_i}$ is a unit in $\mathcal{O}_Q$ for every $i \in C_P^{\geq 1}$.  As before, we can now factor $f$ as 
\[ f = ut^b \left(\prod_{i \in C_P^{<1}} x^{\lambda_i} (\tilde{g_i}/x^{\lambda_i}) \right) \left(\prod_{i \in C_P^{\geq 1}} x^{n_i} (\tilde{g_i}/x^{n_i}) \right)  = \left( u \prod_{i \in C_P^{\geq 1}} (\tilde{g_i}/x^{n_i}) \right) (t/x)^b x^{\nu_E(f)} \left(\prod_{i \in C_P^{< 1}} (\tilde{g_i}/x^{\lambda_i}) \right). \]
Let $h^{< 1} \colonequals (t/x)^b x^{b_P} \left(\prod_{i \in C_P^{< 1}} (\tilde{g_i}/x^{\lambda_i}) \right)$. As before, when combined with Lemma~\ref{blowupiso} and the definition of $b_P$, this yields that the pairs $(\mathcal{O}_Q,\divisor(f))$ and $(\mathcal{O}_Q,\divisor(h^{<1}))$ have isomorphic good embedded resolutions, and that $\divisor(f)_{\mathrm{odd}} = \divisor(h^{<1})_{\mathrm{odd}}$ on the resolution. The reduced special fiber of the good embedded resolution of $(\mathcal{O}_Q,\divisor(f))$ is a closed subset of $(Y^f_s)_{\mathrm{red}}$, with $Q$ identified with the point at $\infty$ on the exceptional curve $E_P$. Let $\Gamma_{ \infty}$ be the strict transform of the special fiber of $\mathbb{P}^1_R$ in the good embedded resolution of the pair $(\P^1_R,\divisor(f_P^{\infty}))$. From the definition of $f_P^{\infty}$, it follows that $\Gamma_{\infty} \subset B^{f_P^{\infty}}$ if and only if $b_P = 1$, or equivalently, if and only if $E_P \subset \divisor(f)_{\mathrm{odd}} = B^f \subset Y^f$. Since $h^{<1}$ is $f_{P}^{\infty}$ up to the change of variables $t/x  \mapsto x$ and $x \mapsto t$ and multiplication by the unit $u^{\Diamond} \in \mathcal{O}_Q$, by identifying $(\Gamma_{\infty})_{\mathrm{red}}$ with $(E_P)_{\mathrm{red}}$ as before, and using the isomorphism of good embedded resolutions of $(\mathcal{O}_Q,\divisor(f))$ and $(\mathcal{O}_Q,\divisor(h^{<1}))$, we also get that $(Y^{f_P^{ \infty}}_s)_{\mathrm{red}}$ (and $(B^{f_P^{ \infty}})_{\mathrm{red}}$ respectively) is a closed subset of $(Y^f_s)_{\mathrm{red}}$ (and $(B^f)_{\mathrm{red}}$ respectively). 
\item Since $Y^f$ is also the good embedded resolution of $(Y_1,\divisor(f))$, since $E_P = \{Q\} \cup E_P \setminus \{Q\}$, and since $(\mathcal{O}_Q,\divisor(h^{<1}))$ and $(E_P \setminus \{Q\} = \Spec R[x/t],\divisor(h^{\geq 1}))$ from part~(a) above have the same good embedded resolutions as $(\mathcal{O}_Q,\divisor(f))$ and $(E_P \setminus \{Q\} = \Spec R[x/t],\divisor(f))$ respectively, the result follows from the identifications and change of variables in part~(a) above.
\item Since $\Gamma = \divi(f)_{\mathrm{odd}}$, the component $\Gamma \subset B$ if and only if $b=1$, the component $E_P \subset B$ if and only if $b_P = 1$, the components $\Gamma_{\infty}, \Gamma_{\neq \infty}$ from the proof of part~(a) appear in $B^{f_P^{\infty}},B^{f_P^{\neq \infty}}$ respectively if and only if $b_P = 1$. Since the left hand side equals the intersection of the left hand side of part~(b) intersected with $B_{s,\mathrm{red}}$, we get part~(c) by intersecting the right hand side of part~(b) with $B_{s,\mathrm{red}}$. \qedhere
\end{enumerate} 
\end{proof}

\begin{corollary}\label{diffchiY} For any $\star \in \{f,f_P^{\neq \infty},f_P^{\infty}\}$, let $Y^\star,B^\star$ be as in Definition~\ref{models}. Then
\begin{itemize}
 \item 
 \begin{multline*} \chi (Y^f_s)- \displaystyle\sum_{\substack{P \in A_\mathrm{bad} \setminus \{\infty\} \\ \deg f_P^{\infty} \geq 1}} \chi (Y^{f^{\infty}_P}_s) -  \displaystyle\sum_{\substack{P \in A_\mathrm{bad} \setminus \{\infty\} \\ \deg f_P^{\neq \infty} \geq 1}} \chi (Y^{f^{\neq \infty}_P}_s) \\ =  
 2+2 b d - \sharp(A_{\mathrm{bad}})+\left(\sum_{\substack{P \in A_{\mathrm{bad}} \setminus \{ \infty \} \\ \deg f_P^{\infty} = 0 \ \textup{or} \ \deg f_P^{\neq \infty} = 0}} 1 \right)- \displaystyle\sum_{\substack{P \in A_\mathrm{bad} \setminus \{\infty\} \\ \deg f_P^{\infty} \geq 1}} (1+b_P  d_P^{\mathrm{nod}}) -  \displaystyle\sum_{\substack{P \in A_\mathrm{bad} \setminus \{\infty\} \\ \deg f_P^{\neq \infty} \geq 1}} (1+b_P  d_P^{\mathrm{sm}}) 
\end{multline*}
\item 
 \[ \chi (B^f_s)-  \displaystyle\sum_{\substack{P \in A_\mathrm{bad} \setminus \{\infty\} \\ \deg f_P^{\infty} \geq 1}} \chi (B^{f^{\infty}_P}_s) -  \displaystyle\sum_{\substack{P \in A_\mathrm{bad} \setminus \{\infty\} \\ \deg f_P^{\neq \infty} \geq 1}} \chi (B^{f^{\neq \infty}_P}_s) \] equals
 \begin{multline*} \sharp(A) - \sharp(A_{\mathrm{bad}}) + b(2+2d-\sharp(A))+\left(\sum_{\substack{P \in A_{\mathrm{bad}} \setminus \{ \infty \} \\ \deg f_P^{\infty} = 0}} (b+b_P-bb_P) \right)+\left(\sum_{\substack{P \in A_{\mathrm{bad}} \setminus \{ \infty \} \\ \deg f_P^{\neq \infty} = 0}} b_P \right) \\ - \sum_{\substack{P \in A_{\mathrm{bad}} \setminus \{ \infty \} \\ \deg f_P^{\infty} \geq 1}} (b_P+d_P^{\mathrm{nod}}) - \sum_{\substack{P \in A_{\mathrm{bad}} \setminus \{ \infty \} \\ \deg f_P^{\neq \infty} \geq 1}} (b_P+d_P^{\mathrm{sm}}). \end{multline*}
 \end{itemize}
\end{corollary}
\begin{proof}
We will continue to use the notation from the lemma above. Since $k$ is algebraically closed, and $\chi$ is an additive functor that takes a disjoint union of locally closed subsets to the corresponding sum of integers, the equalities $\chi(\P^1_k) = 2$ and $\chi(k-\textup{rational point}) = 1$ imply that $\chi(E_P) = 2, \chi(\mathbb{P}^1_k \setminus \{0,\infty\}) = 0$ and $\chi(\Gamma \setminus A_{\mathrm{bad}}) = 2 - \sharp(A_{\mathrm{bad}})$ for every $P \in A_{\mathrm{bad}}$. Since $\chi$ only depends on the underlying reduced subscheme, using the additivity of $\chi$ once again with Lemma~\ref{treereplacethm}~(a,b) and the fact that $d \in \{0,1\}$ and $d=1$ exactly when $\deg(f)$ is odd, we get
\[ \chi (Y^f_s)- \displaystyle\sum_{\substack{P \in A_\mathrm{bad} \setminus \{\infty\} \\ \deg f_P^{\infty} \geq 1}} \chi (Z_{P,s}^{\mathrm{nod}}) -  \displaystyle\sum_{\substack{P \in A_\mathrm{bad} \setminus \{\infty\} \\ \deg f_P^{\neq \infty} \geq 1}} \chi (Z_{P,s}^{\mathrm{sm}}) = 
2 b d + 2 - \sharp(A_{\mathrm{bad}})+\sum_{\substack{P \in A_{\mathrm{bad}} \setminus \{ \infty \} \\ \deg f_P^{\infty} = 0 \ \textup{or} \ \deg f_P^{\neq \infty} = 0}} 1 .\]

The first equality now follows by applying Lemma~\ref{affine} to $(f_P^{\infty},b_P)$ and $(f_P^{\neq \infty},b_P)$ instead of $(f,b)$. 

Observe that $\sharp(A) - \sharp(A_{\mathrm{bad}}) + b(2+2d-\sharp(A))$ equals $\chi(A \setminus A_{\mathrm{bad}})$ when $b=0$, equals $\chi(\Gamma \setminus A_{\mathrm{bad}})$ when $b=1$ and $\deg(f)$ is even and, equals $\chi(E_{\infty})+\chi(\Gamma \setminus A_{\mathrm{bad}})$ when $b=1$ and $\deg(f)$ is odd. For $P \in A_{\mathrm{bad}} \setminus \{\infty\}$, Lemma~\ref{treereplacethm}(c) shows that $Q_P$ is not in the right hand side if and only $b=b_P=0$. Since $b+b_P-bb_P$ is $0$ when $b=b_P=0$ and $1$ otherwise, it follows that $b+b_P-bb_P$ equals $\chi(B_{s,\mathrm{red}} \cap \{Q_P\})$. We also have that $Q'_P$ is in the right hand side of Lemma~\ref{treereplacethm}(c) exactly when $b_P=1$.

The proof of the second equality is now similar to the first and uses Lemma~\ref{treereplacethm}~(a,c) and the second equality of Lemma~\ref{affine} and the observations in the previous paragraph.  \qedhere 
\end{proof}

\begin{thm}\label{condchange}
Keeping the notation from Section~\ref{setup} and Lemma~\ref{diffchiY}, we get
\[ -\Art (X^f/S)- \sum_{\substack{P \in A_{\mathrm{bad}} \setminus \{\infty\} \\ \deg f_P^{\infty} \geq 1}} [-\Art (X^{f_P^{\infty}}/S)]- \sum_{\substack{P \in A_{\mathrm{bad}} \setminus \{\infty\} \\ \deg f_P^{\neq \infty} \geq 1}} [-\Art (X^{f_P^{\neq \infty}}/S)]  \]
equals 
\begin{multline*} -b(2+d)+\displaystyle\sum_{\substack{P \in A \setminus A_{\mathrm{bad}} \\ P \neq \infty, g_i \in C_P}} (n_i-1+b)+(2+b)\sharp(A_{\mathrm{bad}})+\sum_{P \in A_\mathrm{bad}} \sum_{g_i \in C_P^{<1}} \left( n_i - \lambda_i \right) \\
  - \left(\sum_{\substack{P \in A_{\mathrm{bad}} \setminus \{ \infty \} \\ \deg f_P^{\infty} = 0}} (b-bb_P) \right)+\displaystyle\sum_{\substack{P \in A_{\mathrm{bad}} \setminus \{ \infty \} \\ \deg f_P^{\infty} \geq 1 \ \textup{and} \\ \deg f_P^{\neq \infty} \geq 1}} 2b_P-\displaystyle\sum_{\substack{P \in A_\mathrm{bad} \setminus \{\infty\} \\ \deg f_P^{\infty} \geq 1}} (b+2b_P  d_P^{\mathrm{nod}}) -  \displaystyle\sum_{\substack{P \in A_\mathrm{bad} \setminus \{\infty\} \\ \deg f_P^{\neq \infty} \geq 1}} 2b_P  d_P^{\mathrm{sm}}. \end{multline*}
\end{thm}
\begin{proof}
The idea is to combine Corollary~\ref{diffchiY} with the Riemann-Hurwitz formula Lemma~\ref{formula}. 

Since $X^{\star}_{\overline{\eta}}$ is a hyperelliptic curve for $\star \in \{f,f_P^{\neq \infty},f_P^{\infty}\}$, using Definition~\ref{Dparitybit} we have \[ \chi(X^f_{\overline{\eta}}) = 4-d-\deg f, \quad \chi(X^{f_P^{\infty}}_{\overline{\eta}}) = 4-d_P^{\mathrm{nod}}-\deg f_P^{\infty}, \quad \chi(X^{f_P^{\neq \infty}}_{\overline{\eta}}) = 4-d_P^{\mathrm{sm}}-\deg f_P^{\neq \infty}.\] 
From Definition~\ref{defnreppoly} it follows that 
\[ \deg f = \sum_{P \in A} \sum_{i \in C_P} n_i, \quad \quad \deg f_P^{\infty} = b+\sum_{i \in C_P^{<1}} \lambda_i, \quad \quad \deg f_P^{\neq \infty} = \sum_{i \in C_P^{\geq 1}} n_i .\] 
Putting the last two displayed equations together, we get that
\begin{align} \begin{split}&\quad \sum_{\substack{P \in A_{\mathrm{bad}} \setminus \{\infty\} \\ \deg(f_P^{\infty}) \geq 1}} \chi(X_{\overline{\eta}}^{f_P^{\infty}}) + \sum_{\substack{P \in A_{\mathrm{bad}} \setminus \{\infty\} \\ \deg(f_P^{\neq \infty}) \geq 1}} \chi(X_{\overline{\eta}}^{f_P^{\neq \infty}}) - \chi(X^f_{\overline{\eta}}) \\ &= d-4- \sum_{\substack{P \in A_{\mathrm{bad}} \setminus \{\infty\} \\ \deg(f_P^{\infty}) \geq 1 }} (d_P^{\mathrm{nod}}+b-4)- \sum_{\substack{P \in A_{\mathrm{bad}} \setminus \{\infty\} \\ \deg(f_P^{\neq \infty}) \geq 1 }} \left(d_P^{\mathrm{sm}}-4 \right) +\sum_{\substack{P \in A \setminus A_{\mathrm{bad}} \\ P \neq \infty, g_i \in C_P}} n_i + \sum_{\substack{P \in A_{\mathrm{bad}} \\ P \neq \infty}} \sum_{g_i \in C_P^{<1}} (n_i-\lambda_i). \end{split}\end{align}

Similarly, combining Corollary~\ref{diffchiY} with the Riemann-Hurwitz formula and Remark~\ref{RAbaddeg} which says that if $P \in A_{\mathrm{bad}} \setminus \{\infty\}$, then we cannot have $\deg(f_P^{\infty}) = \deg(f_P^{\neq \infty}) = 0$,  we get
\begin{align}\begin{split} &\quad \chi(X^f_{s}) - \sum_{\substack{P \in A_{\mathrm{bad}} \setminus \{\infty\} \\ \deg(f_P^{\infty}) \geq 1}} \chi(X_{s}^{f_P^{\infty}}) - \sum_{\substack{P \in A_{\mathrm{bad}} \setminus \{\infty\} \\ \deg(f_P^{\neq \infty}) \geq 1}} \chi(X_{s}^{f_P^{\neq \infty}}) \\ 
&= \left[ 2 \chi(Y^f_s) - \chi(B^f_s) \right]- \sum_{\substack{P \in A_{\mathrm{bad}} \setminus \{\infty\} \\ \deg(f_P^{\infty}) \geq 1}} \left[ 2 \chi(Y^{f_P^{\infty}}_s) - \chi(B^{f_P^{\infty}}_s) \right] - \sum_{\substack{P \in A_{\mathrm{bad}} \setminus \{\infty\} \\ \deg(f_P^{\neq \infty}) \geq 1}} \left[ 2 \chi(Y^{f_P^{\neq \infty}}_s) - \chi(B^{f_P^{\neq \infty}}_s) \right] \\
&= 2 \left[ 2+2 b d - \sharp(A_{\mathrm{bad}})+\left(\sum_{\substack{P \in A_{\mathrm{bad}} \setminus \{ \infty \} \\ \deg f_P^{\infty} = 0 \ \textup{or} \ \deg f_P^{\neq \infty} = 0}} 1 \right)- \displaystyle\sum_{\substack{P \in A_\mathrm{bad} \setminus \{\infty\} \\ \deg f_P^{\infty} \geq 1}} (1+b_P  d_P^{\mathrm{nod}}) -  \displaystyle\sum_{\substack{P \in A_\mathrm{bad} \setminus \{\infty\} \\ \deg f_P^{\neq \infty} \geq 1}} (1+b_P  d_P^{\mathrm{sm}}) \right] \\ 
&\quad - \left[ \sharp(A) - \sharp(A_{\mathrm{bad}}) + b(2+2d-\sharp(A))+\left(\sum_{\substack{P \in A_{\mathrm{bad}} \setminus \{ \infty \} \\ \deg f_P^{\infty} = 0}} (b+b_P-bb_P) \right)+\left(\sum_{\substack{P \in A_{\mathrm{bad}} \setminus \{ \infty \} \\ \deg f_P^{\neq \infty} = 0}} b_P \right) \right. \\ 
&\hspace*{9.5cm} \left. - \sum_{\substack{P \in A_{\mathrm{bad}} \setminus \{ \infty \} \\ \deg f_P^{\infty} \geq 1}} (b_P+d_P^{\mathrm{nod}}) - \sum_{\substack{P \in A_{\mathrm{bad}} \setminus \{ \infty \} \\ \deg f_P^{\neq \infty} \geq 1}} (b_P+d_P^{\mathrm{sm}}) \right] \\
&= 4-2b+2bd+(b-1)\sharp(A)- \sharp(A_{\mathrm{bad}})+\left(\sum_{\substack{P \in A_{\mathrm{bad}} \setminus \{ \infty \} \\ \deg f_P^{\infty} = 0 \ \textup{or} \ \deg f_P^{\neq \infty} = 0}} (2-b_P) \right) - \left(\sum_{\substack{P \in A_{\mathrm{bad}} \setminus \{ \infty \} \\ \deg f_P^{\infty} = 0}} (b-bb_P) \right) \\
&\hspace*{4cm}-\displaystyle\sum_{\substack{P \in A_\mathrm{bad} \setminus \{\infty\} \\ \deg f_P^{\infty} \geq 1}} (2-b_P+2b_P  d_P^{\mathrm{nod}} -d_P^{\mathrm{nod}}) -  \displaystyle\sum_{\substack{P \in A_\mathrm{bad} \setminus \{\infty\} \\ \deg f_P^{\neq \infty} \geq 1}} (2-b_P+2b_P  d_P^{\mathrm{sm}} -d_P^{\mathrm{sm}}). \end{split}\end{align}
By Lemma~\ref{Lparitydegree}, Definition~\ref{Dparitybit} and the definitions of the sets $A$ and $A_{\mathrm{bad}}$, it follows that $d=1$ precisely when $\deg(f)$ is odd, which is precisely when $\infty \in A$, and similarly $\infty \in A_{\mathrm{bad}}$ when both $b=1$ and $\deg(f)$ is odd, or equivalently when $bd=1$. Using these and rearranging  terms gives the following three equalities.
 \begin{equation} \displaystyle\sum_{\substack{P \in A \setminus A_{\mathrm{bad}} \\ P \neq \infty, g_i \in C_P}} n_i+(b-1)\sharp(A)- \sharp(A_{\mathrm{bad}}) = \displaystyle\sum_{\substack{P \in A \setminus A_{\mathrm{bad}} \\ P \neq \infty, g_i \in C_P}} (n_i-1+b)+{(b-1)d}-2\sharp(A_{\mathrm{bad}})+b\sharp(A_{\mathrm{bad}}). \end{equation}
 \begin{equation} \displaystyle\sum_{\substack{P \in A_{\mathrm{bad}} \setminus \{ \infty \} \\ \deg f_P^{\infty} = 0 \ \textup{or} \\ \deg f_P^{\neq \infty} = 0}} (2-b_P)-\displaystyle\sum_{\substack{P \in A_\mathrm{bad} \setminus \{\infty\} \\ \deg f_P^{\infty} \geq 1}} (2-b_P)-\displaystyle\sum_{\substack{P \in A_\mathrm{bad} \setminus \{\infty\} \\ \deg f_P^{\neq \infty} \geq 1}} (2-b_P) = -\displaystyle\sum_{\substack{P \in A_{\mathrm{bad}} \setminus \{ \infty \} \\ \deg f_P^{\infty} \geq 1 \ \textup{and} \\ \deg f_P^{\neq \infty} \geq 1}} (4-2b_P).\end{equation}
 \begin{equation}\displaystyle\sum_{\substack{P \in A_\mathrm{bad} \setminus \{\infty\} \\ \deg f_P^{\infty} \geq 1}} 4+\displaystyle\sum_{\substack{P \in A_\mathrm{bad} \setminus \{\infty\} \\ \deg f_P^{\neq \infty} \geq 1}} 4 + (b-2)\sharp(A_{\mathrm{bad}})-\displaystyle\sum_{\substack{P \in A_{\mathrm{bad}} \setminus \{ \infty \} \\ \deg f_P^{\infty} \geq 1 \ \textup{and} \\ \deg f_P^{\neq \infty} \geq 1}} 2 = -4bd+(2+b)\sharp(A_{\mathrm{bad}})+\displaystyle\sum_{\substack{P \in A_{\mathrm{bad}} \setminus \{ \infty \} \\ \deg f_P^{\infty} \geq 1 \ \textup{and} \\ \deg f_P^{\neq \infty} \geq 1}} 4.\end{equation}

For $\star \in \{f,f_P^{\neq \infty},f_P^{\infty}\}$, by definition, we have $-\Art(X^\star) = \chi(X^\star_s)-\chi(X^\star_{\overline{\eta}})$. Combining this with the five numbered equations above, it follows that the left hand side equals
\begin{multline*}
 \quad \left[ \chi(X^f_s)-\chi(X^f_{\overline{\eta}}) \right]- \sum_{\substack{P \in A_{\mathrm{bad}} \setminus \{\infty\} \\ \deg(f_P^{\infty}) \geq 1}} \left[ \chi(X^{f_P^{\infty}}_s) - \chi(X^{f_P^{\infty}}_{\overline{\eta}}) \right]  -  \sum_{\substack{P \in A_{\mathrm{bad}} \setminus \{\infty\} \\ \deg(f_P^{\neq \infty}) \geq 1}} \left[ \chi(X^{f_P^{\neq \infty}}_s)- \chi(X^{f_P^{\neq \infty}}_{\overline{\eta}}) \right]  \\
 = -b(2+d)+\displaystyle\sum_{\substack{P \in A \setminus A_{\mathrm{bad}} \\ P \neq \infty, g_i \in C_P}} (n_i-1+b)+(2+b)\sharp(A_{\mathrm{bad}})+\sum_{P \in A_\mathrm{bad}} \sum_{g_i \in C_P^{<1}} \left( n_i - \lambda_i \right) \\
  - \left(\sum_{\substack{P \in A_{\mathrm{bad}} \setminus \{ \infty \} \\ \deg f_P^{\infty} = 0}} (b-bb_P) \right)+\displaystyle\sum_{\substack{P \in A_{\mathrm{bad}} \setminus \{ \infty \} \\ \deg f_P^{\infty} \geq 1 \ \textup{and} \\ \deg f_P^{\neq \infty} \geq 1}} 2b_P-\displaystyle\sum_{\substack{P \in A_\mathrm{bad} \setminus \{\infty\} \\ \deg f_P^{\infty} \geq 1}} (b+2b_P  d_P^{\mathrm{nod}}) -  \displaystyle\sum_{\substack{P \in A_\mathrm{bad} \setminus \{\infty\} \\ \deg f_P^{\neq \infty} \geq 1}} 2b_P  d_P^{\mathrm{sm}}.
\end{multline*}
\end{proof}

\section{Metric trees of polynomials}\label{secmettree}
\subsection{Overview of this section}
For $P \in A_{\mathrm{bad}}$ and $i \in C_P^{<1}$, it is hard to directly relate the discriminant of $g_i$ with the discriminant of the corresponding replacement polynomial $h_i$ (see Remark~\ref{Rcomprep}), and use it to compute $\nu(\Delta_{f_P^{\infty}})$. Instead, we first define the metric tree $T(f)$ attached to a separable polynomial $f \in R[x]$ (See Example~\ref{Exmetpoly} and Figure~\ref{fig:MetTree1}), which is a combinatorial gadget for recording the $t$-adic distances between all pairs of roots. The main results of this section are Theorem~\ref{mettreerep1} and Theorem~\ref{mettreerep2} that describe how to obtain the metric tree of the replacement polynomials $f_P^{\neq \infty}$ and $f_P^{\infty}$ from the metric tree of $f$. 

More precisely, Lemma~\ref{gromov} shows that $\nu(\Delta_f)$ can be computed from the lengths of edges in the tree $T(f)$ for any monic separable polynomial $f$. In Theorem~\ref{expofreplacement}, we describe how to extract certain exponents and corresponding coefficients of the Newton-Puiseux expansions of the roots of the replacement polynomials $h_i$ from those of $g_i$, and use them to build the metric tree $T(f_P^{\infty})$ of the replacement polynomial from the metric tree $T(f)$ of $f$ by appropriately gluing together the metric trees of the irreducible factors of $f_P^{\infty}$. This will then be used together with Lemma~\ref{gromov} in Theorem~\ref{discchange} for estimating how discriminants change under the replacement operation. 

Throughout this section, we will use some basic terminology of Berkovich spaces; see \cite{BakerRumely} for a detailed introduction to the subject.

\subsection{The metric tree $T(f)$ and the discriminant $\Delta_f$}
\begin{defn}
 Let $S$ be a finite subset of $\mathbb{P}^{1,\mathrm{Berk}}_{\overline{K}}$. The {\textup{\textsf{\color{blue}convex hull}}} $C(S)$ of $S$ is the smallest connected metric subtree of $\mathbb{P}^{1,\mathrm{Berk}}_{\overline{K}}$ containing $S$, with the infinite ends towards the type 1 points in $S$ deleted. 
\end{defn}

\begin{eg}\label{firsteg}
 Let $K = \C((t))$. Let $S = \{\zeta, t^{{2/3}}+t^{5/6},t^{{2/3}}-t^{5/6},\omega t^{{2/3}}-\omega^2 t^{5/6},\omega t^{{2/3}}+\omega^2 t^{5/6},\omega^2 t^{{2/3}}+\omega t^{5/6},\omega^2 t^{{2/3}}-\omega t^{5/6} \}$. Then $C(S)$ is the metric tree in Figure~\ref{fig:MetTree1}.
\end{eg}

\begin{defn}\label{mettreedefn}
 Let $f$ be a monic polynomial in $R[x]$. The {\textup{\textsf{\color{blue}metric tree}}} $T(f)$ of $f$ is the convex hull of the Gauss point $\zeta$ and the roots of $f$ (identified with type I points on $\mathbb{P}^{1,\mathrm{Berk}}_{\overline{K}}$).  
\end{defn}

\begin{eg}\label{Exmetpoly}
 Let $f$ be the minimal polynomial of $t^{\frac{2}{3}}+t^{\frac{5}{6}}$ over $\C((t))$. Then $T(f)$ is the metric tree $C(S)$ in Example~\ref{firsteg}.
\end{eg}

\begin{defn}For any two type 1 points $\alpha$ and $\beta$ and a type $2$ point $\gamma$ in $\mathbb{P}^{1,\mathrm{Berk}}_{\overline{K}}$, observe that $C(\{\alpha,\beta,\gamma\})$ is a line segment of finite length. The {\textup{\textsf{\color{blue}Gromov product}}} $(\alpha|\beta)_{\gamma}$ of $\alpha$ and $\beta$ with respect to $\gamma$ is the length of $C(\{\alpha,\beta,\gamma\})$ .  
\end{defn}

\begin{eg}\label{Egromov}
 In Figure~\ref{fig:Gromov}, the metric tree $C(\omega^2 t^{2/3} - \omega t^{5/6},\omega t^{2/3} + \omega^2 t^{5/6},\zeta)$ is coloured $\r{red}$ and the metric tree $C(t^{2/3} - t^{5/6}, t^{2/3} + t^{5/6},\zeta)$ is coloured $\g{green}$. This shows that $\r{(\omega^2 t^{2/3} - \omega t^{5/6}|\omega t^{2/3} + \omega^2 t^{5/6})_{\zeta} = 2/3}$ and $\g{(t^{2/3} - t^{5/6}| t^{2/3} + t^{5/6})_{\zeta} = 5/6}$. 
\end{eg}

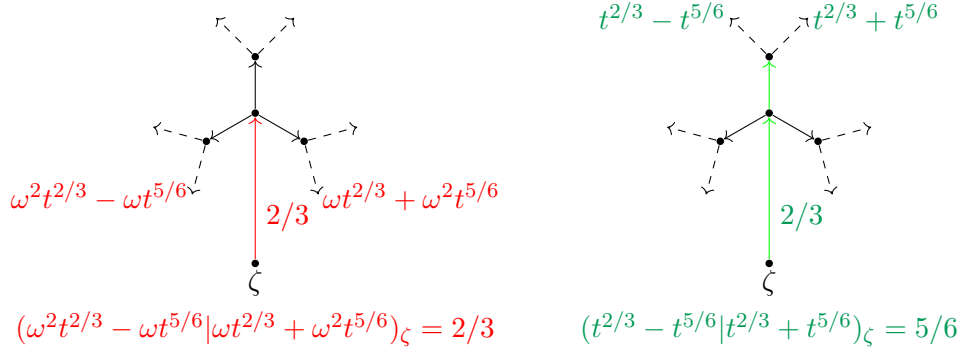
\begin{figure}
\def\r{\color{red}}\def\b{\color{blue}}\def\g{\color{ForestGreen}} 
\hbox to \textwidth{\hss
\tikzstyle{my dot}=[fill,circle,inner sep=1pt]
\tikzstyle{level 1}=[sibling angle=120]
\tikzstyle{level 2}=[sibling angle=90]
\tikzstyle{level 3}=[sibling angle=60]
\tikzstyle{level 4}=[sibling angle=30]
\tikzstyle{small node}=[inner sep=0pt, outer sep=0pt]
\begin{tikzpicture}[rotate=90,grow cyclic,edge from parent/.style={->,draw},scale=0.5, every node/.style={font=\small},circle, inner sep=2pt, every label/.style={rectangle,inner sep=1pt,outer sep=0pt}, baseline=-20mm]
\node (eta) at (-2,0) [my dot,label=below:$\zeta$] {};
\node (noname) at (2,0) [my dot] {} 
  child { node [my dot] {} 
    child {node [small node,font=\small,label=right:$\r{\omega t^{2/3} + \omega^2 t^{5/6}}$] {} edge from parent[dashed]} 
    child {node [small node] {} edge from parent[dashed]} } 
  child { node [my dot] {} 
    child {node [small node] {} edge from parent[color=black,dashed]} 
    child {node [small node] {} edge from parent[color=black,dashed]} edge from parent }
  child { node [my dot] {} 
    child {node [small node] {} edge from parent[dashed]} 
    child {node [small node,font=\small,label=left:$\r{\omega^2 t^{2/3} - \omega t^{5/6}}$] {} edge from parent[dashed]}}; 
\draw [color=red,->] (eta) -- node[anchor=west,pos=.3, xshift=-0.5pt] {$\r{2/3}$} (noname);
\node[below=5mm of eta, shape=rectangle,font=\small] {$\r{(\omega^2 t^{2/3} - \omega t^{5/6}|\omega t^{2/3} + \omega^2 t^{5/6})_{\zeta} = 2/3}$};
\end{tikzpicture}
\tikzstyle{my dot}=[fill,circle,inner sep=1pt]
\tikzstyle{level 1}=[sibling angle=120]
\tikzstyle{level 2}=[sibling angle=90]
\tikzstyle{level 3}=[sibling angle=60]
\tikzstyle{level 4}=[sibling angle=30]
\tikzstyle{small node}=[inner sep=0pt, outer sep=0pt]
\begin{tikzpicture}[rotate=90,grow cyclic,edge from parent/.style={->,draw},scale=0.5, every node/.style={font=\small},circle, inner sep=2pt, every label/.style={rectangle,inner sep=1pt,outer sep=0pt}, baseline=-20mm]
\node (eta) at (-2,0) [my dot,label=below:$\zeta$] {};
\node (noname) at (2,0) [my dot] {} 
  child { node [my dot] {} 
    child {node [small node] {} edge from parent[dashed]} 
    child {node [small node] {} edge from parent[dashed]} } 
  child { node [my dot] {} 
    child {node [small node,label=right:$\g{t^{2/3} + t^{5/6}}$,font=\small] {} edge from parent[color=black,dashed]} 
    child {node [small node,label=left:$\g{t^{2/3} - t^{5/6}}$,font=\small] {} edge from parent[color=black,dashed]} edge from parent[color=green] }
  child { node [my dot] {} 
    child {node [small node] {} edge from parent[dashed]} 
    child {node [small node] {} edge from parent[dashed]}}; 
\draw [color=green,->] (eta) -- node[anchor=west,pos=.3, xshift=-0.5pt] {$\g{2/3}$} (noname);
\node[below=5mm of eta, shape=rectangle,font=\small] {$\g{(t^{2/3} - t^{5/6}| t^{2/3} + t^{5/6})_{\zeta} = 5/6}$} ;
\end{tikzpicture}
\hss}
\caption{Figure~\ref{fig:Gromov}: The Gromov product and $t$-adic distances} \label{fig:Gromov}
\end{figure}

\begin{lemma}\label{gromov}
 Let $f$ be a monic polynomial in $R[x]$. Then
 \[ \nu(\Delta_f) = \sum_{\substack{\alpha_i \neq \alpha_j \\ f(\alpha_i) = f(\alpha_j) = 0}} (\alpha_i|\alpha_j)_{\zeta} .\]
\end{lemma}
\begin{proof}
 This follows from $\nu(\alpha_i-\alpha_j) = (\alpha_i|\alpha_j)_{\zeta}$ and $\nu(\Delta_f) = \sum_{\substack{\alpha_i \neq \alpha_j \\ f(\alpha_i) = f(\alpha_j) = 0}} \nu(\alpha_i -\alpha_j)$.
\end{proof}

\subsection{Metric trees of replacement polynomials}
Our next task is to relate the metric trees of the replacement polynomials to the metric tree of $f$ (Theorem~\ref{mettreerep1} and Theorem~\ref{mettreerep2}). We make a few more definitions before stating the result.

We will continue to use the notation from Section~\ref{setup} in the rest of this section. Let $f,f_P^{\infty},f_P^{\neq \infty}$ be as in Section~\ref{induction} and Definition~\ref{defnreppoly}. Recall that type 2 points in $\mathbb{P}^{1,\mathrm{Berk}}_{\overline{K}}$ can be identified with divisorial valuations on $\overline{K}(\mathbb{P}^1)$. We identify the Gauss point $\zeta$ on $\mathbb{P}^{1,\mathrm{Berk}}_{\overline{K}}$ with the divisorial valuation corresponding to the generic point of the irreducible special fiber of $Y_0^f \cong \P^1_R$. Recall that in subsection~\ref{setup}, we picked $a_P \in k \subset R=k[[t]]$ for every point $P \in A$. The point $a_P$ can be identified with a type 1 point on $\mathbb{P}^{1,\mathrm{Berk}}_{\overline{K}}$. 
\begin{defn}
For any real number $l > 0$, let ${\mathsf{\color{blue}\zeta_P^l}}$ be the point on the unique path connecting the Gauss point $\zeta$ to the type 1 point $a_P$ in $\mathbb{P}^{1,\mathrm{Berk}}_{\overline{K}}$ that is at distance $l$ from $\zeta$.
\end{defn}

\subsubsection{Metric tree of $T(f_P^{\neq \infty})$}
Fix $P \in A_{\mathrm{bad}}$ such that $C_P^{\geq 1}$ is not empty. Let $T = T(f)$. Define a new tree $T'$ as follows. Since $T$ is a tree, $T \setminus \{\zeta_P^1\} = T_0 \bigsqcup_{i \geq 1} T_i$ is a disjoint union of subtrees $T_i$ of $T$, and $\zeta \in T_0$. Let $T' = \{\zeta_P^1\} \bigsqcup_{i \geq 1} T_i$. Then $T'$ is a connected subset of $T$, and therefore also a tree, and it inherits the metric $d$ from $T$. 

\begin{thm}\label{mettreerep1}
 The metric tree $T(f_P^{\neq \infty})$ is isomorphic to the tree $T'$ defined in the paragraph above.
\end{thm}
\begin{proof}
This follows from the following two observations:
\begin{itemize}
 \item The tree $T'$ is the convex hull of $\zeta_P^1$ and the roots of $g_i$ for $i \in C_P^{\geq 1}$.
 \item From Remark~\ref{rootscoefficients}, the collection of roots of $f_P^{\neq \infty}$ are simply the collection of roots of the polynomials $\tilde{g}_i$ for $i \in C_P^{\geq 1}$  divided by $t$. 
\end{itemize}
This means that the Newton-Puiseux expansion of the roots of each factor of $f_P^{\neq\infty}$ is obtained by dropping the leading term and then subtracting $1$ from all of the other exponents of the Newton-Puiseux expansions of the corresponding irreducible factors of $f$. The effect of dropping the leading term and shifting all exponents down by $1$ on the metric tree is deleting the initial segment between $\zeta$ and $\zeta_P^1$. 
\end{proof}

\subsubsection{Metric tree of $T(f_P^{\infty})$}
Fix $P \in A_{\mathrm{bad}}$ such that $C_P^{< 1}$ is not empty. 
\begin{defn}Let
 \[ {\color{blue}{\mathcal{V}_P}} \colonequals \{ \lambda_i/n_i  \ | \ g_i \in C_P^{<1}  \} \]
 be the collection of valuations of the roots of $\tilde{g_i}$ for the $g_i \in C_P^{<1}$. 
 \end{defn}
 
 \begin{defn}
  Let $a/b \in \mathcal{V}_P$ and assume $\gcd(a,b) = 1$. Let {\color{blue}{$S_{P,a/b}$}} be a subset of the roots of $f$ defined as follows:
  \[ {\color{blue}{S_{P,a/b}}} \colonequals \{ \lambda \ | \ g_i(\lambda) = 0 \ \textup{for some }\ g_i \in C_P^{<1} \ \textup{satisfying} \ \lambda_i/n_i = a/b \}. \]  
 \end{defn}

 \begin{defn}
  Let $a/b \in \mathcal{V}_P$ and assume $\gcd(a,b) = 1$. Define {\color{blue}{$T_{P,a/b}$}} to be the metric subtree of $T(f)$ obtained by taking the convex hull of $S_{P,a/b}$.
 \end{defn}
 
 \noindent{\textbf{Galois action on metric trees.}}\label{Galois} Let $G \colonequals \Gal(\overline{K}/K)$. Since the $G$ action on $\mathbb{P}^{1,\mathrm{Berk}}_{\overline{K}}$ fixes the Gauss point $\zeta$ and permutes the roots of any irreducible factor of $f \in K[x]$, we get natural induced $G$ actions on the metric trees $T(\tilde{g}_i),T_{P,a/b},T(f)$ for all $i$ and for all $a/b$. These actions preserve the lengths of edges and the valency at every vertex. Let $\mathcal{C}^{a/b} = \bigsqcup C_i$ be the connected components of $T_{P,a/b} \setminus \{ \zeta_P^{a/b} \}$, let $\overline{C_i} = C_i \cup \{ \zeta_P^{a/b} \}$ and let $\overline{\mathcal{C}^{a/b}}$ be the set of $\overline{C_i}$. 
 
\begin{thm}[Local symmetry of $T_{P,a/b}$ at $\zeta_P^{a/b}$]\label{mettreesymmetry}
Fix $i \in C_P^{<1}$. Let $g_i$ be an irreducible factor of $f$ and let $\tilde{g}_i$ be the shift of $g_i$ as defined in Section~\ref{setup}. Let $n_i = \deg g_i$, let $\omega$ be the chosen $n_i^{\mathrm{th}}$ root of unity in $K$ and let the valuation of any root of $\tilde{g}_i$ be $\lambda_i/n_i = a/b$ with $\gcd(a,b) = 1$. Let $d \colonequals n_i/b$. 
\begin{enumerate}[\upshape (a)]
\item The splitting field of $f$ is a cyclic extension of the form $K(t^{1/n})$ for some integer $n \geq 1$.
\item Let $\eta(t^{1/n_i}) \colonequals \sum_{l \geq 0} a_l t^{l/n_i}$ be the Newton-Puiseux expansion of one root of $\tilde{g}_i$. Then for any other root of $\tilde{g}_i$, there exists a unique integer $j$ with $0 \leq j \leq n_i-1$ such that the Newton-Puiseux expansion of this root is of the form  $\eta(\omega^j t^{1/n_i}) = \sum_{l \geq 0} a_l \omega^{jl} t^{l/n_i}$.
\item The point $\zeta_P^{a/b}$ of $T(f)$ lies on the subtree $T_{P,a/b}$, and is fixed by the natural $G$ action. Each $\overline{C_i}$ is a rooted metric tree with root $\eta_{P}^{a/b}$, and is the hull of $\zeta_{P}^{a/b}$ and a naturally defined subset of $S_{P,a/b}$. 
\item The elements of $\overline{\mathcal{C}^{a/b}}$ are in natural bijection with the coefficients of $t^{a/b}$ in the Newton-Puiseux expansions of the elements of $S_{P,a/b}$. Let $\sigma$ be a generator of the cyclic Galois group $G$ of the splitting field of $f$ over $K$, and let $G'$ be the subgroup generated by $\sigma^b$. For any $\overline{C}_i \in \overline{\mathcal{C}^{a/b}}$, the corresponding subset of $S_{P,a/b}$ is a union of $G'$ orbits for the action of $G'$ on $S_{P,a/b}$.
\item The $G$ action on $T_{P,a/b}$ induces a natural $\Z/b\Z$ action on the set of connected components $\overline{\mathcal{C}^{a/b}}$. If $f$ is irreducible with roots of valuation $a/b$, then the size of $\overline{\mathcal{C}^{a/b}}$ is $b$ and the natural $G$ action on $\overline{\mathcal{C}^{a/b}}$ is transitive. In general, every orbit for this action has size $b$, and the connected components in any given orbit are isomorphic as rooted metric subtrees of $T(f)$.
\item For any polynomial $g$, let $T_{P,a/b}^g$ denote the metric tree described above with the polynomial $g$ in place of the polynomial $f$. The metric tree $T_{P,a/b}^{g_i}$ for $i \in C_P^{<1}$ is isomorphic to a natural metric subtree of $T_{P,a/b}^f$, and $T_{P,a/b}^f$ is the union of the images of $T_{P,a/b}^{g_i}$ under these isomorphisms as we vary over all $i \in C_P^{<1}$.
\end{enumerate}   
\end{thm}
\begin{proof} \hfill 
 \begin{enumerate}[\upshape (a)]
 \item Let $L_i$ be the splitting field of $\tilde{g}_i$ over $K$. Since $n_i \leq 2g+2 < \cha k$ and $[L_i:K]$ divides $(n_i)!$, it follows that $L_i/K$ is a tame totally ramified Galois extension, and therefore cyclic (\cite[Chapter~IV, \S~1 Proposition~1, \S~2 Corollary~2 to Proposition~2]{serrelocal}). This also means that every subextension is Galois and cyclic. Since the residue field $k$ of $K$ is algebraically closed and $K$ is complete, all units in $K$ have $n_i^{\mathrm{th}}$ roots in $K$. Therefore by Kummer theory, it follows that $L_i = K(t^{1/n_i})$ is a cyclic extension and a generator of the Galois group sends $t^{1/n_i}$ to $\omega t^{1/n_i}$, where $\omega$ is a $n_i^{\mathrm{th}}$ root of unity in $K$. The splitting field of $f$ is the compositum of the $L_i$ and therefore equals $K(t^{1/\lcm(n_i)})$, which by the same argument as before is cyclic and Galois. 
 \item   The Galois group of $L_i/K$ is cyclic of order $n_i$ and is generated by the element $\sigma$ that sends $t^{1/n_i}$ to $\omega^i t^{1/n_i}$. Since the Galois group acts transitively on the roots of $\tilde{g}_i$, if $\alpha$ and $\beta$ are any two roots of $f$, then there is a unique $j$ with $0 \leq j \leq n_i-1$ with $\sigma^j(\alpha) = \beta$. If $\alpha = \sum_{l \geq 0} a_l t^{l/n_i}$, since $\sigma^j(\alpha) = \beta$, it follows that $\beta = \sum_{l \geq 0} a_l \omega^{jl} t^{l/n_i}$.  
 \item To show that $\zeta_P^{a/b}$ lies on $T_{P,a/b}$, it is enough to show that there exist two roots of $\tilde{g}_i$ for $i \in C_P^{<1}$ whose Newton-Puiseux expansions start with $t^{a/b}$, and such that the corresponding coefficients of $t^{a/b}$ are not congruent modulo the maximal ideal of $R$. From the previous paragraph and the fact that $\omega^{iad}$ is not congruent to $\omega^{jad}$ if $i \neq j \mod b$ (since $\cha k > 2g+2 \geq n_i = db$), we see that we can take any two roots of $\tilde{g}_i$ that begin with $ut^{a/b}$ and $u \omega^{ad} t^{a/b}$.
 
 Since the connected components of $T_{P,a/b} \setminus \zeta_P^{a/b}$ are in bijection with the coefficients of the leading terms of the Newton-Puiseux expansions of the elements of $S_{P,a/b}$ after subtracting $a_P$ and the $G$ action respects the metric tree structure of $T(\tilde{g}_i)$ and fixes $\zeta$, it follows that the $G$ action also fixes the point $\zeta_{P,a/b}$ and permutes the connected components of $T_{P,a/b} \setminus \zeta_P^{a/b}$. Since $\zeta_{P}^{a/b}$ lies in the closure of the connected component $C_i$, it follows that $\overline{C}_i = C_i \cup \{ \zeta_P^{a/b} \}$ is also connected and a rooted metric subtree of $T_{P,a/b}$. Since $T_{P,a/b}$ is the convex hull of $\zeta_{P}^{a/b}$ and a subset of Type I points $S_{P,a/b}$, the disjoint union decomposition $T_{P,a/b} \setminus \{\zeta_P^{a/b}\} = \bigsqcup C_i$ induces a corresponding disjoint union decomposition of the $S_{P,a/b}$.
 
 \item The edges adjacent to $\zeta_{P,a/b}$ in $T(g_i)$ (not counting the edge towards the Gauss point $\zeta$) are in bijection with the coefficients of the leading order term $t^{a/b}$ of the roots of $\tilde{g}_i$, so in particular, there are $b$ such edges. If we fix an irreducible factor $g_i$, then the roots of $g_i$ in a particular subtree (i.e, with a leading order term $ut^{a/b}$ for a fixed $u$), are precisely the roots in a given $G'$ orbit of a root, since $u\omega^{iad} \equiv u \mod t$ if and only $i \equiv 0 \mod b$. Taking a union over all irreducible factors of $f$ gives the desired result.
 
 \item The description of the $G$ action on the Newton-Puiseux expansions shows that the action on the coefficients of the leading order terms can be thought of as a permutation action of the $b^{\mathrm{th}}$ roots of unity in $K$ and therefore factors through the group $\Z/b\Z$. From the explicit description of the action, it follows that every orbit for this action has size $b$. In particular, these connected components in any given orbit are all isomorphic as rooted metric subtrees of $T(f)$. 
 
 \item These natural identifications arise from restricting the natural identifications of $T(g_i)$ (hull of $\zeta$ and the roots of the irreducible factor $g_i$ of $f$) with a metric subtree of $T(f)$ (hull of $\zeta$ and the roots of $f$). \qedhere
 \end{enumerate}
  \end{proof}
  
\subsection{Characteristic exponents of Newton-Puiseux expansions and metric trees}
We will now set up some notation to relate the Newton-Puiseux expansions of the roots of $f$ to the metric tree $T(f)$. 

\begin{lemma}\label{mettreerationalroots}
 Assume that the roots of $f$ are all $K$-rational. Let $\{a_1,a_2,\ldots,a_l\}$ be the chosen lifts in $R$ of the reduction of the roots of $f$ modulo $t$, and let $f(x) = h_1(x)h_2(x) \ldots h_l(x)$ be a factorization of $f$ such that for every $i$, every irreducible factor of $h_i$ specializes to $\overline{a}_i$ in $\P^1_R$. Let $h_i'(x) = h_i((x-a_i)/t)$. Let $S$ be the metric tree with vertices $\kappa,\kappa_1,\ldots,\kappa_l$ such that there is a single edge of length $1$ connecting $\kappa$ to $\kappa_i$ for every $i$ and no other edges. Then $T(f) \cong (S \bigsqcup_i T(h_i'))/\sim$ where the equivalence relation $\sim$ glues the point $\kappa_i$ to the point of $T(h_i')$ corresponding to the Gauss point, and under this isomorphism the Gauss point $\zeta$ in $T(f)$ gets identified with the point $\kappa$ of $S$.
\end{lemma}
\begin{proof}
 The proof is making the canonical identifications on $\mathbb{P}^{1,\mathrm{Berk}}_{\overline{K}}$ coming from our choice of Newton-Puiseux expansions explicit. The tangent directions from the Gauss point $\zeta$ in $T(f)$ are in bijective correspondence with the reductions of the roots of $f$ modulo $t$. Let $S'$ is the subset of $\mathbb{P}^{1,\mathrm{Berk}}_{\overline{K}}$ that includes the Gauss point $\zeta$ and the points $\zeta_{a_i}^1$ at distance $1$ from $\zeta$ in the direction corresponding to $a_i$ for every $i$. Then $S' \cong S$ and the roots of $f$ specialize to the ends $\kappa_i$ under the canonical retraction of points of $\mathbb{P}^{1,\mathrm{Berk}}_{\overline{K}}$ to $S'$, and the roots specializing to $\kappa_i$ are precisely the roots of $h_i$ for every $i$. The change of coordinates $x \mapsto (x-a_i)/t$ maps the roots of $h_i$ bijectively on to the roots of $h_i'$ and further induces an isomorphism of the hull of the roots of $h_i$ and $\zeta_{a_i}^1$ with the metric tree of $T(h_i'
)$. Since $T(f)$ can also be described as the hull of $\zeta,\zeta_{a_i}^1$ and the roots of $h_i$ for every $i$, this finishes the proof.
\end{proof}

\begin{lemma}\label{changeofscale}
 Let $s \in \overline{K}$ and let $\nu(s) = 1/n$ for some integer $n \geq 1$. If we let $(T_t(f),d_t)$ denote the metric tree from Definition~\ref{mettreedefn} and let $(T_{s}(f),d_s)$ denote the metric tree of $f$ constructed by using Newton-Puiseux expansions using $s$ instead of $t$. Then $T_s(f)$ and $T_t(f)$ are canonically homeomorphic and $d_s = n d_t$.
\end{lemma}
\begin{proof}
 Omitted. Similar to the proof of the previous lemma making canonical identifications explicit.
\end{proof}

We now recall certain definitions and theorems from \cite{GGP} that will let us relate the metric tree $T(f)$ to $T(f_P^{\infty})$. In \cite{GGP}, the authors relate the `essential exponents' and certain coefficients of the Newton-Puiseux expansions of roots of $f$ to those of its `inverse' obtained by reversing the roles of $x$ and $t$. Dividing the dual Newton-Puiseux expansions by $t$ gives us Newton-Puiseux expansions of roots of $f_P^{\infty}$. We will first recall the definition of characteristic exponents and essential exponents and show how these are related to symmetries of the metric tree $T(f)$ in Lemma~\ref{mettreesymmetry}. In Theorem~\ref{expofreplacement}, we will show how the essential exponents of each irreducible factor of the replacement polynomial can be derived from the essential exponents of the corresponding original irreducible factor. In the same theorem, we will also describe how metric trees for each irreducible factor of the replacement polynomial overlap. In the setting of \cite{
GGP}, the ring $R = \mathbb{K}[[t]]$, where $\mathbb{K}$ is an algebraically closed field of characteristic $0$. In our setting, we may have $\cha(\mathbb{K}) > 0$ but the relevant results still hold since we restrict our attention to polynomials of degree $< \cha \mathbb{K}$, which in turn ensures that the Newton-Puiseux expansions have bounded denominators. 
  
\begin{defn}\label{charexponents}
 Let $\eta \in \bigcup_{n \in \Z_{> 1}, (n,p)=1} R(t^{1/n})$. The {\textup{\textsf{\color{blue}{support $S(\eta)$ of $\eta$}}}}  is the set of nonnegative rational numbers with bounded denominators $S(\eta)$ such that $\eta$ has a Newton-Puiseux expansion of the form $\eta = \sum_{m \in S(\eta)} [\eta]_m t^m$ for the chosen lifts $[\eta]_m \in R \setminus \{0\}$. \newline
 Assume further that $0 \notin S(\eta)$. The {\textup{\textsf{\color{blue}{characteristic exponents $\mathcal{E}(\eta)$ of $\eta$}}}} consists of those elements of $S(\eta)$ which, when written as quotients of integers, need a denominator strictly bigger than the lowest common denominator of the previous exponents. That is:
 \[ \mathcal{E}(\eta) \colonequals \{ l \in S(\eta) \ | \ N_l l \notin \Z \}, \quad \textup{where} \quad N_l \colonequals \min \{ N \in \N \setminus \{0\} \ | \ (S(\eta) \cap [0,l)) \subset \frac{1}{N} \Z \}. \]
 The {\textup{\textsf{\color{blue}{sequence of characteristic exponents}}}} is the set of elements of $\mathcal{E}(\eta)$ written in increasing order.
\end{defn}

\begin{remark}{\textup{The sequence of characteristic exponents is finite for any $\eta$ as in the definition above as we assumed that the support of $\eta$ consists of a set of rational numbers with bounded denominators.}}\end{remark}

\begin{eg}
 Let $R = \C[[t]]$. Then $t^{5/2}+t^{8/3}$ and $2t-t^{5/2}+t^{8/3}-3t^{7/2}+t^{23/6}$ both have the same sequence of characteristic exponents namely $\{5/2,8/3\}$.
\end{eg}
 
\begin{defn}\label{essexponents} 
Consider a set $E \subset \Q_+$ with bounded denominators and an integer $p \in \N \setminus \{0\}$. Then the {\textup{\textsf{\color{blue}{sequence $\ess(E,p) \colonequals (\ess(E,p)_l)_l$ of essential elements of $E$ relative to $p$}}}} is defined inductively by:
\begin{itemize}
 \item $\ess(E,p)_0 \colonequals \min E$, and,
 \item if $l \geq 0$, then $\ess(E,p)_l$ is defined if and only if $E$ is not contained in the abelian subgroup $\Z \{p,\ess(E,p)_0,\ldots,\ess(E,p)_{l-1}\}$ of $\Q_+$ generated by $p,\ess(E,p)_0,\ldots,\ess(E,p)_{l-1}$ , and in this case
 \[ \ess(E,p)_l \colonequals \min (E \setminus \Z \{p,\ess(E,p)_0,\ldots,\ess(E,p)_{l-1}\}).  \]
\end{itemize} 
\end{defn}

In \cite[Lemma~3.13]{GGP}, they prove the following lemma relating the characteristic exponents and the essential exponents of a series $\psi \in R[t^{1/n}]$, that we recall for the reader's convenience.
\begin{lemma}\label{essandchar}
 Let $(\alpha_1,\alpha_2,\ldots,\alpha_g)$ be the sequence of characteristic exponents of a series $\psi \in R(t^{1/n})$. Then this sequence can be obtained from the sequence of essential exponents $(\epsilon_0,\epsilon_1,\ldots,\epsilon_d)$ of $\psi$ relative to $1$ in the following way.
 \begin{itemize}
  \item If $\epsilon_0 \notin \Z$, then $g=d+1$ and $\alpha_i = \epsilon_{i-1}$ for all $i \in \{1,2,\ldots,d+1\}$.
  \item If $\epsilon \in \Z$, then $g=d$ and $\alpha_i = \epsilon_i$ for all $i \in \{1,2,\ldots,d\}$.
 \end{itemize}
\end{lemma}

We will now state a theorem that tells us how to build the metric tree $T(f)$ of $f$ from the Newton-Puiseux expansions of the roots of $f$.

\begin{defn}
 We will use the notation introduced in Section~\ref{setup}. For each $P \in A_{\mathrm{bad}}$, let
 \[ \gamma_P^{\mathrm{max}} \colonequals \begin{cases}
                           \max \{\lambda_i/n_i \ | \ i \in C_P^{<1} \} &\quad \textup{if} \ C_P^{\geq 1} = \emptyset \\ 1 &\quad \textup{if} \ C_P^{\geq 1} \neq \emptyset.
                          \end{cases}
\]
Let $S_f$ be the convex hull of $\{ \zeta \} \cup \{ \zeta_P^{\gamma_P} \ | \ P \in A_{\mathrm{bad}} \}$.
\end{defn}

\begin{lemma}\label{retractionontoS}
Let $r \colon \mathbb{P}^{1,\mathrm{Berk}}_{\overline{K}} \rightarrow S_f$ denote the canonical retraction map. Under this retraction, any root of $g_i$ for $i \in C_P^{\geq 1}$ retracts to $\zeta_P^1$ and the roots in $S_{P,a/b}$ retract to $\zeta_P^{a/b}$ for every $a/b \in \mathcal{V}_P$.
\end{lemma}
\begin{proof}
 First pass to a cyclic extension $L = K(t^{1/n})$ to make all roots of $f$ rational. Then $n = kb$. If $s=t^{1/n}$, then the elements of $S_{P,a/b}$ are precisely the roots whose $s$-adic power series expansions begin with $s^{ka}$ and the roots of $\tilde{g}_i$ for $i \in C_P^{\geq 1}$ begin with $s^m$ for some $m \geq n$. We then combine Lemma~\ref{changeofscale} with a repeated application of Lemma~\ref{mettreerationalroots} to get the desired result.
\end{proof}

We will now show that the subtree $\overline{C}_i$ of $T_{P,a/b}$ from Lemma~\ref{mettreesymmetry} is naturally isomorphic to the metric tree of a polynomial over $K(t^{1/b})$. Recall that we proved in Lemma~\ref{mettreesymmetry}~(d) that every such subtree is the hull of $\zeta_P^{a/b}$ and a union of certain $G'$ orbits of roots of $f$. 
\begin{lemma}\label{canonicalsubtreeidentification}  Fix a subtree $\overline{C}_i$ of $T_{P,a/b}$ at $\zeta_{P}^{a/b}$ like in Section~\ref{Galois}, Lemma~\ref{mettreesymmetry}, and let $ut^{a/b}$ be the corresponding leading order term like in Lemma~\ref{mettreesymmetry}~(d). Pick a set of representatives $\alpha_1,\alpha_2,\ldots,\alpha_r$ for each $G'$ orbit of roots corresponding to $\overline{C}_i$, and let $\eta_1(t^{1/n_1}),\eta_2(t^{1/n_2}), \ldots, \eta_r(t^{1/n_r})$ be the corresponding Newton-Puiseux expansions of $\alpha_i-a_P$. Let $s = t^{1/b}$ and let $n_i' = n_i/b$. Fix $l$ with $1 \leq l \leq r$. 
 \begin{enumerate}[\upshape (a)]  
  \item  The minimal polynomial of $\alpha_l-a_P$ over $K(s)$ has degree $n_i/b$ and its roots have Newton-Puiseux expansions $\{ \eta_l( \omega^{bj} s^{1/{n_i'}}) \ | \ 0 \leq j \leq (n_i/b)-1 \}$. In particular, if the characteristic exponents of $\alpha_l-a_P$ over $K$ are $\{\frac{a_0}{b_0},\frac{a_1}{b_0b_1},\ldots,\frac{a_g}{b_0b_1 \ldots b_g}\}$ with $\gcd(b_i,a_i) = 1$, then $a_0=a,b_0=b$ and the characteristic exponents of its minimal polynomial over $K(s)$ are $\{\frac{a_1}{b_1},\ldots,\frac{a_g}{b_1 \ldots b_g}\}$.
  \item Let $j_l$ be the minimal polynomial of $(\alpha_l-a_P-us^a)/s^a$ over $K(s)$.  The characteristic exponents of the roots of $j_l$ over $K(s)$ are $\{\frac{a_1}{b_1}-a,\ldots,\frac{a_g}{b_1 \ldots b_g}-a\}$. Let $j = \prod_{l=1}^r j_l$ and let $(T(j),d)$ the corresponding metric tree over the field $K(s)$. Then $\overline{C}_i$ is isomorphic as a rooted metric tree to $(T(j),\frac{1}{b}d)$ (i.e., the point $\zeta_{P,a/b}$ maps to the Gauss point in $T(j)$).  
 \end{enumerate}
\end{lemma}
\begin{proof} \hfill
\begin{enumerate}
 \item In Lemma~\ref{mettreesymmetry}(d) we showed that if $\sigma$ is the generator of $G$, then $\sigma^b$ is the generator of $G'$ and that the subset $S_{P,a/b}$ of the roots of $f$ is a union of $G'$-orbits. The power series listed here are precisely the elements of the $G'$-orbit listed explicitly, and the computation of their essential exponents is a direct calculation.
 \item In Lemma~\ref{mettreesymmetry}(e,f) we showed that the convex hull of the roots of $f$ in $S_{P,a/b}$ is a metric tree $T_{P,a/b}$ rooted at $\zeta_{P,a/b}$, and that the connected components of $T_{P,a/b} \setminus \{\zeta_{P,a/b}\}$ are in bijective correspondence with the coefficients of the $t^{a/b}$ of the Newton-Puiseux expansions of the elements of $S_{P,a/b}$. Part(b) follows from part(a) and these identifications and Lemma~\ref{changeofscale}. \qedhere
\end{enumerate}  
\end{proof}

\begin{defn}\label{dualseries}[Dual series]
 Two units $\phi(t),\check{\phi}(t) \in k[[t]]^*$ are said to be {\textup{\textsf{\color{blue}{dual}}}} to each other if we have
 \[ t \phi (t \check{\phi}(t)) = t \quad \quad \quad \text{and} \quad \quad \quad t \check{\phi} (t \phi(t)) = t  .\]
\end{defn}
Dual series exist; the dual of $\phi(t)$ is the inverse image of $t$ under the continuous $k$-automorphism of $k[[t]]$ defined by sending $t$ to $t \phi(t)$. 

\begin{defn}
 Given two irreducible polynomials $g,g'$ of degrees $< \cha k$, the {\textup{\textsf{\color{blue}{maximal exponent of contact $\kappa_{g,g'}$ of $g$ and $g'$}}}} is defined to be 
 \[ \kappa_{g,g'} \colonequals \max\ \{ \nu(\alpha-\beta) \ | \ g(\alpha)=g'(\beta) = 0 \}.  \]
\end{defn}

\begin{lemma}\label{exponentoverlap}
 Let $g$ and $g'$ be irreducible polynomials in $R[x]$ such that $2 \leq n \colonequals \deg g < \cha k$, $2 \leq n' \colonequals \deg g' < \cha k$ and $\nu(g(0)) > 0$ and $\nu(g'(0)) > 0$. Let the essential exponents of $g$ and $g'$ be $\{e_0 \colonequals \frac{m}{n},e_1,\ldots,e_h\}$ and $\{e_0' \colonequals \frac{m'}{n'},e_1',\ldots,e_{h'}'\}$ respectively. Let $a_i,b_i$ for $0 \leq i \leq h$ be the positive integers uniquely defined by the relations $e_i = \frac{a_i}{b_1b_2 \ldots b_i}$ and $\gcd(b_i,a_i) = 1$. Assume that $e_r < \kappa \colonequals \kappa_{g,g'} \leq e_{r+1}$ for some $r > 0$. Then,
 \begin{enumerate}[\upshape (a)]
  \item $e_q = e_q'$ for all $q$ such that $0 \leq q \leq r$, 
  \item $\frac{m}{\gcd(m,ne_q)} = \frac{m'}{\gcd(m',n'e_q')} = \frac{a_1 b_2 \ldots b_q}{\gcd(a_1 b_2 \ldots b_q,a_q)}$ and $\frac{ne_q}{\gcd(m,ne_q)} = \frac{n'e_q'}{\gcd(m',n'e_q')} = \frac{a_q}{\gcd(a_1 b_2 \ldots b_q,a_q)}$ for all $q$ such that $0 \leq q \leq r$.
  \item Fix a root $\beta$ of $g'$. Then the multiset $\{ \nu(\alpha-\beta) \ | \ g(\alpha) = 0 \}$ consists of
  \begin{itemize}
   \item $e_q$ occuring $(b_q-1)b_{q+1} b_{q+2} \ldots b_h$ times for $0 \leq q \leq r$, and,
   \item $\kappa$ occuring $b_{r+1} b_{r+2} \ldots b_h$ times.
  \end{itemize}
 \end{enumerate}
\end{lemma}
\begin{proof} \hfill
 \begin{enumerate}[\upshape (a)]
  \item Let $\alpha \colonequals \sum_{q \in S(\alpha)} [\alpha]_q t^q$ and $\beta \colonequals \sum_{q \in S(\beta)} [\beta]_q t^q$ be roots of $g_i$ and $g_j$ respectively such that $\nu(\alpha-\beta) = \kappa$. Then $[\alpha]_q = [\beta]_q$ for all $q < \kappa$ and $[\alpha]_\kappa \neq [\beta]_\kappa$. In particular, $S(\alpha)_{< \kappa} \colonequals \{ q \ | \ [\alpha]_q \neq 0, q < \kappa \} = S(\beta)_{< \kappa} \colonequals \{ q \ | \ [\beta]_q \neq 0, q < \kappa \}$. Since the essential exponents of $\alpha$ that are less than $< \kappa$ only depend on the set $S(\alpha)_{< \kappa}$, it follows that $\alpha$ and $\beta$ have the same set of essential exponents less than $\kappa$, i.e, $e_q = e_q'$ for all $q$ such that $0 \leq q \leq r$.
  \item By the definition of essential exponents there exist positive integers $a_i,b_i,a_j',b_j'$ for $0 \leq i \leq g$, $0 \leq j \leq h$ such that 
  \begin{itemize}
   \item $e_i = \frac{a_i}{b_1b_2 \ldots b_i}$, $e_j' = \frac{a_j}{b_1b_2 \ldots b_j}$ for all $i,j$ such that $0 \leq i \leq g$, $0 \leq j \leq h$,
   \item $a_i = a_i',b_i = b_i'$ for all $0 \leq i \leq r$, and
   \item $n = b_1b_2 \ldots b_g,n' = b_1' b_2' \ldots b_h'$.
  \end{itemize}
  Since $m/n = a_1/b_1$ and $n = b_1b_2 \ldots b_g$, we have 
  \begin{align*} \frac{m}{\gcd(m,ne_q)} &= \frac{a_1 b_2 \ldots b_g}{\gcd(a_1 b_2 \ldots b_g, b_1 b_2 \ldots b_g \frac{a_q}{b_1 b_2 \ldots b_q})} \\
  &= \frac{a_1 b_2 \ldots b_q b_{q+1} b_{q+2} \ldots b_g}{\gcd(a_1 b_2 \ldots b_q b_{q+1} b_{q+2} \ldots b_g, b_{q+1} b_{q+2} \ldots b_g a_q)} \\
  &= \frac{a_1 b_2 \ldots b_q}{\gcd(a_1 b_2 \ldots b_q,a_q)} .\end{align*}
  A similar calculation shows $\frac{m'}{\gcd(m',n'e_q')} = \frac{a_1' b_2' \ldots b_q'}{\gcd(a_1' b_2' \ldots b_q',a_q')}$. Since $a_i = a_i'$ and $b_i = b_i'$ for all $i$ such that $0 \leq i \leq r$, it follows that $\frac{m}{\gcd(m,ne_q)} = \frac{m'}{\gcd(m',n'e_q')}$ for all $q \leq r$.
  We can similarly show that that for all $q \leq r$ we have
\[ \frac{ne_q}{\gcd(m,ne_q)} = \frac{a_q}{\gcd(a_1 b_2 \ldots b_q,a_q)} = \frac{a_q'}{\gcd(a_1' b_2' \ldots b_q',a_q')} = \frac{n'e_q'}{\gcd(m',n'e_q')}.\]
\item This proof can also be found in \cite[Proposition~4.1.3]{Wall}, but we are reproducing it here with our notation for the reader's convenience. Let $d$ be the smallest positive integer such that all roots of $g$ are defined over $k[[t^{1/n}]]$. Since $g$ is irreducible and $\deg g < \cha k$, it follows that $d=n$. By the definition of characteristic exponents, we also have $d = b_0b_1 \ldots b_h$. In the rest of the proof, we will freely use $n = b_0b_1 \ldots b_h$.

The proof of part(c) will be an inductive argument on $\deg(g)$ using Lemma~\ref{canonicalsubtreeidentification}~(b), as we now explain. Since the Galois group acts transitively on the roots of $g'$, it follows that for every root $\beta$ of $g'$, there exists a root $\alpha$ of $g$ such that $\nu(\alpha-\beta) = \kappa$. Fix such an $\alpha$ for the chosen $\beta$, and write down Newton-Puiseux expansions $\alpha(t^{1/n}) \colonequals \sum_{q \in S(\alpha)} [\alpha]_q t^q, \beta(t^{1/n'}) \colonequals \sum_{q \in S(\beta)} [\beta]_q t^q$. The other roots of $g$ have the form $\alpha(\omega^i t^{1/n})$ where $\omega$ is a chosen $n^{\mathrm{th}}$ root of unity and $i$ runs between $0$ and $n-1$. Since $\kappa > m/n = m'/{n'} = a_0/b_0$, we have $[\alpha]_{a_0/b_0} = [\beta]_{a_0/b_0}$. This in turn implies that the roots $\alpha'$ of $g$ with a leading order term different from that of $\beta$ (i.e, have $\nu(\alpha'-\beta) = a_0/b_0$) are those of the form $\{ \alpha(\omega^i t^{1/n}) \ | \ 0 \leq i \leq n-1,b_0 \nmid i \}$, and there are precisely 
$n-(n/b_0) = (b_
0-1)b_1b_2 \ldots b_h$ such roots.  

For the inductive step, we see that the roots $\alpha'$ of $g$ such that $\nu(\alpha'-\beta) > a_0/b_0$ are precisely those corresponding to the subtree $\overline{C}_i$ in Lemma~\ref{canonicalsubtreeidentification}. In the notation of that Lemma, our assumptions guarantee that $a_P = 0$, and $u = [\alpha]_{a_0/b_0} = [\beta]_{a_0/b_0}$. Recall that we proved in Lemma~\ref{canonicalsubtreeidentification}~(b) that the set $\{ (\alpha'-ut^{a_0/b_0})/t^{a_0/b_0} \ | \ g(\alpha') = 0, \ [\alpha']_{a_0/b_0} = [\alpha]_{a_0/b_0} = [\beta]_{a_0/b_0} \}$ are precisely the Galois conjugates of $(\alpha-ut^{a_0/b_0})/t^{a_0/b_0}$ over $k[[t^{1/b_0}]]$, and the essential exponents of the corresponding minimal polynomial are $\{ \frac{a_1}{b_1}, \frac{a_2}{b_1b_2}, \ldots, \frac{a_h}{b_1b_2 \ldots b_h} \}$.  Replacing $\beta$ by $(\beta-ut^{a_0/b_0})/t^{a_0/b_0}$ and the roots $\alpha'$ with $\nu(\alpha'-\beta) > a_0/b_0$ by $(\alpha'-ut^{a_0/b_0})/t^{a_0/b_0}$ and working over $K(t^{1/b_0})$ combined with the induction 
hypothesis gives us the desired result. \qedhere
 \end{enumerate}
\end{proof}

\begin{thm}\label{expofreplacement}
Let $h_i$ be the replacement polynomial for $g_i$ as in Definition~\ref{defnreppoly}. Let the sequence of characteristic exponents of any root of $g_i$ (which equals the sequence of characteristic/essential exponents of $\tilde{g}_i$) be $\{\frac{m_i}{n_i},e_1,\ldots,e_d\}$ with $n_i = \deg g_i$.
 \begin{enumerate}[\upshape (a)]
  \item The sequence of essential exponents of any root of $h_i$ relative to $1$ are $\{\frac{n_i}{m_i}-1,\frac{n_i}{m_i}(e_1+1)-2,\ldots,\frac{n_i}{m_i}(e_d+1)-2\}$.
  \item The replacement polynomials $h_i$ are irreducible. Let $i \neq j$. 
  \begin{itemize}
   \item If $m_i/n_i < m_j/n_j$, then $\kappa_{g_i,g_j} = m_i/n_i$ and $\kappa_{h_i,h_j} = (n_j/m_j)-1$.
   \item If $m_i/n_i = m_j/n_j$, then $\kappa_{h_i,h_j} = \frac{n_i}{m_i}(\kappa_{g_i,g_j}+1)-2$.
  \end{itemize}
  \item Let $P \in A_\mathrm{bad}, m/n \in \mathcal{V}_P$. Fix a subtree $(\overline{C},d)$ of $T_{P,m/n}$ corresponding to roots of $f(x-a_P)$ with leading order term $ut^{m/n}$ for some $u$ like in Lemma~\ref{mettreesymmetry}. Let $v \in R$ such that $v^m = u^n$. Then the subtree $\overline{D}$ of $T(f_P^{\mathrm{\infty}})$ obtained by taking the hull of $\zeta_0^{(n/m)-1}$ and the roots of $f_P^{\infty}$ with leading order term $vt^{(n/m)-1}$ is isomorphic to $(\overline{C},(n/m)d)$.
  \item Fix a set of representatives $\{u_1 t^{m/n}, \ldots, u_l t^{m/n} \}$ for the Galois orbits of leading order terms of roots of $f(x-a_P)$ of valuation $m/n$ like in Lemma~\ref{mettreesymmetry}~(e). For each such $u_i$, let $\overline{D}_{u_i}$ be the subtree of $T(f_P^{\infty})$ described in the previous part of the theorem. The subtree $T_{0,(n/m)-1}$ of $T(f_P^{\infty})$ is isometrically isomorphic to the tree obtained by gluing together the following subtrees ($lm$ in total) at the common point $\zeta_0^{(n/m)-1}$: for each value of $u_i$, take $m$ subtrees each isometrically isomorphic to the subtree $\overline{D}_{u_i}$.
 \end{enumerate}
\end{thm}
\begin{proof}\hfill
 \begin{enumerate}[\upshape (a)]
  \item Recall $R=k[[t]]$. Let $\eta(s) \in k[[s]]$ be such that $\eta(t^{1/n_i}) = ut^{m_i/n_i}+\ldots \in k[[t^{1/n_i}]]$ is the Newton-Puiseux expansion of a root of $\tilde{g}_i$, and let $u' \in R$ be such that $u'^{m_i} = u$, and let $\eta'(s) \in k[[s]]$ be such that $\eta'(0) = u'$ and $(s\eta'(s))^{m_i} = \eta(s)$. Note that these two equations uniquely define the power series $\eta'$. Let $u\xi'(u) \in k[[u]]$ be the dual series of $s\eta'(s)$, and let $\xi(u) = (u\xi'(u))^{n_i}$. 
  
  Let $\tilde{g}_i(x) = \sum_{j=0}^{\lambda_i} (\sum_{l} a_{jl}t^l) t^{\lambda_i-j}x^j + \sum_{j=\lambda_i+1}^{n_i} (\sum_{l} a_{jl}t^l) x^j$ like in Section~\ref{defreplacepoly}. Then $h_i^{\Diamond}(x) = \sum_{j=0}^{\lambda_i} (\sum_{l} a_{jl} x^l t^l) x^{\lambda_i-j} + \sum_{j=\lambda_i+1}^{n_i} (\sum_{l} a_{jl} x^l t^l) t^{j-\lambda_i}$. Let $h_i^{\Diamond'}(x) \colonequals t^{\lambda_i} h_i^{\Diamond}(x/t) = \sum_{j=0}^{\lambda_i} (\sum_{l} a_{jl} x^l) x^{\lambda_i-j} t^j + \sum_{j=\lambda_i+1}^{n_i} (\sum_{l} a_{jl} x^l) t^{j}$. Viewing $h_i^{\Diamond'}(x,t)$ and $\tilde{g}_i(t,x)$ as elements of $k[[t,x]]$, we see that $h_i^{\Diamond'}(t,x) = \tilde{g}_i(x,t)$. Since $\tilde{g}_i(\eta(t^{1/n_i}),t) = 0$ and $\eta'$ and $\xi'$ are dual series, the same argument as in \cite[Section~4.1]{GGP} shows that $h_i^{\Diamond'}(\xi(t^{1/m_i}),t) = 0$. Since $k[[t^{1/m_i}]]$ is a domain, it follows that $h_i^{\Diamond}(\xi(t^{1/m_i})/t) = 0$. Since $h_i^{\Diamond} = h_i u$ for some unit $u$ in $k[[t]][[x]]$, we 
also have $h_i(\xi(t^{1/m_i})/t) = 0$.
  
  By the Halphen-Stolz theorem (\cite[Corollary~4.5]{GGP}), we know that the essential exponents of $\xi(t^{1/m_i})$ are $\{\frac{n_i}{m_i},\frac{n_i}{m_i}(e_1+1)-1,\ldots,\frac{n_i}{m_i}(e_d+1)-1\}$ and therefore the essential exponents of $\xi(t^{1/m_i})/t$ are $\{\frac{n_i}{m_i}-1,\frac{n_i}{m_i}(e_1+1)-2,\ldots,\frac{n_i}{m_i}(e_d+1)-2\}$. 
  
  \item Let $\eta(t^{1/{n_i}}) \in k[[t^{1/n_i}]]$ be a root of $\tilde{g}_i$, and choose an $m_i^{\mathrm{th}}$ root of $a \in k$ of the coefficient of the leading order term of $\eta$. Then there exist coefficients $c_l \in k$ for $l > n_i$ such that 
  \[ \eta(t^{1/{n_i}}) = a^{m_i} t^{m_i/n_i} \left( 1 + \sum_{l > m_i} c_l t^{\frac{l-m_i}{n_i}} \right) .\]
  If $\xi(t^{1/m_i}) \colonequals \sum_{q \geq n_i} [\xi]_{q/m_i} t^{q/m_i}$, then we have the following formula for $[\xi]_{q/m_i}$ for $q \geq n_i$ from \cite[Propn~4.10]{GGP}.
  \[ [\xi]_{q/m_i} = \frac{n_i}{q} a^{-q} \left[ 1+ \sum_{i \geq 1} \binom{-q/{m_i}}{i} \left( \sum_{l > m_i} c_l t^{\frac{l-m_i}{n_i}} \right)^i \right]_{-1+ \frac{q}{n_i}} .\]
  The notation $[ \cdot ]_r$ denotes the coefficient of $t^r$ in the enclosed Puiseux series. Varying over all possible roots $\eta(t^{1/{n_i}})$ of $\tilde{g}_i$ and possible $m_i^{\mathrm{th}}$ roots $a$ gives us all possible Newton-Puiseux expansions $\xi(t^{1/m_i})/t$ of roots of $h_i$.
  
  Observe that to compute the coefficient $[\xi]_{q/m_i}$, the only terms that contribute are all integers $l$ in the range $m_i+1 \leq l \leq q+m_i-n_i$. Furthermore, for each $l$, only finitely many $i$ contribute (which can in turn be bounded in terms of $q$ and $l$). In particular for $l = q+m_i-n_i$, the only term that contributes is $i=1$. Let $\eta$ and $\eta'$ be two Newton-Puiseux series for roots of $\tilde{g}_i$ and $\tilde{g}_j$ respectively, with corresponding series $\xi$ and $\xi'$ constructed as above. 
  
  If $m_i/n_i < m_j/n_j$, then $(n_j/m_j) - 1 < (n_i/m_i) -1$ and therefore $\nu(\xi/t) = (n_i/m_i) - 1, \nu(\xi'/t) = (n_j/m_j) - 1$ and $\nu((\xi - \xi')/t) = (n_j/m_j) - 1$. In this case $\kappa_{g_i,g_j} = \max\{ \nu \left(\eta(\omega_{n_i} t^{1/n_i})-\eta'(\omega_{n_j} t^{1/n_j'}) \right) \ | \ \omega_{n_i}^{n_i} = \omega_{n_j}^{n_j} = 1 \} =  m_i/n_i$ and similarly $\kappa_{h_1,h_2} = (n_j/m_j)-1$. 
  
  Now assume that $m_i/n_i = m_j/n_j$. Let $\kappa \colonequals \nu(\eta-\eta') = \kappa_{g_1,g_2}$. There are two possibilities:
  \begin{enumerate}[\upshape (i)]
   \item $m_i/n_i = m_j/n_j = \kappa$
   \item $m_i/n_i = m_j/n_j < \kappa$.
  \end{enumerate}
  
  \noindent{\textbf{Case (i)}}: \\ If $m_i/n_i = m_j/n_j = \kappa$, we will show that $a^{n_i} \neq b^{n_j}$ for any $a,b$ such that $a^{m_i} = [\eta]_{m_i/n_i}$ and ${b}^{m_j} = [\eta']_{m_j/n_j}$.  This would in turn imply that $\nu((\xi - \xi')/t) = (n_i/m_i) - 1 = (n_j/m_j)-1$, and since $\eta,\eta',a,b$ are allowed to vary, this would imply that $\kappa_{h_i,h_j} = \frac{n_i}{m_i}(\kappa_{g_i,g_j}+1)-2$. 
  
  Let $m/n \colonequals m_i/n_i = m_j/n_j$ with $\gcd(m,n) = 1$ and $l_i \colonequals (m_i/m) = (n_i/n)$ and $l_j \colonequals (m_j/m) = (n_j/n)$. First we claim that the value of $a^{n_i}$ for $a$ such that $a^{m_i} = [\eta]_{m_i/n_i}$ only depends on the value of $a^{l_i}$, i.e., if $a,\tilde{a}$ both satisfy $a^{l_i} = \tilde{a}^{l_i} = c$ and $c^m = [\eta]_{m_i/n_i}$, then $a^{n_i} = (a^{l_i})^n = c^n = (\tilde{a}^{l_i})^n = \tilde{a}^{n_i}$. Therefore it is enough to prove that $a^{n} \neq a'^{n}$ for any $a,b$ such that $a^{m} = [\eta]_{m_i/n_i}$ and ${b}^{m} = [\eta']_{m_j/n_j}$, with the further assumption that $\gcd(m,n) = 1$.

  If $a^{n} = b^{n}$, then $a^{mn} = b^{mn}$, i.e, $[\eta]_{m/n}^n = [\eta']_{m/n}^n$. This in turn means that if we let $\omega$ be the $n^{\mathrm{th}}$ root of unity such that $\omega^{m} = [\eta']_{m/n}/[\eta]_{m/n}$ (possible since $m$ and $n$ are coprime, so $z \rightarrow z^m$ induces an automorphism of the set of $n^{\mathrm{th}}$ roots of unity), and let $\omega'$ be an $n_i^{\mathrm{th}}$ root of unity such that $\omega'^{l_i} = \omega$, then $\nu(\eta(\omega' t^{1/n_i}) - \eta'(t^{1/{n_i'}})) > \kappa = m_i/n_i = m_j/n_j$. This contradicts the definition of $\kappa$ since $\eta(\omega' t^{1/n_i})$ is also a root of $\tilde{g}_i$. This means that $a^{n_i} \neq b^{n_j}$ for any $a,b$ such that $a^{m_i} = [\eta]_{m_i/n_i}$ and ${b}^{m_j} = [\eta']_{m_j/n_j}$, which in turn implies that $\kappa_{h_i,h_j} = \frac{n_i}{m_i}(\kappa_{g_i,g_j}+1)-2$.

  \noindent{\textbf{Case (ii)}}: \\ Now assume that $m_i/n_i = m_j/n_j < \kappa$. Let \begin{align*} \eta(t^{1/n_i}) &\colonequals \eta(t^{1/{n_i}}) = c_0 t^{m_i/n_i} \left( 1 + \sum_{l > m_i} c_l t^{\frac{l-m_i}{n_i}} \right) \in k[[t^{1/n_i}]] \\ 
  \eta'(t^{1/n_j}) &\colonequals \eta'(t^{1/{n_j}}) = d_0 t^{m_j/n_j} \left( 1 + \sum_{l' > m_j} d_{l'} t^{\frac{{l'}-m_j}{n_j}} \right) \in k[[t^{1/n_j}]] \end{align*} be roots of $\tilde{g}_i$ and $\tilde{g}_j$ respectively such that $\nu(\eta-\eta') = \kappa$. This means that
  \begin{itemize}
   \item $\tilde{g}_i(\eta) = 0, \tilde{g}_j(\eta') = 0$,
   \item $c_0=d_0$,   
   \item if $l/n_i < \kappa$ and $(ln_j)/{n_i}$ is not an integer, then $c_l = 0$ and similarly if $l'/n_j < \kappa$ and $(l'n_i)/{n_j}$ is not an integer, then $d_{l'} = 0$,
   \item if $l$ and $l'$ are integers such that $l/n_i = l'/n_j < \kappa$, then $c_l = d_{l'}$, and,   
   \item $c_{\kappa n_i} \neq d_{\kappa n_j}$.
  \end{itemize}
  By Lemma~\ref{exponentoverlap}, there exists a unique index $r > 0$ such that $e_r < \kappa \leq e_{r+1}$ and $e_q' = e_q$ for all $q \leq r$, and such that $ \frac{m_i}{\gcd(m_i,n_ie_q)} = \frac{m_j}{\gcd(m_j,n_je_q')} \ \ \ \ \textup{for all} \ q \leq r$. Let $m_i/n_i = m_j/n_j = \tilde{m}/\tilde{n}$ such that $\gcd(\tilde{m},\tilde{n}) = 1$ and let
  \[ m \colonequals \lcm \left( \{\tilde{m}\} \cup \left\{\frac{m_i}{\gcd(m_i,n_ie_q)} \ \Bigg| \ 0 \leq q \leq r \right\} \right) = \lcm \left( \{\tilde{m}\} \cup \left\{\frac{m_j}{\gcd(m_j,n_je_q')} \ \Bigg| \ 0 \leq q \leq r \right\} \right) .\]
  Since $\tilde{m} \mid m_i$, $\tilde{m} \mid m_j$, we also have $m \mid m_i$ and $m \mid m_j$. Let $c \in k$ such that $c^m = c_0 = d_0$, and choose $a,b \in k$ such that $a^{m_i/m} = b^{m_j/m} = c$. Then $a^{m_i} = b^{m_j} = c_0 = d_0$. We can construct dual series $\xi$ and $\xi'$ for $\eta$ and $\eta'$ with these choices for $a,b$. We will now show that because of these careful choices of $a,b$ we have $\nu(\xi-\xi') = (\tilde{n}/\tilde{m})(\kappa+1)-1$. This would imply that $\kappa_{h_i,h_j} \geq \nu((\xi-\xi')/t) = \frac{\tilde{n}}{\tilde{m}}(\kappa_{g_i,g_j}+1)-2$. 
  
  We will first prove that $a^{n_ie_q} = b^{n_j e_q'}$ for all $q \leq r$ and then use this to prove $a^q = b^{q'}$ for any pair of integers $q,q'$ such that $q/m_i = q'/m_j < (\frac{\tilde{n}}{\tilde{m}})(\kappa+1)-1$. Since the definition of $m$ implies that $\frac{m\gcd(m_i,n_ie_k)}{m_i} = \frac{m\gcd(m_j,n_je_k)}{m_j}$ is an integer for $k \leq r$, it follows that for such $k$ we have
  \[ a^{\gcd(m_i,n_ie_k)} = a^{\frac{m_i}{m} \frac{m\gcd(m_i,n_ie_k)}{m_i}} = c^{\frac{m\gcd(m_i,n_ie_k)}{m_i}} = c^{\frac{m\gcd(m_j,n_je_k')}{m_j}} = b^{\frac{m_j}{m} \frac{m\gcd(m_j,n_je_k')}{m_j}} = b^{\gcd(m_j,n_je_k')} .\] 
  Since Lemma~\ref{exponentoverlap} implies that $\frac{n_ie_q}{\gcd(m_i,n_ie_q)} = \frac{n_je_q'}{\gcd(m_j,n_je_q')}$, we also have
  \[ a^{n_ie_q} =  a^{\gcd(m_i,n_ie_k) \frac{n_ie_q}{\gcd(m_i,n_ie_q)}} = b^{\gcd(m_j,n_je_k') \frac{n_ie_q}{\gcd(m_i,n_ie_q)}} = b^{\gcd(m_j,n_je_k') \frac{n_je_q'}{\gcd(m_j,n_je_q')}} = b^{n_je_q'}. \]
  We already have 
  \[ a^{m_i} = c^m = b^{m_j}.\] 
  Similarly, we also have
  \[ a^{n_i} = a^{\frac{m_i}{\tilde{m}} \tilde{n}} = a^{\frac{m_i}{m} \frac{m}{\tilde{m}} \tilde{n}} = c^{\frac{m}{\tilde{m}} \tilde{n}} =  b^{\frac{m_j}{m} \frac{m}{\tilde{m}} \tilde{n}} = b^{n_j}. \]
  Putting the last three equalities together, we have $a^{n_i(e_l+1)-m_i} = b^{n_j(e_l'+1)-m_j}$ for all $l \leq r$. For all $l \leq g$ and $l' \leq h$, let 
  \[ f_l \colonequals \frac{n_i}{m_i}(e_l+1) - 1, \quad \quad \text{and} \quad \quad f_{l'} \colonequals \frac{n_j}{m_j}(e_{l'}+1) - 1 .\]
  By part(a) of this lemma, we know that these are the essential exponents of $\xi$ and $\xi'$ respectively. Let $q,q'$ be integers such that $q/m_i = q'/m_j < (\tilde{n}/\tilde{m})(\kappa+1)-1$. Since $(\tilde{n}/\tilde{m})(\kappa+1)-1 \leq f_{r+1}$, by the definition of essential exponents, there exist integers $\lambda_0,\ldots,\lambda_r$ such that $q=\sum_{l=0}^r \lambda_l m_i f_q =\sum_{l=0}^r \lambda_l [n_i(e_l+1)-m_i]$. Since $q' = (m_jq_i)/m_i$, we also have $q'=\sum_{l=0}^r \lambda_l m_j f_q=\sum_{l=0}^r \lambda_l m_j f_q' =\sum_{l=0}^r \lambda_l [n_j(e_l'+1)-m_j]$. Since we already know $a^{n_i(e_l+1)-m_i} = b^{n_j(e_l'+1)-m_j}$ for all $l \leq r$, this implies that 
  \[ a^q = a^{\sum_{l=0}^r \lambda_l [n_i(e_l+1)-m_i]} = \prod_{l=0}^r a^{\lambda_l [n_i(e_q+1)-m_i]} = \prod_{l=0}^r b^{\lambda_l [n_j(e_q'+1)-m_j]} = b^{q'} .\]
  If $\kappa = e_{r+1}$, then we also have $\kappa=e_{r+1}'$, and a suitable modification of the above argument also shows that $a^{n_i(\kappa+1)-m_i} = b^{n_j(\kappa+1)-m_j}$.
  
  Now we are finally ready to prove that if $\xi \colonequals \sum_{l \in S(\xi)} [\xi]_l t^l,\xi' \colonequals \sum_{l \in S(\xi')} [\xi']_l t^l, \kappa' \colonequals (\tilde{n}/\tilde{m})(\kappa+1)-1$, then $[\xi]_l = [\xi']_l$ if $l < \kappa'$ and $[\xi]_{\kappa'} \neq [\xi']_{\kappa'}$. These equalities/inequalities now follow from the facts
  \begin{itemize}
   \item if $l/n_i < \kappa$ and $(ln_j)/{n_i}$ is not an integer, then $c_l = 0$ and similarly if $l'/n_j < \kappa$ and $(l'n_i)/{n_j}$ is not an integer, then $d_{l'} = 0$,
   \item if $l$ and $l'$ are integers such that $l/n_i = l'/n_j < \kappa$, then $c_l = d_{l'}$, and,   
   \item $c_{\kappa n_i} \neq d_{\kappa n_j}$,
  \end{itemize}
  the explicit formulae \cite[Propn~4.10]{GGP}.
  \[ [\xi]_{q/m_i} = \frac{n_i}{q} a^{-q} \left[ 1+ \sum_{i \geq 1} \binom{-q/{m_i}}{i} \left( \sum_{l > m_i} c_l t^{\frac{l-m_i}{n_i}} \right)^i \right]_{-1+ \frac{q}{n_i}} ,\]
  \[ [\xi']_{q'/m_j} = \frac{n_j}{q'} a^{-q'} \left[ 1+ \sum_{i \geq 1} \binom{-q'/{m_j}}{i} \left( \sum_{l' > m_j} d_{l'} t^{\frac{l'-m_j}{n_j}} \right)^i \right]_{-1+ \frac{q'}{n_j}} ,\]  
  and the observation that for a fixed $q$ (respectively $q'$),to compute the coefficient $[\xi]_{q/m_i}$, the only terms that contribute are all integers $l$ in the range $m_i+1 \leq l \leq q+m_i-n_i$, and furthermore, for each $l$, only finitely many $i$ contribute, that can in turn be bounded in terms of $q$ and $l$ (similar modifications for $l'$). In particular for $l = q+m_i-n_i$, the only term that contributes is $i=1$ (similarly for $l'$).
  
  We will now show that the inequality $\kappa_{h_i,h_j} \geq \nu((\xi-\xi')/t) = \frac{\tilde{n}}{\tilde{m}}(\kappa_{g_i,g_j}+1)-2$ is actually an equality. Since $r \mapsto (n/m)(r+1)-1$ is an increasing bijection $[0,\infty) \rightarrow [0,\infty)$ (with inverse bijection given by $r \mapsto (m/n)(r+1)-1$) that maps $\kappa$ to $(m/n)(\kappa+1)-1$, and since $(\eta,\eta') \mapsto (\xi,\xi')$ can be reversed to produce roots of $\tilde{g}_i, \tilde{g}_j$ from roots of $h_i,h_j$, if there are roots $\alpha,\beta$ of $h_i,h_j$ such that $\nu(\alpha-\beta) > \frac{n_i}{m_i}(\kappa_{g_i,g_j}+1)-2$, then we can produce a pair of power series that are roots of $\tilde{g}_i$ and $\tilde{g}_j$ respectively such that the $\nu$-adic valuation of the difference is higher than $\kappa$, which will contradict the definition of $\kappa$.  
  
  \item 
  Let 
  \[ S \colonequals \{ \eta(t^{1/{n'}})+a_P \ | \ f(\eta+a_P) = 0, \nu(\eta) = m/n, [\eta]_{m/n} = u \}. \]
  Then Lemma~\ref{mettreesymmetry}~[(c),(d)] tells us that $\overline{C}$ is the convex hull of $\zeta_P^{m/n}$ and $S$.
  
  Let $(g_i)_{i \in I}$ be the irreducible factors of $f$ such that $g_i(\alpha) = 0$ for some $\alpha \in S$, and let $n_i = \deg g_i$ as before. By Lemma~\ref{mettreesymmetry}~[(d),(e)], the cardinality of $S$ is $\prod_{i \in I} (n_i/n)$. 
  
  We need to show that we can mimic the compatible choice of $a,b$ for the construction of dual branches $\xi,\xi'$ in the previous part of this theorem, to show the following. First pick $v \in R$ such that $v^m = u^n$ as in the statement of the theorem.
  
  Let 
  \[ S' \colonequals \{ \xi(t^{1/{m'}}) \ | \ \nu(\xi) = n/m, [\xi]_{n/m} = v, \xi(t^{1/{m'}}) \textup{is dual to} \ \eta(t^{1/{n'}})\ \textup{for some} \ \eta+a_P \in S.   \} \]
  Then $\overline{D}$ is the the convex hull of $\zeta_0^{(n/m)-1}$ and $\{ \xi/t \ | \ \xi \in S'\}$. We need to show that $\overline{D}$ is a natural subtree of $T(f_P^{\infty})$ that is isomorphic as a metric tree to $(\overline{C},(n/m)d)$.

We will first show that $\left\vert S \right\vert = \left\vert S' \right\vert$. Let $h_i$ be the replacement polynomial for $g_i$ for every $i \in I$. Let $m_i \colonequals \deg h_i = (m/n) n_i$. Since $h_i$ is irreducible and $\deg h_i < \cha k$, by Lemma~\ref{mettreesymmetry}~(b) applied to $h_i$, there is a natural partition of the roots of $h_i$ into $m$ sets of size $(m_i/m)$ based on the coefficient of the leading order term of the root. The roots of $h_i$ as we vary over $i \in I$ with leading coefficient $v$ are precisely the elements of the form $\beta/t$ for some $\beta \in S'$. Therefore the cardinality of $S'$ is $\prod_{i \in I} (m_i/m)$. Since $\prod_{i=1}^l (n_l/n) = \prod_{i=1}^l (m_l/m)$, it follows that $\left\vert S \right\vert = \left\vert S' \right\vert$. 

Observe that $\overline{C}$ and $\overline{D}$ are both rooted trees with roots $\zeta_P^{m/n}$ and $\zeta_0^{(n/m)-1}$ respectively and are defined as convex hulls of the root and the sets $S,S'$ of the same cardinality of roots of $f,f_P^{\infty}$ respectively, with specified leading order terms $u,v$ respectively. Therefore to prove the claimed isomorphism of metric trees, it suffices to show that if there are elements $\eta_1+a_P,\eta_2+a_P,\ldots,\eta_l+a_P \in S$ such that $[\eta_i]_q = [\eta_j]_q$ for all $i \neq j$ and for all $q < \rho$ (this corresponds to a common segment of length $\rho-(m/n)$ in $\overline{C}$ in the unique path connecting $\zeta_P^{m/n}$ to $\eta_i+a_P$ as we vary over $i$ -- we have subtracted $m/n$ from $\rho$ to remove the length of the initial segment between $\zeta$ and $\zeta_P^{m/n}$ that all roots in $S$ share), then there exist corresponding $\xi_1,\ldots,\xi_l \in S'$ such that $[\xi_i]_q = [\xi_j]_q$ for all $i \neq j$ and for all $q < (n/m)(\rho+1)-1$ (this corresponds 
to a common segment of length $(n/m)(\rho-(m/n))$ in $\overline{D}$ in the unique path connecting $\zeta_0^{(n/m)-1}$ to $(\xi_i)/t$ as we vary over $i$ -- once again we have subtracted $n/m$ from $(n/m)(\rho+1)-1$ to remove the length of the initial segment between $\zeta$ and $\zeta_0^{n/m}$ that all elements in $S'$ share; dividing all elements of $S'$ by $t$ only translates the metric tree and moves $\zeta_0^{n/m}$ to $\zeta_0^{(n/m)-1}$ and does not change relative distances).

Let $e_i$ be the set of essential exponents of $\eta_1$, and assume that $e_r < \rho \leq e_{r+1}$. By Lemma~\ref{exponentoverlap}~(a), all the series $\eta_i$ have the same set of essential exponents $<\rho$, namely $e_0,e_1,e_2,\ldots,e_r$ and let $e_i = a_i/(b_0b_1 \ldots b_i)$ with $\gcd(a_i.b_i)=1$ for $i \leq r$ as in Lemma~\ref{exponentoverlap}. If $\eta_j+a_P$ is a root of $g_j$ and if $\sigma_j$ is a generator of the Galois group of the splitting field of $g_j$ over $K$, and $G_j$ is the subgroup generated by $\sigma_j^{b_0b_1 \ldots b_r}$, then by the definition of essential exponents and the explicit formula for the Galois action, we see that $[g(\eta_j+a_P)]_q = [\eta_j+a_P]_q$ for all $g \in G_j$ and for all $q < e_{r+1}$. So at the very beginning we may replace the set $\{ \eta_1,\ldots,\eta_l \}$ by this possibly larger Galois saturated set $\bigcup_{j=1}^l G_j \eta_j$ without loss of generality. Since the Galois group of the splitting field of $g_j$ acts transitively on the roots of $g_j$ 
without fixed points, the size of the $G_j$ orbit $G_j \eta_j$ is $\left\vert G_j \right\vert = \frac{\deg g_j}{(b_0 b_1 \ldots b_r)} = \frac{n_j}{b_0 b_1 \ldots b_r}$.
  
  First partition the set $\{ \eta_1,\ldots,\eta_l \}$ based on which irreducible factor $\tilde{g}_i$ the elements satisfy, and pick exactly one root for each irreducible factor to make a subset $(\eta_j)_{j \in J}$ of $\{ \eta_1,\ldots,\eta_l \}$. Then $\left\vert \bigsqcup_{j \in J} G_j \eta_j \right\vert =  \prod_{j \in J} \frac{n_j}{b_0b_1 \ldots b_r} = l$. 
 
  We will first construct dual series $\xi_j$ for $j \in J$.  Part(a) of this theorem tells us that the first $r+1$ essential exponents of $\xi_j$ are $\{\frac{b_0}{a_0}, \ldots, \frac{a_r+(b_0-a_0)b_1 \ldots b_r}{a_0b_1 \ldots b_r} \}$ and $\gcd(b_i,a_i+(b_0-a_0)b_1 \ldots b_i) = \gcd(b_i,a_i) = 1$. Using this fact and repeating the count in the previous paragraph tells us that including all Galois conjugates that share the same coefficients up to $(n/m)(e_r+1)-1$ (and therefore same coefficients up till $(n/m)(\rho+1)-1$) gives us $\prod_{j \in J} \frac{m_j}{a_0b_1\ldots b_r} = \prod_{j \in J} \frac{n_j a_0}{b_0} \frac{1}{a_0b_1\ldots b_r} = \prod_{j \in J} \frac{n_j}{b_0b_1 \ldots b_r}$ dual series,  and therefore a full set of dual series $\{ \xi_1,\ldots,\xi_l \}$ that is in bijection with the original set $\{ \eta_1,\ldots,\eta_l \}$. 
  
  To construct $\xi_j$ from $\eta_j$, we need to make a choice of $a_j \in k$ such that $a_j^{m_j} = [\eta_j]_{m/n}$. To ensure $[\xi_i]_q = [\xi_j]_q$ for all $i \neq j$ and for all $q < (n/m)(\rho+1)-1$, from the explicit formula for the dual series and mimicing the argument in the previous part(b) theorem, we need to ensure that $a_j^{qm_j}$ is independent of $j$ for all $q < (n/m)(\rho+1)-1$. Furthermore, the same argument as in part(b) tells us that it is enough to prove this for $q$ of the form $(n/m)(e_r+1)-1$ for all essential exponents $e_r$ of the $\eta_i$ such that $e_r < \rho$. (Recall that by Lemma~\ref{exponentoverlap}, the series $\eta_i$ have the same essential exponents less than $\rho$ as we vary over $i$). The main observation that makes the argument work is that the condition that we need to impose on the $a_i$ to ensure this coincidence (choosing an intermediate $c \in R$ and $m \in \Z$ dividing all the $m_i$ such that $c^m = [\eta_i]_{m/n}$  like in part(b) of this theorem) only depends 
on the value of the essential exponents of these series less than $\rho$ and is the exact same condition for multiple branches as it is for two branches.
  
\item The roots of $f_P^{\infty}$ with valuation $(n/m)-1$ have a leading order term of the form $vt^{(n/m)-1}$ where $v$ satisfies $v^m = u^n$ for some leading order term $ut^{m/n}$ of a root of $f(x-a_P)$. Since each value of $u^n$ corresponding to one Galois orbit for the action on the leading order terms of roots of $f(x-a_P)$ gives rise to $m$ distinct values of $v$, and the explicit formula for the Galois action (Lemma~\ref{mettreesymmetry}~(b)) tells us that these $m$ values get permuted transitively, combining this with Lemma~\ref{mettreesymmetry}~[(d),(e)] gives us the desired result. \qedhere
 \end{enumerate}
\end{proof}

Let $i \in C_P^{<1}$ and let $(\widetilde{n_i},\widetilde{\lambda_i})$ is the pair of integers associated to the replacement polynomial $h_i$ the same way $(n_i,\lambda_i)$ is associated to $f_i$.
\begin{corollary}\label{Cnewpair}
 $(\widetilde{n_i},\widetilde{\lambda_i}) = (\lambda_i,n_i-\lambda_i)$.
\end{corollary}
\begin{proof}
 Lemma~\ref{degrep} shows $\deg(h_i) = m_i$. The result now follows from Definition~\ref{essexponents} and Theorem~\ref{expofreplacement}(a).
\end{proof}

\subsection{Metric tree of the replacement polynomial $f_P^{\infty}$}
\begin{defn}
Let $\mathcal{D}$ be a multiset indexing pairs $(T_d,\gamma_d)$ for each $d \in \mathcal{D}$, where $T_d$ is a rooted metric tree and $\gamma_d \in \R$. Let $\gamma_{\mathrm{max}} \colonequals \sup_{d \in \mathcal{D}} \gamma_d$. Let $S$ be a directed line segment of length $\gamma_{\mathrm{max}}$ with starting point $O$. For $0 \leq r \leq \gamma_{\mathrm{max}}$, let $O^r$ be the unique point on $S$ at distance $r$ from $O$. The {\textup{\textsf{\color{blue}amalgamated tree}}} $T_{\mathcal{D}}$ of the multiset $\mathcal{D}$ is the rooted metric tree with root $O$ obtained by taking $(S \bigsqcup_{d \in \mathcal{D}} T_d)/\sim$ where the equivalence $\sim$ identifies the root of $T_d$ with the point $O^{\gamma_d}$ of $S$ for every $d \in \mathcal{D}$.
\end{defn}

\begin{defn}
 Given a metric tree $(T,d)$ and a real number $\alpha > 0$, the {\textup{\textsf{\color{blue}scaled metric tree}}} $T^{\alpha}$  is the metric tree $(T,\alpha d)$.
\end{defn}

Now fix notation as in Theorem~\ref{mettreesymmetry}. Let $a/b \in \mathcal{V}_P$ and let $I_{a/b}$ be the set of orbits for the $\Z/b\Z$ action on the collection of rooted metric trees $\overline{\mathcal{C}^{a/b}}$. For each $i \in I_{a/b}$, choose a rooted metric tree $T_i$ to represent the isomorphism class of rooted metric trees in the orbit corresponding to $i$. Let $\mathcal{D}$ be a multiset defined as follows.
\[ \mathcal{D} \colonequals \bigcup_{a/b \in \mathcal{V}_P, i \in I_{a/b}} \{ \underbrace{(T_i^{b/a},(b/a)-1),\ldots,(T_i^{b/a},(b/a)-1)}_{a \ \text{pairs}}   \}. \]

\begin{thm}\label{mettreerep2}  The metric tree $T(f_P^{\infty})$ is the amalgamated tree of the multiset $\mathcal{D}$ defined in the paragraph above.
\end{thm}
\begin{proof}
 The roots of $f_P^{\infty}$ all have valuation $(\frac{b}{a}-1) > 0$ for some $a/b \in \mathcal{V}_P$. Since $T(f_P^{\infty})$ is obtained by gluing together $T_{0,(\frac{b}{a}-1)}^{f_P^{\infty}}$, the result now follows from Theorem~\ref{expofreplacement}~(d).
\end{proof}

\section{Change in discriminant under replacement}\label{Sdiscchange}
Let $f = ut^bg_1 g_2 \ldots g_l$ where the $g_i$ are distinct monic irreducible polynomials in $R[x]$ and $b \in \{0,1\}$. Let $f_P^{\infty},f_P^{\neq \infty}$ be the replacement polynomials for each $P \in A_{\mathrm{bad}} \setminus \{\infty\}$ like in Definition~\ref{defnreppoly}. The goal of this section is to prove the following theorem. 
\begin{thm}\label{discchange} The quantity 
\[ \nu(\Delta_f)-\left( \sum_{\substack{P \in A \setminus A_{\mathrm{bad}} \\ P \neq \infty,g_i \in C_P}} (n_i-1) \right) - \sum_{\substack{P \in A_{\mathrm{bad}}\setminus \{\infty\} \\ \deg(f_P^\infty) \geq 1 }} \nu(\Delta_{f_P^{\infty}}) - \sum_{\substack{P \in A_{\mathrm{bad}} \setminus \{\infty\} \\ \deg(f_P^{\neq \infty}) \geq 1 }} \nu(\Delta_{f_P^{\neq \infty}}) \] equals 
\begin{multline*} 2b(d+\deg f -1) - \sum_{\substack{P \in A_\mathrm{bad} \\ P \neq \{\infty\}}} \left[ \overbrace{2b_P \left( d_P^{\mathrm{nod}}+b - 1 + \sum_{g_i \in C_P^{< 1}} \lambda_i \right) + \sum_{g_i \in C_P^{<1}} 2b(n_i-\lambda_i)}^{\textup{appears only if} \ \deg(f_P^\infty) \geq 1} + \overbrace{2b_P \left( d_P^{\mathrm{sm}} - 1+\sum_{g_i \in C_P^{\geq 1}} n_i \right)}^{\textup{appears only if} \ \deg(f_P^{\neq \infty}) \geq 1}  \right]
\\ + \sum_{\substack{P \in A_\mathrm{bad} \\ P \neq \{\infty\}}} \left[ \sum_{\substack{g_i \in C_P^{<1} \\ g_j \in C_P^{\geq 1}}} 2\lambda_i n_j +  \sum_{g_i \in C_P^{< 1}} \lambda_i \left(\lambda_i + \frac{n_i}{\lambda_i}-2 \right) + \sum_{\substack{i < j \\ g_i,g_j \in C_P^{< 1}}} 2  \lambda_i \lambda_j   + \sum_{g_i \in C_P^{\geq 1}} n_i(n_i-1) + \sum_{\substack{i < j \\ g_i,g_j \in C_P^{\geq 1}}} 2 n_i n_j \right]  . \end{multline*}
\end{thm}

\begin{lemma}\label{discnotinfty}
Let $P \in A_{\mathrm{bad}} \setminus \{\infty\}$ such that $\deg(f_P^{\neq \infty}) \geq 1$ (or equivalently $C_P^{\geq 1} \neq \emptyset$). Let $f_P^+ \colonequals \prod_{g_i \in C_P^{\geq 1}} g_i$. Then
\[ \nu(\Delta_{f_P^{\neq \infty}}) = 2b_P \left( d_P^{\mathrm{sm}} - 1 + \sum_{g_i \in C_P^{\geq 1}} n_i \right) + \sum_{\substack{\alpha \neq \alpha' \\ f_P^+(\alpha) = f_P^+(\alpha') = 0}} (\alpha | \alpha')_{\zeta} - \left[ \sum_{g_i \in C_P^{\geq 1}} n_i(n_i-1) + \sum_{\substack{i < j \\ g_i,g_j \in C_P^{\geq 1}}} 2 n_i n_j  \right]   \] 
\end{lemma}
\begin{proof}
 Let $h_P^+ \colonequals \prod_{g_i \in C_P^{\geq 1}} h_i(x)$. Then $f_P^{\neq \infty} = t^{b_P} h_P^+$ and $\nu(\Delta_{f_P^{\neq \infty}}) = 2b_P(d_P^{\mathrm{sm}}+\deg h_P^+ - 1) + \nu(\Delta_{h_P^+})$. By Definition~\ref{defnreppoly} and Remark~\ref{RAbaddeg} we have $\deg (h_i) = \deg (g_i) = n_i$ for every $g_i \in C_P^{\geq 1}$. Combining this with Lemma~\ref{gromov} we get
 \[ \nu(\Delta_{f_P^{\neq \infty}}) = 2b_P \left(d_P^{\mathrm{sm}} - 1+ \sum_{g_i \in C_P^{\geq 1}} n_i \right) + \sum_{\substack{\beta \neq \beta' \\ h_P^+(\beta) = h_P^+(\beta') = 0}} (\beta | \beta')_{\zeta} .\]
 Remark~\ref{rootscoefficients} and the fact that $\tilde{g}_i(x) = g_i(x+a_P)$ for $g_i \in C_P^{\geq 1}$ imply that the map $\alpha \mapsto \beta \colonequals (\alpha-a_P)/t$ induces a bijection from the roots of $f_P^+$ to the roots of $h_P^+$. This in turn means that if $(\alpha,\alpha')$ maps to the pair $(\beta,\beta')$, then $(\alpha|\alpha')_{\zeta} = (\beta|\beta')_{\zeta} + 1$.  Therefore 
 \begin{align*} \nu(\Delta_{f_P^{\neq \infty}}) &= 2b_P \left(d_P^{\mathrm{sm}} - 1 + \sum_{g_i \in C_P^{\geq 1}} n_i \right) + \sum_{\substack{\beta \neq \beta' \\ h_P^+(\beta) = h_P^+(\beta') = 0}} (\beta | \beta')_{\zeta} \\
 &= 2b_P \left( d_P^{\mathrm{sm}}- 1+\sum_{g_i \in C_P^{\geq 1}} n_i \right) + \sum_{\substack{\alpha \neq \alpha' \\ f_P^+(\alpha) = f_P^+(\alpha') = 0}} \left( (\alpha | \alpha')_{\zeta} - 1 \right) \\
 &= 2b_P \left( d_P^{\mathrm{sm}} - 1+\sum_{g_i \in C_P^{\geq 1}} n_i \right) + \sum_{\substack{\alpha \neq \alpha' \\ f_P^+(\alpha) = f_P^+(\alpha') = 0}} (\alpha | \alpha')_{\zeta} - \left[ \sum_{g_i \in C_P^{\geq 1}} n_i(n_i-1) + \sum_{\substack{i < j \\ g_i,g_j \in C_P^{\geq 1}}} 2 n_i n_j  \right]. \end{align*}
\end{proof}

\begin{lemma}\label{discinfty}
 Let $P \in A_{\mathrm{bad}} \setminus \{\infty\}$ such that $\deg(f_P^\infty) \geq 1$.
 %$C_P^{<1} \neq \emptyset$. 
 Let $f_P^- \colonequals \prod_{g_i \in C_P^{< 1}} g_i$. Then
 \begin{multline*} \nu(\Delta_{f_P^{\infty}}) = 2b_P \left( d_P^{\mathrm{nod}}+b - 1 + \sum_{g_i \in C_P^{< 1}} \lambda_i \right) + \sum_{g_i \in C_P^{<1}} 2b(n_i-\lambda_i) + \sum_{\substack{\alpha \neq \alpha' \\ f_P^-(\alpha) = f_P^-(\alpha') = 0}} (\alpha | \alpha')_{\zeta} \\ - \left[ \sum_{g_i \in C_P^{< 1}} \lambda_i \left(\lambda_i + \frac{n_i}{\lambda_i}-2 \right) + \sum_{\substack{i < j \\ g_i,g_j \in C_P^{< 1}}} 2  \lambda_i \lambda_j  \right]  .\end{multline*}
\end{lemma}
\begin{proof}
 Let $h_P^-(x) \colonequals \prod_{g_i \in C_P^{< 1}} h_i(x)$. Then $f_P^{\infty}(x) = t^{b_P} x^b h_P^-(x)$ and $\nu(\Delta_{f_P^{\infty}}) = 2b_P(d_P^{\mathrm{nod}}+\deg h_P^- + b - 1) + \nu(\Delta_{x^b h_P^-})$. By Definition~\ref{defnreppoly} and Lemma~\ref{degrep} we have $\deg (h_i) = \lambda_i$ for every $g_i \in C_P^{< 1}$. Combining this with Lemma~\ref{gromov} we get
 \[ \nu(\Delta_{f_P^{\infty}}) = 2b_P \left( d_P^{\mathrm{nod}}+b - 1 + \sum_{g_i \in C_P^{< 1}} \lambda_i \right) + \sum_{\xi \colon h_P^-(\xi) = 0} 2b(0|\xi)_{\zeta} + \sum_{\substack{\xi \neq \xi' \\ h_P^-(\xi) = h_P^-(\xi') = 0}} (\xi | \xi')_{\zeta}  .\]
 For each $g_i \in C_P^{< 1}$, Lemma~\ref{expofreplacement}~[(a),(b)] imply that $h_i$ is irreducible and that $\nu(\xi) = (\frac{n_i}{\lambda_i}-1)$ if $h_i(\xi) = 0$. Since $\deg (h_i) = \lambda_i$, it follows that
 \[ \sum_{\xi \colon h_P^-(\xi) = 0} 2b(0|\xi)_{\zeta} = \sum_{\xi \colon h_P^-(\xi) = 0} 2b\nu(\xi) = \sum_{g_i \in C_P^{<1}} \sum_{\xi \colon h_i(\xi) = 0} 2b\nu(\xi) = \sum_{g_i \in C_P^{<1}} 2b(n_i-\lambda_i).\]
  For integers $r,s > 0$, let 
 \[ f_{P,r} \colonequals \prod_{\substack{\eta \colon f_P^-(\eta+a_P) = 0 \\ \nu(\eta) = r }} (x-(\eta+a_P)),\ \textup{and.} \quad \quad \quad \quad h_{P,s} \colonequals \prod_{\substack{\xi \colon h_P^-(\xi) = 0 \\ \nu(\xi) = s }} (x-\xi).\] 
 Let $s,s'$ be integers with  $s' \leq s$. Since $\deg g_i = \lambda_i$, we have
 \begin{align*} \sum_{\substack{\xi \neq \xi' \\ h_{P,s}(\xi) = 0 \\ h_{P,s'}(\xi') = 0}} (\xi | \xi')_{\zeta} &=  \sum_{\substack{\xi \neq \xi' \\ h_{P,s}(\xi) = 0\\ h_{P,s'}(\xi') = 0}} \left[ s' + (\xi | \xi')_{\zeta_0^{s'}} \right] \\
  &= \sum_{\substack{g_i,g_j \in C_P^{<1} \\ \frac{n_i}{\lambda_i}-1=s \\ \frac{n_j}{\lambda_j}-1=s' }}  \sum_{\substack{\xi \neq \xi' \\ g_i(\xi) = 0 \\ g_j(\xi') = 0}} s' +  \sum_{\substack{\xi \neq \xi' \\ h_{P,s}(\xi) = 0\\ h_{P,s'}(\xi') = 0}} (\xi | \xi')_{\zeta_0^{s'}} \\
  &= \begin{cases} \displaystyle\sum_{\substack{g_j \in C_P^{<1} \\ \frac{n_j}{\lambda_j}-1=s' }}  (\lambda_j - 1) ( n_j-\lambda_j ) + \displaystyle\sum_{\substack{g_i,g_j \in C_P^{<1}, i \neq j \\ \frac{n_i}{\lambda_i}-1=s \\ \frac{n_j}{\lambda_j}-1=s' }}  2\lambda_i \left( n_j-\lambda_j \right) +  \displaystyle\sum_{\substack{\xi \neq \xi' \\ h_{P,s}(\xi) = 0\\ h_{P,s'}(\xi') = 0}} (\xi | \xi')_{\zeta_0^{s'}} \quad & \textup{if}\ s=s'  \\ 
  \displaystyle\sum_{\substack{g_i,g_j \in C_P^{<1}, i \neq j \\ \frac{n_i}{\lambda_i}-1=s \\ \frac{n_j}{\lambda_j}-1=s' }}  2\lambda_i \left( n_j-\lambda_j \right) +  \displaystyle\sum_{\substack{\xi \neq \xi' \\ h_{P,s}(\xi) = 0\\ h_{P,s'}(\xi') = 0}} (\xi | \xi')_{\zeta_0^{s'}} \quad & \textup{if}\ s \neq s'.\end{cases}
\end{align*}
Similarly if $r \leq r'$, we have
\[ \sum_{\substack{\eta \neq \eta' \\ f_{P,r}(\eta) = 0 \\ f_{P,r'}(\eta') = 0}} (\eta | \eta')_{\zeta} = 
\begin{cases}
\displaystyle\sum_{\substack{g_j \in C_P^{<1} \\ \frac{\lambda_j}{n_j}=r }}  \lambda_j ( n_j-1 ) + \displaystyle\sum_{\substack{g_i,g_j \in C_P^{<1}, i \neq j \\ \frac{\lambda_i}{n_i}=r , \frac{\lambda_j}{n_j}=r' }}  2\lambda_i n_j +  \displaystyle\sum_{\substack{\eta \neq \eta' \\ f_{P,r}(\eta) = 0\\ f_{P,r'}(\eta') = 0}} (\eta | \eta')_{\zeta_P^{r}} \quad & \textup{if}\ r=r'  \\ 
\displaystyle\sum_{\substack{g_i,g_j \in C_P^{<1}, i \neq j \\ \frac{\lambda_i}{n_i}=r , \frac{\lambda_j}{n_j}=r' }}  2\lambda_i n_j +  \displaystyle\sum_{\substack{\eta \neq \eta' \\ f_{P,r}(\eta) = 0\\ f_{P,r'}(\eta') = 0}} (\eta | \eta')_{\zeta_P^{r}} \quad & \textup{if}\ r \neq r' .
\end{cases}
\]
Let $g_i \in C_P^{<1}$ and let $\eta,\xi$ satisfy $g_i(\eta+a_P) = 0$ and $h_i(\xi) = 0$. If $r \colonequals \nu(\eta)$ and $s \colonequals \nu(\xi)$ then Lemma~\ref{expofreplacement} implies that $s = \frac{1}{r}-1$.  Let $r,r'$ be integers with $r \leq r'$ and let $s \colonequals \frac{1}{r}-1$ and $s' \colonequals \frac{1}{r'}-1$.

\noindent\textbf{Key claim:}  
\[ \sum_{\substack{\xi \neq \xi' \\ h_{P,s}(\xi) = 0\\ h_{P,s'}(\xi') = 0}} (\xi | \xi')_{\zeta_0^{s'}} = \sum_{\substack{\eta \neq \eta' \\ f_{P,r}(\eta) = 0\\ f_{P,r'}(\eta') = 0}} (\eta | \eta')_{\zeta_P^{r}}.\]

For $r,r',s,s'$ as above we see that the claim along with some algebra implies that
\[\displaystyle\sum_{\substack{\eta \neq \eta' \\ f_{P,r}(\eta) = 0 \\ f_{P,r'}(\eta') = 0}} (\eta | \eta')_{\zeta} - \displaystyle\sum_{\substack{\xi \neq \xi' \\ h_{P,s}(\xi) = 0 \\ h_{P,s'}(\xi') = 0}} (\xi | \xi')_{\zeta} = \begin{cases}
 \displaystyle\sum_{\substack{g_j \in C_P^{<1} \\ \frac{\lambda_j}{n_j}=r }} \lambda_j \left( \lambda_j + \frac{n_j}{\lambda_j} -2 \right) + 
 \displaystyle\sum_{\substack{g_i,g_j \in C_P^{<1}, i \neq j \\ \frac{\lambda_i}{n_i}=r , \frac{\lambda_j}{n_j}=r' }}   2\lambda_i \lambda_j 
 \quad & \textup{if}\ r = r' \\ 
 \displaystyle\sum_{\substack{g_i,g_j \in C_P^{<1}, i \neq j \\ \frac{\lambda_i}{n_i}=r , \frac{\lambda_j}{n_j}=r' }}  2\lambda_i \lambda_j \quad & \textup{if}\ r \neq r' .
\end{cases}\]
Adding these over all possible $r,r'$ then finishes the proof of the lemma. 

Now we prove the key claim. If $r \neq r'$, then $s \neq s'$ and both sides of the equality are $0$ by Lemma~\ref{retractionontoS}. So from now on we may assume that $r=r' \colonequals a/b$ with $\gcd(a,b) = 1$ and $s=s'$ and that there exist roots of $f(x+a_P)$ of valuation $r$ (otherwise both sides are $0$ since we are summing over the empty set). Like in Theorem~\ref{mettreerep2}, let $I_{a/b}$ be the set of orbits for the $\Z/b\Z$ action on the collection of rooted metric trees $\overline{\mathcal{C}^{a/b}}$. For each $\gamma \in I_{a/b}$, choose a rooted metric tree $T_\gamma$ to represent the isomorphism class of rooted metric trees in the orbit corresponding to $\gamma$ and let $\Gamma$ be the corresponding set of roots of $f$ (i.e, $T_{\gamma}$ is the convex hull of $\zeta_{P}^{a/b}$ and $\Gamma$). Since the Galois action induces an isometry of $\mathbb{P}^{1,\mathrm{Berk}}_{\overline{K}}$, we have
\[ \sum_{\substack{\eta \neq \eta' \\ f_{P,r}(\eta) = 0 \\ f_{P,r'}(\eta') = 0}} (\eta | \eta')_{\zeta_P^r} = \sum_{\gamma \in  I_{a/b}} \sum_{\substack{\eta \neq \eta' \\ \eta,\eta' \in \Gamma}} b (\eta | \eta')_{\zeta_P^r}. \]
Theorem~\ref{expofreplacement}~(c) gives us a scaled metric $\overline{D}_{\gamma}$ of $T(f_P^{\infty})$ that is homeomorphic to $\overline{C}_{\gamma}$ but where all lengths are scaled by $b/a$. We similarly have a set of roots $\Delta$ of $f_P^{\infty}$ corresponding to $\overline{D}_{\gamma}$ of valuation $(b/a)-1$. The Galois action on the subtree $T_{0,(b/a)-1}$ of $T(f_P^{\infty})$ factors via $\Z/a\Z$, and the description in Theorem~\ref{expofreplacement}~(d) tells us that the corresponding set of orbits for the action on the subtrees are still indexed by $I_{a/b}$ with $a$ elements in each orbit. So we now have 
\[ \sum_{\substack{\xi \neq \xi' \\ h_{P,s}(\xi) = 0 \\ h_{P,s'}(\xi') = 0}} (\xi | \xi')_{\zeta_0^s} = \sum_{\gamma \in  I_{a/b}} a \cdot \left( \sum_{\substack{\xi \neq \xi' \\ \xi,\xi' \in \Delta}} (\xi | \xi')_{\zeta_0^s} \right) = \sum_{\gamma \in  I_{a/b}} a \cdot (b/a) \cdot \left( \sum_{\substack{\eta \neq \eta' \\ \eta,\eta' \in \Gamma}} (\eta | \eta')_{\zeta_P^r} \right) = \sum_{\substack{\eta \neq \eta' \\ f_{P,r}(\eta) = 0 \\ f_{P,r'}(\eta') = 0}} (\eta | \eta')_{\zeta_P^r}, \]
where the second to last equality relies on the scaled isomorphism between $\overline{C}_{\gamma}$ and $\overline{D}_{\gamma}$.
\end{proof}

Now we can prove the main theorem of this section.
\begin{proof}[Proof of Theorem~\ref{discchange}]
 Lemma~\ref{gromov} implies that
 \begin{align*} \nu(\Delta_f) &= 2b(d+\deg f -1) + \nu(\Delta_{g_1 g_2 \ldots g_l}) \\ &= 2b(d+\deg f -1) + \sum_{\substack{\alpha \neq \alpha' \\ f(\alpha) = f(\alpha') = 0}} (\alpha|\alpha')_{\zeta}. \end{align*}
 If $P \in A \setminus A_{\mathrm{bad}}$ and $P \neq \infty$, then Lemma~\ref{badpoints} implies that there exists a unique index $i$ such that $C_P = \{g_i\}$ and for that $i$, either $g_i \in C_P^{\geq 1}$ and $n_i = 1$ or $g_i \in C_P^{< 1}, n_i > 1,\lambda_i = 1$ and 
 \[ \sum_{\substack{\alpha \neq \alpha' \\ g_i(\alpha) = g_i(\alpha') = 0}} (\alpha|\alpha')_{\zeta} =  \sum_{\substack{\alpha \neq \alpha' \\ g_i(\alpha) = g_i(\alpha') = 0}} \nu(\alpha-\alpha') = \left( \sum_{\substack{\alpha \neq \alpha' \\ g_i(\alpha) = g_i(\alpha') = 0}} \frac{1}{n_i} \right)= n_i-1.\] 
 Since $(\alpha|\alpha')_{\zeta} = \nu(\alpha-\alpha')$, if $\alpha$ and $\alpha'$ specialize to distinct points in $\mathbb{P}^1_R$, then $(\alpha|\alpha')_{\zeta} = 0$. So we can write the above formula for $\nu(\Delta_f)$ as 
 \begin{align*} \nu(\Delta_f) &= 2b(d+\deg f -1) + \sum_{\substack{P \in A \setminus A_{\mathrm{bad}} \\ P \neq \infty}} \sum_{g_i \in C_P} \sum_{\substack{\alpha \neq \alpha' \\ g_i(\alpha) = g_i(\alpha') = 0}} (\alpha|\alpha')_{\zeta} + \sum_{P \in A_\mathrm{bad}} \sum_{g_i,g_j \in C_P} \sum_{\substack{\alpha \neq \alpha' \\ g_i(\alpha) = g_j(\alpha') = 0}} (\alpha|\alpha')_{\zeta} \\ 
 &= 2b(d+\deg f -1) + \left( \sum_{\substack{P \in A \setminus A_{\mathrm{bad}} \\ P \neq \infty}} n_P-1 \right)  \\
 &\quad \quad \quad \quad \quad \quad + \sum_{P \in A_\mathrm{bad}} \left[ \sum_{\substack{\alpha \neq \alpha' \\ g_i,g_j \in C_P^{<1} \\ g_i(\alpha) = g_j(\alpha') = 0}} (\alpha|\alpha')_{\zeta} + \sum_{\substack{\alpha \neq \alpha' \\ g_i \in C_P^{<1},g_j \in C_P^{\geq 1} \\ g_i(\alpha) = g_j(\alpha') = 0}} 2(\alpha|\alpha')_{\zeta} + \sum_{\substack{\alpha \neq \alpha' \\ g_i,g_j \in C_P^{\geq 1} \\ g_i(\alpha) = g_j(\alpha') = 0}} (\alpha|\alpha')_{\zeta} \right].
 \end{align*}
 Since 
 \[
  \sum_{\substack{g_i \in C_P^{<1},g_j \in C_P^{\geq 1} \\ g_i(\alpha) = g_j(\alpha') = 0}} (\alpha|\alpha')_{\zeta} 
  = \sum_{\substack{g_i \in C_P^{<1} \\ g_j \in C_P^{\geq 1}}} \sum_{\substack{g_i(\alpha) = 0 \\ g_j(\alpha') = 0}} \nu(\alpha-\alpha') \\
  = \sum_{\substack{g_i \in C_P^{<1} \\ g_j \in C_P^{\geq 1}}} \sum_{\substack{g_i(\alpha) = 0 \\ g_j(\alpha') = 0}} \nu(\alpha-a_P) \\
  = \sum_{\substack{g_i \in C_P^{<1} \\ g_j \in C_P^{\geq 1}}} \lambda_i n_j,
 \]
the equality further simplifies to
\begin{multline*} \nu(\Delta_f) = 2b(d+\deg f -1) + \left( \sum_{\substack{P \in A \setminus A_{\mathrm{bad}} \\ P \neq \infty,g_i \in C_P}} n_i-1 \right) \\ + \sum_{P \in A_\mathrm{bad}} \left[ \sum_{\substack{\alpha \neq \alpha' \\ g_i,g_j \in C_P^{<1} \\ g_i(\alpha) = g_j(\alpha') = 0}} (\alpha|\alpha')_{\zeta} + \sum_{\substack{g_i \in C_P^{<1} \\ g_j \in C_P^{\geq 1}}} 2\lambda_i n_j + \sum_{\substack{\alpha \neq \alpha' \\ g_i,g_j \in C_P^{\geq 1} \\ g_i(\alpha) = g_j(\alpha') = 0}} (\alpha|\alpha')_{\zeta} \right] .\end{multline*}
Rewriting $\displaystyle\sum_{P \in A_\mathrm{bad}} \displaystyle\sum_{\substack{\alpha \neq \alpha' \\ g_i,g_j \in C_P^{<1} \\ g_i(\alpha) = g_j(\alpha') = 0}} (\alpha|\alpha')_{\zeta}$ and $\displaystyle\sum_{P \in A_\mathrm{bad}} \displaystyle\sum_{\substack{\alpha \neq \alpha' \\ g_i,g_j \in C_P^{\geq 1} \\ g_i(\alpha) = g_j(\alpha') = 0}} (\alpha|\alpha')_{\zeta}$ using Lemma~\ref{discinfty} and Lemma~\ref{discnotinfty} we get that 
\[ \nu(\Delta_f) -\left( \sum_{\substack{P \in A \setminus A_{\mathrm{bad}} \\ P \neq \infty, g_i \in C_P}} (n_i-1) \right) - \sum_{\substack{P \in A_{\mathrm{bad}} \setminus \{\infty\} \\ \deg(f_P^\infty) \geq 1 }} \nu(\Delta_{f_P^{\infty}}) - \sum_{\substack{P \in A_{\mathrm{bad}} \setminus \{\infty\} \\ \deg(f_P^{\neq \infty}) \geq 1 }} \nu(\Delta_{f_P^{\neq \infty}})\] equals
\begin{multline*} 2b(d+\deg f -1) + \sum_{P \in A_\mathrm{bad}} \left[ \sum_{\substack{g_i \in C_P^{<1} \\ g_j \in C_P^{\geq 1}}} 2\lambda_i n_j - \overbrace{2b_P \left(d_P^{\mathrm{sm}} - 1 + \sum_{g_i \in C_P^{\geq 1}} n_i \right)}^{\textup{appears only if} \ \deg(f_P^{\neq \infty}) \geq 1} + \sum_{g_i \in C_P^{\geq 1}} n_i(n_i-1) + \sum_{\substack{i < j \\ g_i,g_j \in C_P^{\geq 1}}} 2 n_i n_j \right] \\
 + \sum_{P \in A_\mathrm{bad}} \left[  \underbrace{- 2b_P \left( d_P^{\mathrm{nod}}+b - 1 + \sum_{g_i \in C_P^{< 1}} \lambda_i \right) - \sum_{g_i \in C_P^{<1}} 2b(n_i-\lambda_i)}_{\textup{appears only if} \ \deg(f_P^{ \infty}) \geq 1} + \sum_{g_i \in C_P^{< 1}} \lambda_i \left(\lambda_i + \frac{n_i}{\lambda_i}-2 \right) + \sum_{\substack{i < j \\ g_i,g_j \in C_P^{< 1}}} 2  \lambda_i \lambda_j   \right]
   , \end{multline*}
 which on rearrangement gives the desired equality.
\end{proof}

\section{Proof of inequality}\label{finalproof}
\subsection{Change in conductor is less than change in discriminant}
In this section, we will combine Theorem~\ref{condchange} and Theorem~\ref{discchange} to establish the key inductive inequality Theorem~\ref{replacementchange}, namely, that the change on the conductor side under the replacement operation is less than the change on the discriminant side. The proof of Theorem~\ref{replacementchange} then follows from an application of the following simple numerical inequalities and some careful book-keeping.
\begin{lemma}\label{convexity} Let $a_1,a_2,\ldots,a_l$ be a finite set of integers, each $\geq 1$ with $\sum_{i=1}^l {a_i} \geq 2$. 
\begin{enumerate}[\upshape(a)]
 \item $\sum_{i} a_i(a_i-1) + 2 \sum_{i < j}a_i a_j \geq 2$.
 \item If $\sum_{i} {a_i}$ is odd, then $\sum_{i} a_i(a_i-3) + 2 \sum_{i < j}a_i a_j \geq 0$.
 \item If $\sum_i a_i$ is even, then equality holds in {\upshape{(a)}} if and only if one of the following holds:
 \begin{itemize}
  \item $l=1$ and $a_1=2$, or,
  \item $l=2$ and $a_1=a_2=1$. 
 \end{itemize}
 \item If $\sum_i a_i$ is odd, then equality holds in {\upshape{(b)}} if and only if one of the following holds:
 \begin{itemize}
  \item $l=1$ and $a_1=3$, or,
  \item $l=2$ and $\{a_1,a_2\} = \{1,2\}$, or,
  \item $l=3$ and $a_1=a_2=a_3=1$.
 \end{itemize}

\end{enumerate} 
\end{lemma}
\begin{proof} 
Let $\sum a_i = S$. We then have
\[ \sum_{i} a_i(a_i-1) + 2 \sum_{i < j}a_i a_j = S(S-1)\]
\[ \sum_{i} a_i(a_i-3) + 2 \sum_{i < j}a_i a_j = S(S-3)\]
If $S \geq 2$, then $S(S-1) \geq 2$. Furthermore, if $S$ is even and $S(S-1)=2$, then we have $\sum a_i = S=2$. Similarly, if $S \geq 3$, then $S(S-3) \geq 0$ and $S(S-3) = 0$ when $\sum a_i = S=3$. The other two parts follow since the $a_i$ are nonnegative integers. \qedhere
\end{proof}

We are now finally ready to prove Theorem~\ref{replacementchange}.
\begin{proof}[Proof of Theorem~\ref{replacementchange}]
 Theorem~\ref{condchange} tells us that the quantity
\[ -\Art (X^f/S)- \left( \sum_{\substack{P \in A_{\mathrm{bad}} \setminus \{\infty\} \\ \deg(f_P^{\infty}) \geq 1 }}-\Art (X^{f_P^{\infty}}/S) +\sum_{\substack{P \in A_{\mathrm{bad}} \setminus \{\infty\} \\ \deg(f_P^{\neq \infty}) \geq 1 }}-\Art (X^{f_P^{\neq \infty}}/S) \right) \]
equals 
\begin{multline*} -b(2+d)+\displaystyle\sum_{\substack{P \in A \setminus A_{\mathrm{bad}} \\ P \neq \infty}} (n_i-1+b)+(2+b)\sharp(A_{\mathrm{bad}})+\sum_{P \in A_\mathrm{bad} \setminus \{\infty\}}  \sum_{g_i \in C_P^{<1}} \left( n_i - \lambda_i \right) \\
  - \left(\sum_{\substack{P \in A_{\mathrm{bad}} \setminus \{ \infty \} \\ \deg f_P^{\infty} = 0}} (b-bb_P) \right)+\displaystyle\sum_{\substack{P \in A_{\mathrm{bad}} \setminus \{ \infty \} \\ \deg f_P^{\infty} \geq 1 \ \textup{and} \\ \deg f_P^{\neq \infty} \geq 1}} 2b_P-\displaystyle\sum_{\substack{P \in A_\mathrm{bad} \setminus \{\infty\} \\ \deg f_P^{\infty} \geq 1}} (b+2b_P  d_P^{\mathrm{nod}}) -  \displaystyle\sum_{\substack{P \in A_\mathrm{bad} \setminus \{\infty\} \\ \deg f_P^{\neq \infty} \geq 1}} 2b_P  d_P^{\mathrm{sm}}. \end{multline*}

 Theorem~\ref{discchange} and the equality $\lambda_i(\lambda_i+ \frac{n_i}{\lambda_i} - 2) = (n_i-\lambda_i) + \lambda_i(\lambda_i-1)$ tell us that the quantity
 \[\nu(\Delta_f) -\left( \sum_{\substack{P \in A \setminus A_{\mathrm{bad}} \\ P \neq \infty, g_i \in C_P}} (n_i-1) \right) - \sum_{\substack{P \in A_{\mathrm{bad}} \setminus \{\infty\} \\ \deg(f_P^\infty) \geq 1 }} \nu(\Delta_{f_P^{\infty}}) - \sum_{\substack{P \in A_{\mathrm{bad}} \setminus \{\infty\} \\ \deg(f_P^{\neq \infty}) \geq 1 }} \nu(\Delta_{f_P^{\neq \infty}})\] 
 equals 
\begin{multline*} 2b(d+\deg f -1) - \sum_{\substack{P \in A_\mathrm{bad} \\ P \neq \{\infty\}}} \left[ \overbrace{2b_P \left( d_P^{\mathrm{nod}}+b - 1 + \sum_{g_i \in C_P^{< 1}} \lambda_i \right) + \sum_{g_i \in C_P^{<1}} 2b(n_i-\lambda_i)}^{\textup{appears only if} \ \deg(f_P^{\infty}) \geq 1 } + \overbrace{2b_P \left(d_P^{\mathrm{nod}} - 1+ \sum_{g_i \in C_P^{\geq 1}} n_i \right)}^{\textup{appears only if} \ \deg(f_P^{\neq \infty}) \geq 1}  \right]
\\ + \sum_{\substack{P \in A_\mathrm{bad} \\ P \neq \{\infty\}}} \left[ \sum_{\substack{g_i \in C_P^{<1} \\ g_j \in C_P^{\geq 1}}} 2\lambda_i n_j +  \sum_{g_i \in C_P^{< 1}} \left[ (n_i - \lambda_i) + \lambda_i (\lambda_i -1) \right]  + \sum_{\substack{i < j \\ g_i,g_j \in C_P^{< 1}}} 2  \lambda_i \lambda_j   + \sum_{g_i \in C_P^{\geq 1}} n_i(n_i-1) + \sum_{\substack{i < j \\ g_i,g_j \in C_P^{\geq 1}}} 2 n_i n_j \right]  . \end{multline*}

\noindent\textbf{Case I: $\mathbf{b= 0}$} \\
In this case, we have to prove that
\begin{multline*} \displaystyle\sum_{\substack{P \in A \setminus A_{\mathrm{bad}} \\ P \neq \infty}} (n_i-1)+2\sharp(A_{\mathrm{bad}})+\sum_{P \in A_\mathrm{bad} \setminus \{\infty\}} \sum_{g_i \in C_P^{<1}} \left( n_i - \lambda_i \right)   +\displaystyle\sum_{\substack{P \in A_{\mathrm{bad}} \setminus \{ \infty \} \\ \deg f_P^{\infty} \geq 1 \ \textup{and} \\ \deg f_P^{\neq \infty} \geq 1}} 2b_P \\
-\displaystyle\sum_{\substack{P \in A_\mathrm{bad} \setminus \{\infty\} \\ \deg f_P^{\infty} \geq 1}} 2b_P  d_P^{\mathrm{nod}} -  \displaystyle\sum_{\substack{P \in A_\mathrm{bad} \setminus \{\infty\} \\ \deg f_P^{\neq \infty} \geq 1}} 2b_P  d_P^{\mathrm{sm}} \end{multline*}
  is less than or equal to
  \begin{multline*} \sum_{\substack{P \in (A \setminus A_{\mathrm{bad}}) \\ P \neq \infty,g_i \in C_P}} (n_i-1) - \sum_{\substack{P \in A_\mathrm{bad} \\ P \neq \infty}} \left[ \overbrace{2b_P \left(d_P^{\mathrm{nod}}- 1 + \sum_{g_i \in C_P^{< 1}} \lambda_i \right) }^{\textup{appears only if} \ \deg(f_P^{\infty}) \geq 1} + \overbrace{2b_P \left( d_P^{\mathrm{sm}} - 1+ \sum_{g_i \in C_P^{\geq 1}} n_i \right)}^{\textup{appears only if} \deg(f_P^{\neq \infty}) \geq 1}  \right] +   \sum_{\substack{P \in A_\mathrm{bad} \\ P \neq \infty}} \sum_{g_i \in C_P^{< 1}} (n_i - \lambda_i)
\\  + \sum_{\substack{P \in A_\mathrm{bad} \\ P \neq \infty}} \left[ \sum_{\substack{g_i \in C_P^{<1} \\ g_j \in C_P^{\geq 1}}} 2\lambda_i n_j +  \sum_{g_i \in C_P^{< 1}} \lambda_i (\lambda_i -1)  + \sum_{\substack{i < j \\ g_i,g_j \in C_P^{< 1}}} 2  \lambda_i \lambda_j   + \sum_{g_i \in C_P^{\geq 1}} n_i(n_i-1) + \sum_{\substack{i < j \\ g_i,g_j \in C_P^{\geq 1}}} 2 n_i n_j \right]  . \end{multline*}

Since $b=0$, by definition of $A_{\mathrm{bad}}$ it follows that $\infty \notin A_{\mathrm{bad}}$. Therefore, it suffices to show that for each $P \in A_{\mathrm{bad}}$, we have the inequality
\begin{multline}\label{Eeqeven}
\begin{split}
 2 + \overbrace{2b_P}^{\textup{appears only if} \ \deg(f_P^{\infty}) \geq 1 \ \textup{and} \  \deg(f_P^{\neq \infty}) \geq 1} - \overbrace{2b_Pd_P^{\mathrm{nod}}}^{\textup{appears only if} \ \deg(f_P^{\infty}) \geq 1}-\overbrace{2b_Pd_P^{\mathrm{sm}}}^{\textup{appears only if} \deg(f_P^{\neq \infty}) \geq 1} \\
 \leq \sum_{\substack{g_i \in C_P^{<1} \\ g_j \in C_P^{\geq 1}}} 2\lambda_i n_j +  \sum_{g_i \in C_P^{< 1}} \lambda_i (\lambda_i -1)  + \sum_{\substack{i < j \\ g_i,g_j \in C_P^{< 1}}} 2  \lambda_i \lambda_j   + \sum_{g_i \in C_P^{\geq 1}} n_i(n_i-1) + \sum_{\substack{i < j \\ g_i,g_j \in C_P^{\geq 1}}} 2 n_i n_j\\
 -\underbrace{2b_P \left(d_P^{\mathrm{nod}}- 1 + \sum_{g_i \in C_P^{< 1}} \lambda_i \right) }_{\textup{appears only if} \ \deg(f_P^{\infty}) \geq 1} - \underbrace{2b_P \left( d_P^{\mathrm{sm}} - 1+\sum_{g_i \in C_P^{\geq 1}} n_i \right)}_{\textup{appears only if} \deg(f_P^{\neq \infty}) \geq 1} 
\end{split}
 \end{multline}
Note that when $b=0$, it follows from Definition~\ref{defnreppoly}, Lemma~\ref{degrep} and Remark~\ref{rootscoefficients} that $\deg(f_P^{\infty}) = \sum_{g_i \in C_P^{<1}} \lambda_i$ and $\deg(f_P^{\neq \infty}) = \sum_{g_i \in C_P^{\geq 1}} n_i$. Since $P \in A_{\mathrm{bad}}$, we have $\sum_{g_i \in C_P^{<1}} \lambda_i + \sum_{g_i \in C_P^{\geq 1}} n_i \geq 2$. If $b_P=0$, the inequality now follows from Lemma~\ref{convexity}(a) applied to $\{n_1,n_2,\ldots \} \cup \{\lambda_1,\lambda_2,\ldots \}$, and furthermore note that the left hand side of Theorem~\ref{replacementchange} is strictly positive. In this case, we note from Lemma~\ref{convexity}(c) that we have equality only if 
\[ \widetilde{\wt_P} = \sum_{i \in C_P^{\geq 1}} n_i + \sum_{i \in C_P^{<1}} \lambda_i=2.\] 

If $b_P=1$, then by Corollary~\ref{fmultexc}, $\sum_{g_i \in C_P^{<1}} \lambda_i + \sum_{g_i \in C_P^{\geq 1}} n_i$ is odd. Since $d_P^{\mathrm{nod}}$ is the parity of $\deg(f_P^{\infty}) = \sum_{g_i \in C_P^{<1}} \lambda_i$ and  $d_P^{\mathrm{sm}}$ is the parity of $\deg(f_P^{\neq \infty}) = \sum_{g_i \in C_P^{\geq 1}} n_i$, it follows that $\{ d_P^{\mathrm{nod}},d_P^{\mathrm{sm}} \} = \{0,1\}$, and the left hand side of our inequality is $2$ if $\deg(f_P^{\infty}) \geq 1$ and $\deg(f_P^{\neq \infty}) \geq 1$, and $0$ otherwise. In this case, the inequality now follows from Lemma~\ref{convexity}(b) applied to $\{n_1,n_2,\ldots \} \cup \{\lambda_1,\lambda_2,\ldots \}$. For the purpose of Corollary~\ref{Cequality}, we note from Lemma~\ref{convexity}(d) that we have equality only if 
\[ \widetilde{\wt_P} = \sum_{i \in C_P^{\geq 1}} n_i + \sum_{i \in C_P^{<1}} \lambda_i=3.\]
Note that the left hand side of Theorem~\ref{replacementchange} is nonnegative, and the right hand side is $0$ only when $\sum_{g_i \in C_P^{<1}} \lambda_i + \sum_{g_i \in C_P^{\geq 1}} n_i = 3$ and one of $\deg(f_P^{\neq \infty})$ and $\deg(f_P^{\infty})$ is $0$. \\

\noindent\textbf{Case II: $\mathbf{b= 1}$} \\

When $b=1$, by Definition~\ref{defnreppoly}, we have $\deg(f_P^{\infty}) \geq 1$ for all $P \in A_\mathrm{bad}$ and therefore 
\begin{equation}\label{Eo1}\left(\sum_{\substack{P \in A_{\mathrm{bad}} \setminus \{ \infty \} \\ \deg f_P^{\infty} = 0}} (b-bb_P) \right) = 0.\end{equation}
Since $(\sum_{i \in C_P^{<1}} \lambda_i + \sum_{i \in C_P^{\geq 1}} n_i) \geq 1$ for all $P \in A$ and since $b=1$, Lemma~\ref{badpoints} implies that $A = A_{\mathrm{bad}}$ and therefore 
\begin{equation}\label{Eo2}\displaystyle\sum_{\substack{P \in A \setminus A_{\mathrm{bad}} \\ P \neq \infty}} (n_i-1+b) = 0.\end{equation}
For each $P \in A_{\mathrm{bad}} \setminus \{\infty\}$, we have
\begin{equation}\label{Eo3} \sum_{i \in C_P^{<1}} 2n_i+\sum_{g_i \in C_P^{\geq 1}} 2n_i - \sum_{i \in C_P^{<1}} 2(n_i-\lambda_i) = \sum_{g_i \in C_P^{\geq 1}} 2n_i +  \sum_{g_i \in C_P^{<1}} 2\lambda_i. \end{equation}
Since $b=1$, by Lemma~\ref{Lparitydegree} we have $\infty \in A_{\mathrm{bad}}$ exactly when $d=1$. This implies that
\begin{equation}\label{Eo4} -bd+b \sharp(A_{\mathrm{bad}})-\sum_{\substack{P \in A_{\mathrm{bad}} \setminus \{\infty\} \\ \deg(f_P^{\infty})=0}}b-\sum_{\substack{P \in A_{\mathrm{bad}} \setminus \{\infty\} \\ \deg(f_P^{\infty}) \geq 1}}b = 0.\end{equation}
Note we also have 
\begin{equation}\label{Eo5} \deg f = \sum_{P \in A \setminus \{\infty\}} \left( \sum_{g_i \in C_P^{<1}} n_i + \sum_{g_i \in C_P^{\geq 1}} n_i \right),\end{equation}
Using equations~\ref{Eo1},\ref{Eo2},\ref{Eo3},\ref{Eo4} and \ref{Eo5} and arguing as in the case $b=0$, to prove Theorem~\ref{replacementchange} when $b=1$, it now suffices to prove that the following inequality holds for each $P \in A_{\mathrm{bad}} \setminus \{\infty\}$. (Recall that we showed $\deg(f_P^{\infty}) \geq 1$ for all $P \in A_{\mathrm{bad}} \setminus \{\infty\}$.)
\begin{multline}\label{Eeqodd}
\begin{split}
 2 + \overbrace{2b_P}^{\textup{appears only if} \  \deg(f_P^{\neq \infty}) \geq 1} - 2b_Pd_P^{\mathrm{nod}}-\overbrace{2b_Pd_P^{\mathrm{sm}}}^{\textup{appears only if} \deg(f_P^{\neq \infty}) \geq 1} \\
 \leq \sum_{\substack{g_i \in C_P^{<1} \\ g_j \in C_P^{\geq 1}}} 2\lambda_i n_j +  \sum_{g_i \in C_P^{< 1}} \lambda_i (\lambda_i -1)  + \sum_{\substack{i < j \\ g_i,g_j \in C_P^{< 1}}} 2  \lambda_i \lambda_j   + \sum_{g_i \in C_P^{\geq 1}} n_i(n_i-1) + \sum_{\substack{i < j \\ g_i,g_j \in C_P^{\geq 1}}} 2 n_i n_j\\
 \sum_{g_i \in C_P^{\geq 1}} 2n_i +  \sum_{g_i \in C_P^{<1}} 2\lambda_i-2b_P \left(d_P^{\mathrm{nod}}- 1 + \sum_{g_i \in C_P^{< 1}} \lambda_i \right)  - \underbrace{2b_P \left( d_P^{\mathrm{sm}} - 1+\sum_{g_i \in C_P^{\geq 1}} n_i \right)}_{\textup{appears only if} \deg(f_P^{\neq \infty}) \geq 1} 
\end{split}
 \end{multline}

Since $b_P \in \{0,1\}$ and by Lemma~\ref{fmultexc} we have $b_P=1$ if and only if $1+\sum_{i \in C_P^{\geq 1}} n_i+\sum_{i \in C_P^{< 1}} \lambda_i$ is odd and since $A=A_{\mathrm{bad}}$, has to be at least $3$. If $b_P=0$ (equivalently $\sum_{i \in C_P^{\geq 1}} n_i+\sum_{i \in C_P^{< 1}} \lambda_i$ is odd), the desired inequality now follows from Lemma~\ref{convexity}(a) if $\sum_{i \in C_P^{\geq 1}} n_i+\sum_{i \in C_P^{< 1}} \lambda_i \geq 2$ and from the inequality 
\[ 2\leq  \sum_{g_i \in C_P^{\geq 1}} 2n_i +  \sum_{g_i \in C_P^{<1}} 2\lambda_i \]
if $\sum_{i \in C_P^{\geq 1}} n_i+\sum_{i \in C_P^{< 1}} \lambda_i = 1$. In both cases, note that the left hand side of Theorem~\ref{replacementchange} is strictly positive. Similarly, if $b_P=1$ (equivalently $\sum_{i \in C_P^{\geq 1}} n_i+\sum_{i \in C_P^{< 1}} \lambda_i$ is even, and therefore $\geq 2$), then
\[ \sum_{g_i \in C_P^{\geq 1}} 2n_i +  \sum_{g_i \in C_P^{<1}} 2\lambda_i-2b_P \left(\sum_{g_i \in C_P^{< 1}} \lambda_i \right)  - 2b_P \left(\sum_{g_i \in C_P^{\geq 1}} n_i \right) = 0 \]
and therefore the desired inequality follows from Lemma~\ref{convexity}(a) as before. Once again note that the right hand side of Theorem~\ref{replacementchange} is strictly positive. Observe that we have equality in Equation~\ref{Eeqodd} precisely when $\widetilde{\wt_P} = \sum_{i \in C_P^{\geq 1}} n_i + \sum_{i \in C_P^{<1}} \lambda_i$ is either $1$ or $2$.

Finally note that if $b=0$, we have $\wt_P \in \{2,3\}$ precisely when $\widetilde{\wt_P} \in \{2,3\}$, and that if $b=1$, we have $\wt_P \in \{2,3\}$ precisely when $\widetilde{\wt_P} \in \{1,2\}$. These are precisely the cases we found for equality above.
\end{proof}

\begin{corollary}\label{Cdiscind} Let $g$ be a replacement polynomial for $f$ and assume that $\deg(g) \geq 1$. Then $(\deg(g),\nu(\Delta_g)) \leq (\deg(f),\nu(\Delta_f))$ in the lexicographic ordering. Equality can possibly hold only when for every $P \in A_{\mathrm{bad}}$, we have $b=0$ and $\wt_P = 3$. In this case, for every replacement polynomial $h$ of $g$, we have $(\deg(h),\nu(\Delta_h)) < (\deg(g),\nu(\Delta_g))$ in the lexicographic ordering. In particular, the inductive process outlined in Section~\ref{outline} terminates.
\end{corollary}
\begin{proof}
 Theorem~\ref{replacementchange}(c) shows that 
 \[0 \leq \nu(\Delta_f) -\left( \sum_{\substack{P \in A \setminus A_{\mathrm{bad}} \\ P \neq \infty, g_i \in C_P}} (n_i-1) \right) - \sum_{\substack{P \in A_{\mathrm{bad}} \setminus \{\infty\} \\ \deg(f_P^\infty) \geq 1 }} \nu(\Delta_{f_P^{\infty}}) - \sum_{\substack{P \in A_{\mathrm{bad}} \setminus \{\infty\} \\ \deg(f_P^{\neq \infty}) \geq 1 }} \nu(\Delta_{f_P^{\neq \infty}}).\]
 Note that by Remark~\ref{rootscoefficients} and Remark~\ref{Rcomprep} the degrees of the replacement polynomials are non-increasing.
 Combining the previous two sentences, we see that for any replacement polynomial $g$ with $\deg(g) \geq 1$, we have $(\deg(g),\nu(\Delta_g)) \leq (\deg(f),\nu(\Delta_f))$ in the lexicographic ordering. Furthermore, by Theorem~\ref{replacementchange}(c) the displayed inequality of discriminants above is strict unless for every $P \in A_{\mathrm{bad}}$, we have $b=0, \wt_P = 3$ and that one of $\deg(f_P^{\neq \infty})$ and $\deg(f_P^{\infty})$ is $0$. This further shows that the only case when we can possibly have $(\deg(g),\nu(\Delta_g)) = (\deg(f),\nu(\Delta_f))$ for a replacement polynomial $g$ is when $b=0, \wt_P=3$. In these cases, by the definition of $b_P$ and Lemma~\ref{fmultexc}, we have $b_P = 1$, and therefore once again by Theorem~\ref{replacementchange}(c), we see any of the replacement polynomials $h$ for $g$ satisfy $(\deg(h),\nu(\Delta_h)) < (\deg(g),\nu(\Delta_g)) = (\deg(f),\nu(\Delta_f))$ in the lexicographic ordering. 
\end{proof}

We need an alternate characterization of good weight $3$ points from Definition~\ref{D:goodtype3} before we can prove Theorem~\ref{Cexbal}. Let $\widetilde{f_P} = \prod_{f_i \in C_P} f_i$.
\begin{lemma}\label{L:altgoodtype3}
 Suppose $b=0$ and $P$ in $\divi(f)$ satisfies $\wt_P = 3$. Then $P$ is a good weight $3$ point if and only if $\widetilde{\wt_Q} \leq 2$ for every $Q$ in $\divi(g)$ for every replacement polynomial $g$ of $f_P$.
\end{lemma}
\begin{proof}
 Since $b=0$ and $\wt_P = 3$, this means $\widetilde{\wt_P} = \sum_{i \in C_P} \min(n_i,\lambda_i) = 3$. 
 
 Since the contribution to $\widetilde{\wt_Q}$ from each irreducible factor is at least $1$, and non-decreasing if we replace $f_i \in C_P$ by any of its replacement polynomials, it follows that if the irreducible polynomials in $C_P$ specialize to more than one point after one blow-up, then $\widetilde{\wt_Q} \leq \widetilde{\wt_P}-1 = 2$ for every $Q$ in $\divi(g)$ for every replacement polynomial $g$ of $f_P$. This is the first case in the definition of a good weight $3$ point.
 
 We may now further assume that all irreducible polynomials in $C_P$ specialize to the same point $Q$ on the exceptional curve $E_P$ after one blowup. Since the contribution from each irreducible polynomial $f_i$ in $C_P$ to $\widetilde{\wt_P}$ is at least $1$, it follows that $C_P$ consists of at most three irreducible polynomials. Furthermore, if $C_P$ consists of $3$ irreducible polynomials $f_1,f_2,f_3$ and $\min(n_i,\lambda_i) = 1$ for every $i$, since the replacement polynomial for each $f_i$ contributes at least $1$ to $\widetilde{\wt_Q}$, it follows that $\widetilde{\wt_Q} = 3$. This case is excluded from the definition of a good weight $3$ point. It remains to analyze the cases when $C_P$ has at most two distinct irreducible factors. 
 
 The remaining possibilities for $\widetilde{\wt_P} = 3$ and $\widetilde{\wt_Q} \leq 2$ are 
\begin{enumerate}
 \item $C_P$ consists of $2$ irreducible polynomials $f_1,f_2$ and $\min(n_1,\lambda_1) = 1$ and $\min(n_2,\lambda_2) = 2$, and the pair of integers $(\widetilde{n_2},\widetilde{\lambda_2})$ for the replacement polynomial $h_2$ of $f_2$ satisfy $\min(\widetilde{n_2},\widetilde{\lambda_2}) = 1$, and,
 \item $C_P$ consists of a single irreducible polynomial $f_1$ and $\min(n_1,\lambda_1) = 3$, and the pair of integers $(\widetilde{n_1},\widetilde{\lambda_1})$ for the replacement polynomial $h_1$ of $f_1$ satisfy $1 \leq \min(\widetilde{n_1},\widetilde{\lambda_1}) \leq 2$.
\end{enumerate}
Since Remark~\ref{rootscoefficients} and Remark~\ref{Rcomprep} show that for each $i$ we have
\[ (\widetilde{n_i},\widetilde{\lambda_i}) = \begin{cases} (n_i,\lambda_i-n_i) \quad\quad \textup{if } \lambda_i \geq n_i, \\ (\lambda_i,n_i-\lambda_i) \quad\quad \textup{if } \lambda_i < n_i,
\end{cases} \]
it follows that
\begin{enumerate}
 \item $\min(n_2,\lambda_2) = 2$ and $\min(\widetilde{n_2},\widetilde{\lambda_2}) = 1$ if and only if $(n_2,\lambda_2) \in \{ (3,2),(2,3) \}$, and,
 \item $\min(n_1,\lambda_1) = 3$ and $1 \leq \min(\widetilde{n_1},\widetilde{\lambda_1}) \leq 2$ if and only if $(n_1,\lambda_1) \in \{ (3,4),(4,3),(3,5),(5,3) \}$.
\end{enumerate}
These are precisely the remaining cases in Definition~\ref{D:goodtype3}. 
\end{proof}

\begin{proof}[Proof of Theorem~\ref{Cexbal}]
We have $-(\Art(X^f)) = \nu(\Delta_f)$ if and only if the condition for equality in Theorem~\ref{replacementchange}(b) holds for $f$ and all its replacement polynomials. In particular, for $-(\Art(X^f)) = \nu(\Delta_f)$, it is necessary that $\wt_P \leq 3$ for every $P \in A$. 

We first show that if $\widetilde{\wt_P} \leq 2$, then the condition for equality is satisfied by all the replacement polynomials coming from $f_i$ in $C_P$. Remark~\ref{rootscoefficients} and Remark~\ref{Rcomprep} show that the contribution to $\widetilde{\wt}$ is non-decreasing when we replace $f_i$ by its replacement polynomials. So once again using $\wt-\widetilde{\wt} \leq 1$, we see that $\wt \leq 3$ and the condition for equality is satisfied by all the replacement polynomials coming from $f_i$ in $C_P$. In particular, if $\wt_P \leq 2$ or if $b=1$ and $\wt_P \leq 3$, then the condition for equality is satisfied by all the replacement polynomials coming from $f_i$ in $C_P$.

It remains to analyze the case $b=0$ and $\wt_P = 3$. Since $\wt_P$ is odd, by Corollary~\ref{fmultexc} and the Definition of $b_P$, we have $b_P = 1$.  Remark~\ref{rootscoefficients} and Remark~\ref{Rcomprep} show that $\widetilde{\wt_Q} \leq \widetilde{\wt_P} = 3$ for every $Q$ in $\divi(g)$ for every replacement polynomial $g$ of $f_P \colonequals \prod_{i \in C_P} f_i$. If $\widetilde{\wt_Q} = 3$, then $\wt_Q = b_P+\widetilde{\wt_Q} = 4$ and the condition for equality in Theorem~\ref{replacementchange}(b) fails at the second stage.  If $\widetilde{\wt_Q} \leq 2$ for every $Q$ in $\divi(g)$ for every replacement polynomial $g$ of $f_P$, then by repeating the same argument as in the case $\widetilde{\wt_P} \leq 2$, we see that the condition for equality in Theorem~\ref{replacementchange}(b) is satisfied by all further replacement polynomials. Combining the previous two sentences with Lemma~\ref{L:altgoodtype3} completes the analysis in the case $b=0$ and $\wt_P = 3$.  \qedhere

\end{proof}

\begin{proof}[Proof of Corollary~\ref{Rgenus1}]
 The only way we have a point $P$ in $\divi(f)$ with $\wt_P \geq 4$ is if $b=1$ and all roots of $f$ specialize to $P$ and have valuation $\geq 1$. In this case the replacement polynomial for $f$ after one blowup defines the same elliptic curve, but has strictly smaller discriminant, so the original equation $y^2=f(x)$ is not minimal. Similarly, if $\wt_P = 3$, and there is a point $Q$ in the exceptional curve at the blowup at $P$ such that $\widetilde{\wt_Q} \geq 3$, then once again it must be the case that all roots of $f$ specialize to $Q$ and that the replacement polynomial of $f$ has strictly smaller discriminant. In other words, if $f$ is a polynomial that realizes the minimal discriminant of the curve $y^2=f(x)$, then by Lemma~\ref{L:altgoodtype3}, the conditions of Theorem~\ref{Cexbal} are satisfied, and we have $-(\Art(X^f)) = \nu(\Delta_f)$. 
 %Furthermore, one can check that the model $X^f$ has no contractible curves. In other words, the Ogg-Saito equality holds for genus $1$ curves where the residue characteristic is at least $5$.
\end{proof}

\begin{proof}[Proof of Corollary~\ref{Ctoomuchcollision}]
 Each irreducible factor contributes at least $1$ to the weight.
\end{proof}

\begin{eg}\label{Ecombeq}[Combinatorics to rule out equality]
 The genus $2$ hyperelliptic curve corresponding to the equation $y^2=(x-1)(x-2)(x-3)(x-t^2)(x-2t^2)(x-3t^2)$ over $K=\C((t))$ has $-\Art(X^f) < \nu(\Delta_f)$ since the point $P \colon x=t=0$ is not a good weight $3$ point. The replacement polynomial $f_0^{\neq \infty}(x) = t(x-t)(x-2t)(x-3t)$ and has weight $4$ at the unique point of specialization on $E_P$.
\end{eg}

\begin{eg}\label{Eotherdirection}
Let $g \geq 2$ be an even integer. Pick $g$ elements $a_1,a_2,\ldots,a_g \in R$ with pairwise distinct residues in $k \setminus \{0,1,-1\}$. Let $f(x) = x(x+1)(x-ta_1)(x-ta_2)\cdots(x-t_ag)(x-1-ta_1)(x-1-ta_2)\cdots(x-1-t_ag)$. One can check that the model $Y^f$ is a chain of $3$ projective lines, and that $X^f$ is the minimal regular (even semistable) model, and compute that $\Delta_C = 2g(g-1)$ and $-\Art(C/K)=4$. This example shows that for higher $g$, the difference between $-\Art(C/K)$ and $\nu(\Delta_C)$ can be as large as a quadratic function of $g$.
\end{eg}

\begin{remark}\label{Rotherdirection}
 Since $\sum_{i \in C_P^{\geq 1}} n_i+\sum_{i \in C_P^{< 1}} \lambda_i \leq \deg(f)$, and the degrees of the replacement polynomials are at most the degree of $f$, the inductive inequality in Theorem~\ref{replacementchange} also gives $\nu(\Delta_f) \leq \deg(f)(\deg(f)-1) (-\Art(X^f))$. Since we have not analyzed how many contractible components, the model $X^f$ has in general, it is not clear to us if this also gives $\nu(\Delta_f) = \nu(\Delta_C) \leq (g+1)(2g-1) (-\Art(C/K))$. 
\end{remark}

\subsection{Termination of induction and the conductor-discriminant inequality}\label{terminate}
\begin{proof}[Proof of \ref{Tfinalthm}]
Since regularity is preserved under unramified base extensions and since these invariants are unchanged under unramified base extensions, we may assume that $k$ is algebraically closed by extending scalars to the Henselization. Let $f \in R[x]$ be a separable polynomial such that $\Delta_f = \Delta_C$. We may assume that $R=k[[t]]$ using Proposition~\ref{tiltdisc}. Let $X^f$ be the regular model of $C$ from Definition~\ref{models}, Lemma~\ref{useful}. Since $-\Art(C/K) \leq -\Art(X^f)$ by \cite[Proposition~1]{liup}, it now suffices to prove $-\Art(X^f) \leq \nu(\Delta_f)$.

The proof is by induction on the ordered pair $(\deg(f),\nu(\Delta_f))$. The base case of the induction is when the set $A_{\mathrm{bad}}$ from Section~\ref{setup} is empty, and in this case the inequality follows from Lemma~\ref{badpoints}, Corollary~\ref{badpointscorollary} and Lemma~\ref{evenbase}. If $A_{\mathrm{bad}}$ is not empty, define replacement polynomials as in Definition~\ref{defnreppoly} for each $P \in A_{\mathrm{bad}}$. By Remark~\ref{RAbaddeg} and Corollary~\ref{Cdiscind} the induction hypothesis applies after at most two replacement steps, and it follows that that the conductor-discriminant inequality holds for all the replacement polynomials. Adding these inequalities to the inequality in Theorem~\ref{replacementchange} proves the conductor-discriminant inequality for $f$.
\end{proof}

\section*{Acknowledgements}

I would like to thank Matt Baker, Bjorn Poonen, Joe Rabinoff and Kirsten Wickelgren for several helpful conversations. I would like to thank Matt Baker, Bjorn Poonen, Joe Rabinoff, Doug Ulmer and Kirsten Wickelgren  for their continued support and mentorship. I would like to thank Borys Kadets and Nicholas Triantafillou for suggestions for improving the exposition, and Isabel Vogt for a helpful conversation. I would also like to thank the ``A Room of One's Own initiative'' for focused research time.

\end{document}